\documentclass[a4paper]{article}
\pdfoutput=1
\usepackage{fullpage}

\usepackage{cmbright}
\usepackage[utf8]{inputenc}
\usepackage[T1]{fontenc}
\usepackage[english]{babel}
\usepackage{authblk}
\usepackage{amsmath}
\usepackage{tikz}
\usetikzlibrary{3d}
\usepackage{ulem}

\usepackage{empheq}
\usepackage{enumerate}
\usepackage{graphicx}
\usepackage{abraces}
\usepackage{graphicx}

\usepackage{xcolor}
\usepackage{natbib}
\usepackage{url}
\usepackage{algorithm,algpseudocode}

\usepackage{amsmath,amsthm,amsfonts,amssymb,bm}
\usepackage{dsfont}

\newtheorem{theorem}{Theorem}[section]
\newtheorem{proposition}{Proposition}[section]
\newtheorem{lemma}[theorem]{Lemma}

\newtheorem{remark}{Remark}[section]

\def\eq{\begin{equation}}
\def\qe{\end{equation}}

\def\1{\bm{1}}

\def\ve{{\bm{e}}}
\def\vf{{\bm{f}}}
\def\vg{{\bm{g}}}
\def\vh{{\bm{h}}}

\def\vk{{\bm{k}}}

\def\vs{{\bm{s}}}
\def\vt{{\bm{t}}}
\def\vu{{\bm{u}}}
\def\vv{{\bm{v}}}

\def\vy{{\bm{y}}}

\def\mD{{\bm{D}}}

\def\mJ{{\bm{J}}}

\def\mX{{\bm{X}}}

\DeclareMathAlphabet{\mathsfit}{\encodingdefault}{\sfdefault}{m}{sl}
\SetMathAlphabet{\mathsfit}{bold}{\encodingdefault}{\sfdefault}{bx}{n}

\newcommand{\R}{\mathds{R}}

\numberwithin{equation}{section}

\usepackage{hyperref}
\definecolor{burgundy}{rgb}{0.5, 0.0, 0.13}
\definecolor{camel}{rgb}{0.76, 0.6, 0.42}
\definecolor{chamoisee}{rgb}{0.63, 0.47, 0.35}
\definecolor{grey1}{RGB}{128,128,128}
\hypersetup{colorlinks=true,linkcolor=burgundy,citecolor=chamoisee,urlcolor=burgundy,linktoc=page}

\newcommand{\cX}{\mathcal X}

\newcommand{\bbE}{\mathds E}
\newcommand{\bbH}{\mathds H}
\newcommand{\cD}{\mathcal D}
\newcommand{\cC}{\mathcal C}
\newcommand{\cM}{\mathcal M}
\newcommand{\bbR}{\mathds R}

\newcommand{\measet}{\cM(\cX)}

\newcommand{\ii}{\mathfrak{i}}
\newcommand{\mus}{\mu^{\star}}
\newcommand{\nuk}{\nu_k}
\newcommand{\nukp}{\nu_{k+1}}
\newcommand{\nuke}{\nu^{\varepsilon}_k}
\newcommand{\nukpe}{\nu^{\varepsilon}_{k+1}}
\newcommand{\pos}{\mathds{T}}
\newcommand{\weights}{\mathds{W}}
\newcommand{\Kernel}{\mathds{K}}
\newcommand{\rad}{\cX}

\DeclareMathOperator*{\esssup}{ess\,sup}
\DeclareMathOperator*{\essinf}{ess\,inf}

\title{FastPart: Over-Parametrized Stochastic Gradient Descent\\ for Sparse optimization on Measures}

\bgroup

\author[1,4]{Yohann De Castro}
\author[2,4]{, Sébastien Gadat}
\author[3]{ and Clément Marteau}

{
\affil[1]{{École Centrale de Lyon, CNRS UMR 5208, Institut Camille Jordan, Écully, France.}}
\affil[2]{Université Toulouse 1 Capitole, Toulouse School of Economics, France. } 
\affil[3]{Univ. Claude Bernard, CNRS UMR 5208, Institut Camille Jordan, Villeurbane, France.}
\affil[4]{Institut Universitaire de France (IUF)}
}
\date{version as of \today}
\egroup

\begin{document}
\maketitle

\noindent Corresponding Author: Sébastien Gadat\\
\noindent  Affiliation: Toulouse School of Economics\\
\noindent E-mail Address: sebastien.gadat@tse-fr.eu

\begin{abstract}
    This paper presents a novel algorithm that leverages Stochastic Gradient Descent strategies in conjunction with Random Features to augment the scalability of Conic Particle Gradient Descent (CPGD) specifically tailored for solving sparse optimization problems on measures. By formulating the CPGD steps within a variational framework, we provide rigorous mathematical proofs demonstrating the following key findings: $\mathrm{(i)}$ The total variation norms of the solution measures along the descent trajectory remain bounded, ensuring stability and preventing undesirable divergence; $\mathrm{(ii)}$ We establish a global convergence guarantee with a convergence rate of $\mathcal{O}(\log(K)/\sqrt{K})$ over $K$ iterations, showcasing the efficiency and effectiveness of our algorithm, $\mathrm{(iii)}$ Additionally, we analyse and establish local control over the first-order condition discrepancy, contributing to a deeper understanding of the algorithm's behaviour and reliability in practical applications.
\end{abstract}

\section{Introduction}

\subsection{Convex programming for sparse optimization on measures}

Convex optimization on the space of measures has gained attention during the past decade, {\it e.g.,} \cite{bach2021gradient,chizat2018global,chizat2022sparse,de2021supermix,poon2021geometry,de2012exact,candes2014towards} and references therein. 
It is also a popular field of investigation to derive global optimization methods (a.k.a. simulated annealing) through the embedding of~$\mathds{R}^d$ into the space of measures \citep{BolteMiclo,MicloSolo}.

\medskip

At his core, it can be viewed as a fruitful way of expressing many non-convex signal processing and machine learning tasks into a convex one, where one searches for an element of a Hilbert space $\bbH$ that can be described as a linear combination of a few, say $\bar s$, elements: 
\begin{equation}
    \label{eq:sparse solution}
    \bar\vy = \sum_{j=1}^{\bar s}\bar \omega_j\varphi_{\bar \vt_j}\,,
\end{equation}
from a given parameterized set $\big\{\varphi_\vt\,:\,\vt\in\cX\big\}$ where $\bar \omega_j\in\bbR\setminus\{0\}$ and $\cX$ is a {closed compact and convex} set of $\bbR^d$.  {From now on, and throughout the remainder of this paper, we shall assume without loss of generality that $\cX = \bar B(0,R_{\mathcal X})$, the closed centered ball in $\mathbb{R}^d$ with some radius $R_{\mathcal X}>0$. }

Given an empirical observation $\vy$, we would like to find a sparse representation, akin to~\eqref{eq:sparse solution}, that explains~$\vy$ and for which the learnt parameters $(\omega_j,\vt_j)_{j=1}^s$ encode an output solution for which generalization properties can be proven. A common practice is to minimize: 
\begin{equation}
    \label{eq:mse}
    (\omega_j,\vt_j)_{j=1}^s\mapsto \Big\|\vy - \sum_{j=1}^{s}\omega_j\varphi_{\vt_j}\Big\|^2_\bbH\,,
\end{equation}
which is a non-convex program. In the expression \eqref{eq:mse} above, $s\geq 1$ is a tuning parameter quantifying the so called {\it sparsity} of the solution. 

\medskip

A substantial body of literature pertains to the minimization of Mean Squared Error (MSE) and aligns with the framework outlined herein. In this paper, we will expound upon our methodology in a generalized format that can effectively encompass a majority of fields. Notably, certain specific instances will be elaborated upon within this paper, including but not limited to sparse de-convolution \citep{de2012exact,candes2014towards}, infinitely wide neural networks \citep{bach2021gradient,chizat2018global}, or Mixture Models \citep{de2021supermix}. Our approach is deployed according to the following steps. 

\paragraph{Lifting on the space of signed measures}
First, we lift Program \eqref{eq:mse} onto the space $(\measet,\|\cdot\|_{\mathrm{TV}})$ of Radon measures with finite total variation norm on $\cX$, defined a the dual space of the continuous functions $(\mathcal C(\mathcal X),\|\cdot\|_{\mathrm{\infty}})$. Consider the positive definite kernel~$\Kernel$ defined as the dot product~$\Kernel(\vt,\vt'):=\langle \varphi_\vt,\varphi_{\vt'}\rangle_{\bbH}$ for all $t,t'\in \cX$, and assume that it is continuous:

\medskip

\noindent
\textbf{Assumption \eqref{ass:kernel_continuity}.} \textit{The feature map function
\begin{equation}
    \label{ass:kernel_continuity}
    \tag{$\mathbf{A_{\mathrm{C}}}$}
    \vt\in\cX\mapsto\varphi_\vt\in\bbH
\end{equation} 
is continuous.
}

\medskip

\noindent
Consider the {\it kernel measure embedding} (we refer to Appendix~\ref{app:kme} for further details),
 \begin{equation}
 \label{def:Phi}
 \Phi:\ \nu\in\measet\mapsto\int_\cX\varphi_\vt\mathrm d\nu(\vt) \in\bbH\,.
 \end{equation}
It is proven in the appendix (see Lemma~\ref{lem:KME}) that $\Phi$ is a bounded linear map under \eqref{ass:kernel_continuity}. We deduce that: 
\begin{equation}
    \label{eq:lifting}
    \nu\in\measet\mapsto \big\|\vy - \Phi(\nu)\big\|^2_\bbH\,,
\end{equation}
is a convex function on $\measet$. Taking the set of discrete measures given by $\nu=\sum_{j=1}^{s}\omega_j\delta_{\vt_j}$, where~$\delta_\vt$ is the Dirac mass at point $\vt$, we uncover the parametrization~\eqref{eq:mse}. Hence, we have lifted a non-convex program on $(\omega_j,\vt_j)_{j=1}^s$ onto a convex program over a much larger space, the space of signed measures.

\paragraph{Total variation norm regularization}
The second step is regularization. One key parameter is $s$, the number of learnt parameters, that should be considered to estimate (an approximation of) the true function~\eqref{eq:sparse solution}. In practice, it can be cumbersome to tune this parameter and it might be better to resort to regularization. One benefit of the lifting on the space of measures is that this can be simply done by the $\mathrm{TV}$-norm. Inspired by $L^1$-regularization in (high-dimensional) inverse problems, we study the so-called {\it Beurling LASSO} \citep{de2012exact,candes2014towards} referred to as BLASSO below, whose convex objective function is given by:
\begin{equation}
    \label{def:J}
    J(\nu) := \frac{1}{2} \big\| \vy-\Phi(\nu)\big\|_\bbH^2 +  \lambda \|\nu\|_{\mathrm{TV}}\,,
\end{equation} 
where $\lambda>0$ is a tuning parameter. We denote by $\mu^\star\in\measet$ a solution to BLASSO 
\begin{equation}
\tag{$\mathcal B$}
\label{def:Blasso}
    J(\mu^\star) = \min_{\mu \in \measet } J(\mu)\,.
\end{equation}
{
The existence of a  solution to the problem at hand is not manifestly evident. Seminal contributions in the field, as articulated in \citep[Proposition 3.1]{Bredis_Pikkarainen_13} and \citep[Theorem~3.1]{hofmann2007convergence}, have shown the existence of solutions upon continuity prerequisites imposed on the operator~$\Phi$ or its pre-dual counterpart. Nevertheless, the arduousness associated with ascertaining the continuity and well-defined attributes of the operator $\Phi$ as expounded in Equation \eqref{def:Phi}, within a given framework, underscores the complexity involved. In contrast, Condition~\eqref{ass:kernel_continuity} affords a more tractable means of validation. We present a result displayed in Theorem \ref{theo:mu_star} below (established in Appendix~\ref{app:thm_existence}) that demonstrates the existence of solutions for any convex optimization problem formulated in the manner of Equation~\eqref{eq:general_min}, subject to the condition of continuity stipulated in Equation~\eqref{ass:kernel_continuity}. This result is aligned with a nice recent work \citep[Proposition 2.3]{bredies2024asymptotic} for general regularization terms which assume sequentially weak*-to-strong continuity of $\Phi$. As shown in the proof (see Appendix~\ref{app:thm_existence}), we establish weak*-to-weak continuity for $\Phi$ when using the $TV$-norm regularization.}

\pagebreak[3]

\begin{theorem}\label{theo:mu_star}
Let $\bbH$ be separable Hilbert space and let $\cX$ be compact metric space. Consider the problem
\begin{equation}
    \label{eq:general_min}
    \inf_{\mu\in\measet}\Big\{L(\Phi\mu)+\lambda\|\mu\|_{\mathrm{TV}}\Big\}
\end{equation}
where $L\,:\,\vh\in\bbH\to L(\vh)\in [0,\infty]$ is convex and lower semi-continuous. If \eqref{ass:kernel_continuity} holds then there exists a measure $\mu^\star\in\measet$ solution to \eqref{eq:general_min}. Furthermore, if $L$ is strictly convex, then the vector $\Phi(\mu^\star)\in\bbH$ is unique $($it does not depend on the choice of the solution $\mu^\star)$.
\end{theorem}

\begin{remark}
   By choosing  $L(\vh)=(1/2)\|\vy-\vh\|_\bbH^2$ in Theorem~\ref{theo:mu_star}, it is established that there exists a signed measure~$\mu^\star\in\measet$ solution to BLASSO~\eqref{def:Blasso}.  
\end{remark}

Over the last decade, several investigations on the performances of the estimator associated to the solution of (\ref{eq:general_min}) have been proposed in several specific situations. This solution can be proven to be close, for some partial Wasserstein $2$ distance \citep{poon2021geometry}, to the target measure $\bar \mu=\sum_{j=1}^{\bar s}\bar \omega_j\delta_{\bar t_j}$ involved in Equation \eqref{eq:sparse solution} in some cases of interest $(${\it e.g.,} Mixture Models \citep{de2021supermix} or sparse de-convolution \citep{poon2021geometry,de2012exact}$\,)$ as soon as the support points ${\bar t_j}$ of the target $\bar \mu$ are sufficiently separated.  Moreover, if the bounded linear map $\Phi$ has  finite rank $m\geq 1$, then there exists a solution to \eqref{def:Blasso} with at most $m$ atoms, as proven by \cite[Section 4]{boyer2019representer}. 

\subsection{Conic Particle Gradient Descent (CPGD)}
Solving \eqref{def:Blasso} from a practical point of view is not an immediate task, due to the infinite dimensional nature of the target. In this context, the Sliding Franck-Wolfe algorithm (see, e.g., \cite{denoyelle2019sliding}) provides an answer to this question. In this paper, we focus instead on the convergence of a Stochastic and Random Feature version of the Conic Particle Gradient Descent (CPGD)  \citep{chizat2022sparse} towards a minimum of Program~\eqref{def:Blasso}.

\medskip

Writing the weights $\weights:=(\omega_1,\ldots,\omega_p)$ and the positions $\pos:=(\vt_1,\ldots,\vt_p)\in\cX^p$, we consider a generic measure with $p$ weighted particles by:
\begin{equation}\label{def:nu}
    \nu(\weights,\pos) := \sum_{j=1}^p \varepsilon_j\omega_j \delta_{\vt_j},\,
\end{equation}
where $\omega_j>0$ (resp. $\varepsilon_j=\pm1$) refers to the weight (resp. the sign) of the particle $j$. The signs are fixed along the descent while the positions $\pos$ and weights $\weights$ are updated at each gradient step. By a symmetrization argument, see for instance \cite[Appendix A]{chizat2022sparse}, we consider, without loss of generality, that~$\varepsilon=1$. It holds that minimizing \eqref{def:Blasso} or minimizing  $J$ defined by:
\begin{equation}
    \label{def:Blasso+}
    \tag{$\mathcal B_+$}
    J(\mu^\star) = \min_{\mu \in \measet_+ } J(\mu)\,.
\end{equation}
are equivalent, in a sense made precise by \cite[Proposition A.1]{chizat2022sparse} for instance; where $\measet_+$ is the set of non-negative measures with finite $\mathrm{TV}$-norm. The attentive reader can uncover the next results on $\measet$ by replacing $\omega_j$ by $\varepsilon_j\omega_j$. The gradient descent dynamics are the same and our results also holds in this latter case. 

Our algorithm makes use of particles measures as a proxy for solving the problem \eqref{def:Blasso}. To this end, we adapt the notation of objective functions and related quantities accordingly. Denoting $\bm\lambda:=(\lambda,\ldots,\lambda)$, the definition of $\nu(\weights,\pos)$ in \eqref{def:nu} then yields
\begin{align*}
J(\nu(\weights,\pos)) 
    =  \frac{1}{2} \Big\| \vy - \sum_{j=1}^p \omega_j \varphi_{\vt_j} \Big\|_\bbH^2 + \lambda \sum_{j=1}^p \omega_j
    :=  F(\weights,\pos)+\frac{1}{2}\| \vy \|_\bbH^2\,, 
\end{align*}
where $\vk_{\pos}:=(\langle \vy,\varphi_{\vt_1}\rangle_{\bbH},\ldots,\langle \vy,\varphi_{\vt_p}\rangle_{\bbH})\in\bbR^p$, $\Kernel_\pos$ is a $(p\times p)$ matrix with entries~$\Kernel(\vt_i,\vt_j)$ defined by:
\begin{equation}
    \label{def:kernel}
    \Kernel(\vt_i,\vt_j) = \langle \varphi_{\vt_i},\varphi_{\vt_j} \rangle_{\bbH},
\end{equation}
and 
\begin{equation}
\label{def:F}
    F(\weights,\pos)
    :=\langle \bm\lambda-\vk_{\pos}, \weights \rangle
    + \frac{1}{2}\weights^T \Kernel_\pos \weights\,,
\end{equation}
is equal to $J(\nu(\weights,\pos)) $ up to the additive constant term $(1/2)\| \vy \|_\bbH^2$ (that only depends on the observations and not on the parameters of the measure we are optimizing on). 

\medskip

In contrast to the original problem presented in \eqref{def:Blasso}, the optimization process now operates within a different domain. Instead of working within the space of measures, it focuses on particle measures with a set of fixed size $p$. This shift in perspective involves optimizing over both the positions~$\pos$ and weights~$\weights$. Although this adjustment serves to simplify the model's complexity to some extent, it introduces certain computational challenges. Firstly, for each pair of parameters $(\weights,\pos)$, the computation of $F(\weights,\pos)$ necessitates the evaluation of $\vk_\pos$ and $\Kernel_\pos$. Depending on the structure of the Hilbert space $\mathds{H}$ and the associated scalar products, this computation can be time-consuming. The need to calculate these quantities at each iteration of a gradient descent algorithm can become problematic, especially when dealing with high-dimensional spaces ($d$ being significant). Furthermore, considering a large number of particles during the optimization process, which is often essential, can substantially escalate the computational burden. This computational overhead needs to be carefully managed and optimized to ensure the efficiency of the optimization procedure. In this paper, we address these issues by presenting a novel algorithm that incorporates Stochastic Gradient Descent (SGD) iterations. To the best of our knowledge, this approach has not been previously explored in the context of sparse optimization within the domain of measures. We rigorously examine the properties of this algorithm, conducting a comprehensive investigation from both theoretical and practical perspectives.  \\


{The paper is organized as follows. Sections \ref{s:algo_sto} and \ref{s:mainresults} present our stochastic algorithm and our main convergence results. Section \ref{s:algo} then describes the rationale behind the construction of the algorithm. Section \ref{s:examples} is illustrated by the example of mixture models, an important topic in unsupervised learning. A numerical illustration on some toy examples are discussed in Section \ref{sec:experiments}. Proofs and related technical results are gathered in Section \ref{sec:proofs} and in Appendices~\ref{app:technical}, \ref{sec:gradients} and \ref{app:tech_tools}. }

\subsection{Stochastic \& Random Feature CPGD (\texttt{FastPart})}
\label{s:algo_sto}
Iterative algorithms solving~\eqref{def:Blasso+} have already been at the core of theoretical investigations. We refer for instance to \cite{chizat2022sparse} among others. The latter investigates an algorithm that requires some frequent calls to the gradient $\nabla F$ of the objective $F$ defined in~\eqref{def:F}. This gradient can be related to the Fréchet differential function $\vt\mapsto J'_\nu(\vt)$ of $J(.)$ at point $\nu\in \mathcal{M}(\cX)_+$ and its gradient~$\nabla_\vt J'_\nu$. The Fréchet differential~$J'_\nu$ is defined through the following first order Taylor expansion:
\begin{equation}
    \label{def:Frechet_diff}
    \forall \nu \in \measet_+, \quad \nu+\sigma\in\measet_+ \qquad 
J(\nu+\sigma)-J(\nu) = \langle J'_\nu,\sigma \rangle_{\measet^\star,\measet}
    +q(\sigma)\,, 
\end{equation}
where $J'_{\nu}$ is the Fr\'echet gradient  and $q$ is a second order term.
 We refer to Proposition \ref{prop:jnuprime} for details.
According to Proposition~\ref{prop:grad_F}, for any $\vt\in \cX$, we have by Equation \eqref{eq:jnu1} that, given $\nu = \sum_{j=1}^p \omega_j \delta_{\bm t_j}$,
\begin{equation}
\label{eq:def_J_prime}
    J'_\nu(\vt) 
    = \sum_{j=1}^p \omega_j \langle \varphi_{\bm t}, \varphi_{\vt_j} \rangle_\mathds{H} -   \langle \varphi_\vt,  \bm y \rangle_{\mathds{H}} + \lambda\,.
\end{equation}
The computation of these functions is  time-consuming for the three reasons listed below. For each of these challenges, we outline our strategy to address them, and we introduce three notation—namely, the three random variables $U$, $V$, and $T$—which will be elaborated upon in the subsequent section. In Section~\ref{sec:deconvolution_mixture}, we will provide a concrete example illustrating this phenomenon.

\medskip

\paragraph{Kernel evaluation: random variable $U$ for Random feature strategies}
Firstly, it is important to note that these functions may necessitate integral approximations due to their lack of closed-form expressions. For instance, the computation of $J'_\nu(\vt)$ within each iteration of a gradient descent algorithm involves multiple evaluations of the kernel $\Kernel(\vt,\vt')$, defined in Equation \eqref{def:kernel}. In many cases, these evaluations rely on non-explicit integrals, posing computational challenges. To circumvent this issue, we will present two strategies: random Fourier feature and convolution by a non-negative function. These strategies aim to approximate the kernel $\Kernel$ by using a low-rank random kernel, achieved by evaluating the integral defining the kernel through independent Monte Carlo sampling. We need the following assumption, which is standard in the theory of Reproducing Kernel Hilbert Spaces (RKHS).

\medskip

\noindent
\textbf{Assumption \eqref{ass:kernel_invariant}.} \textit{We assume that the kernel $\Kernel$ is \textit{translation invariant}:
\begin{equation}
    \label{ass:kernel_invariant}
    \tag{$\mathbf{A}_{\mathrm{T.I.}}$}
    \forall \vt,\vt'\in\cX\,,\quad \Kernel(\vt,\vt') = k(\vt-\vt')\,,
    \end{equation}
for some function $k(\cdot)$.
}

\medskip

{\bf $\bullet$ Random Fourier feature strategy}
\noindent
By Bochner's theorem to decompose it into its spectral form and use a Monte Carlo approximation as follows:
\begin{equation}
\notag
    \Kernel(\vt,\vt') = k(\vt-\vt')
    =k(0)\int  e^{-\ii \langle \vu,\vt-\vt'\rangle}\mathrm d\sigma(\vu)
    \simeq \frac{k(0)}{m}\sum_{j=1}^m 
        \mathcal{R}\mathrm{e}\big(
            e^{-\ii \langle U_j,\vt\rangle} \overline{\displaystyle{e^{-\ii \langle U_j,\vt'\rangle}}}
        \big)
    =\frac{1}{m}\sum_{j=1}^m\vg_{\vt,\vt'}(U_j)\,,
\end{equation}
where $\mathcal{R}\mathrm{e}$ denotes the real part of a complex, $\sigma$ is the spectral probability measure associated to the kernel, $U_i$ are i.i.d. with law $\sigma$ and 
\[
\vg_{\vt,\vt'}(\vu):=\mathcal{R}\mathrm{e}\big(e^{-\ii \langle \vu,\vt\rangle} \overline{\displaystyle{e^{-\ii \langle \vu,\vt'\rangle}}}\big)\,,
\]
the bar denoting complex conjugation. Hence, the $(p\times p)$ matrix $\Kernel_\pos$ can be approximated by the following explicit rank $m$ approximation 
\begin{equation}
\notag
    \Kernel_\pos\simeq \mathbb U_\pos \mathbb U_\pos^\star
    \quad 
    \text{where}
    \quad 
    \mathbb U_\pos:=\sqrt{\frac{k(0)}{m}}\big(\displaystyle{e^{-\ii \langle U_j,\vt_i\rangle}}\big)_{\substack{1\le i\le p\\ 1\le j\le m}}\,.
\end{equation}

\medskip

{\bf $\bullet$ Convolution by non-negative function strategy}
A second way is to consider that $k(\cdot)$ can be written as a convolution $k=\tilde k\star \sigma$ for some non-negative integrable function $\sigma$, say a probability density function without loss of generality. Invoke a Monte Carlo approximation as follows:
\[
\Kernel(\vt,\vt') = (\tilde k\star \sigma)(\vt-\vt')=\int_{\R^d} \tilde k(\vt-\vt'-\vu)\sigma(\vu)\mathrm{d}\vu \simeq\frac{1}{m}\sum_{j=1}^m\vg_{\vt,\vt'}(U_j)\,,
\]
where $U_i$ are i.i.d. with law $\sigma$ and 
\begin{equation}
\label{eq:approx_kernel_convolution}
    \vg_{\vt,\vt'}(\vu):=\tilde k(\vt-\vt'-\vu)\,,    
\end{equation}
to get an approximation of the kernel. This strategy can be used in Mixture Models (MM) as shown in~\eqref{eq:kernel_MM_convolution}. As a side note, it is not hard to see that if $\tilde k(\cdot)$ is a positive type function ({\it i.e.,} it defines a RKHS kernel) so is $k(\cdot)$. 
\medskip

\paragraph{Large samples: random variable $V$ for picking a data sample at random}
A second strategy is to employ stochastic gradient computation with batch sub-sampling, which entails selecting a single data point at random--a fundamental component of various stochastic algorithms. Within our framework, this can be realized by utilizing the observation vector $\vy$. It is worth emphasizing that in specific scenarios, the observed signal $\bm y$ can itself be a random variable. For instance, consider the case where 
\begin{equation}
\label{eq:y_decompo_sum}
 \vy=\frac{1}{N}\sum_{i=1}^N \vy_i
 \quad\text{and}\quad
 \langle\varphi_\vt, {\vy}\rangle_{\mathds{H}}=\vh_\vt(V)\,,   
\end{equation}
consisting of $N$ i.i.d.~samples $\vy_i$, $V$ uniformly distributed over $\{1,\ldots,N\}$ and 
\[
\vh_\vt(\vv):=\langle\varphi_\vt,\vy_\vv\rangle_{\mathds{H}}\,.
\]
In such cases, it becomes feasible to generate an unbiased stochastic version of $\vy$ by randomly selecting a single or a mini-batch data sample. Moreover, even if $\vy$ is deterministic or does not match the aforementioned identity \eqref{eq:y_decompo_sum}, it is always possible to employ a similar strategy as described above (Random Fourier feature or convolution by a probability density function) to approximate $\langle\varphi_\vt, {\vy}\rangle_{\mathds{H}}$ itself. Certainly, note that the unbiased stochastic version of $\vy$ and random kernel approximation of $\langle\varphi_\vt, {\vy}\rangle_{\mathds{H}}$ should be jointly considered whenever possible.

\medskip

\paragraph{Many particles: random variable $T$ for picking a particle at random} The necessity for a large number of particles to attain convergence guarantees increases exponentially with the dimension $d$ of $\mathcal X$. {All these particles are simultaneously involved in the evaluation of gradient terms. As a final ingredient, we opt to use only a particle or a mini-batch of particles, selected randomly, to evaluate these terms at each step of the algorithm. The time complexity of gradient computations can be diminished by a factor $p$ (number of particles) as one can see using \eqref{eq:def_J_prime} where the sum over $p$ particles is replaced by one evaluation at a random particle's location:
\[
J'_\nu(\vt) 
= \sum_{j=1}^p \omega_j \langle \varphi_{\bm t}, \varphi_{\vt_j} \rangle_\mathds{H} -   \langle \varphi_\vt,  \bm y \rangle_{\mathds{H}} + \lambda
= \|\nu\|_{\mathrm{TV}} \mathbb{E}\big[\langle \varphi_{\bm t}, \varphi_{T} \rangle_\mathds{H}\big] -   \langle \varphi_\vt,  \bm y \rangle_{\mathds{H}} + \lambda\,,
\]
with $T\sim \nu/\|\nu\|_{\mathrm{TV}}$. Instead of choosing one particle at random, one can also pick a mini-batch of particles at random, say $m_T$ of them using, for instance the law $T=(T_1,\ldots,T_{m_T})\sim\big(\nu/\|\nu\|_{\mathrm{TV}}\big)^{\otimes m_T}$ with the gradient rule based on 
\[
J'_\nu(\vt) 
= \sum_{j=1}^p \omega_j \langle \varphi_{\bm t}, \varphi_{\vt_j} \rangle_\mathds{H} -   \langle \varphi_\vt,  \bm y \rangle_{\mathds{H}} + \lambda
= \frac{\|\nu\|_{\mathrm{TV}}}{m_T} \mathbb{E}\sum_{i=1}^{m_T}\big[\langle \varphi_{\bm t}, \varphi_{T_i} \rangle_\mathds{H}\big] -   \langle \varphi_\vt,  \bm y \rangle_{\mathds{H}} + \lambda\,.
\]
}

\subsubsection{Stochastic \& Random Feature approximations of the gradient}
\label{s:algosto}
We provide in this paragraph a general framework allowing the construction of the stochastic conic gradient particle algorithm we are studying.  Specific examples are discussed in Section \ref{s:examples} below. 

\medskip

The formulation of a stochastic approximation for the gradient $\nabla F$ necessitates certain general assumptions, which are outlined below. Of particular significance is the introduction of a random variable $Z=(T,U,V)$, which serves as a way to alleviate the computational burden.

\medskip

\noindent
\textbf{Assumption \eqref{A1}.} \textit{There exists a pair of random variables $(U,V)$ $($not necessarily independent$)$ such that, for any $\vt,\vt' \in \cX$,}
\begin{equation}
\label{A1}
\tag{$\mathbf{A_1}$}
  \langle \varphi_{\vt},\varphi_{\vt'}\rangle_{\bbH} = \bbE_U \vg_{\vt,\vt'}(U) \quad  \mathrm{and} \quad  \langle \varphi_\vt,  \vy\rangle_{\bbH} = \mathbb{E}_V \vh_\vt(V)\,,    
\end{equation}
\textit{for some explicit \textbf{bounded} functions $\vg$ and $\vh$}.

\medskip

\noindent
The latter assumption allows for a stochastic approximation of the functional $J'_\nu$. 
It exactly corresponds to what happens in Equation \eqref{eq:grad_sto1} below for the mixture model.
Indeed, if we {consider from now on the} random variable $T$ with distribution $\nu/\nu(\cX)$, sampled independently from $(U,V)$, and introduce:
\begin{equation}
	\mJ'_\nu(\vt,Z) := \|\nu\|_{\mathrm{TV}} \, \vg_{\vt,T}(U)  - \vh_\vt(V) + \lambda  \quad \mathrm{where} \quad Z:=(T,U,V)\,.
 \label{eq:J'sto}
\end{equation}
According to \eqref{A1}, we can write that 
\begin{equation}\label{def:grad_sto_J}
	\mJ'_\nu(\vt,Z) := J'_\nu(\vt) + \xi_\nu(\vt,Z) \quad \mathrm{with} \quad \mathbb{E}_Z  \xi_\nu(\vt,Z) = 0 \quad \forall \vt\in \cX\,.
\end{equation}
We can construct a similar stochastic approximation for the gradient (w.r.t. the position parameter) of the functional $J'_\nu$. First remark that 
\[
	\nabla_\vt J'_\nu(\vt) 
	= \sum_{j=1}^p \omega_j \nabla_\vt \langle \varphi_\vt, \varphi_{\vt_j} \rangle_\mathds{H} -  \nabla_\vt \langle \varphi_\vt,  \vy \rangle_{\mathds{H}}\,.
\]

\medskip

The following assumption allows the commutativity between derivation and expectation and is satisfied in many situations, including batch or mini-batch strategies and smooth integral computations. The associated term in the mixture model is \eqref{eq:grad_J_sto} and its stochastic counterpart is \eqref{eq:grad_sto2}.  

\medskip

\noindent
\textbf{Assumption \eqref{A2}.} \textit{The couple $(U,V)$ and the functions $\vg,\vh$ introduced in Assumption \eqref{A1} satisfy}
\begin{equation}
    \label{A2}\tag{$\mathbf{A_2}$}
    	\nabla_\vt \bbE_U \vg_{\vt,\vt'}(U) = \bbE_U \nabla_t \vg_{\vt,\vt'}(U) 
	\quad  \mathrm{and} \quad  
	\nabla_\vt \mathbb{E}_V \vh_\vt(V) = \mathbb{E}_V\nabla_\vt  \vh_\vt(V)\,,
\end{equation}
\textit{for any $\vt,\vt' \in \cX$} and $\bm g$ and $\bm h$ have \textbf{bounded} derivatives.

\medskip

\noindent
Then, introduce
\begin{equation}
    \label{eq:def_D}
    \mD_\nu(\vt,Z) := \|\nu\|_{\mathrm{TV}} \,\nabla_\vt \vg_{\vt,T}(U)  - \nabla_\vt \vh_\vt(V)\,,
\end{equation}
where $Z=(T,U,V)$ still denotes the same variable as above. We can then observe that 
\begin{equation}\label{def:grad_sto_gradJ}
	\mD_\nu(\vt,Z) := \nabla_\vt J'_\nu(\vt) + \zeta_\nu(\vt,Z) \quad \mathrm{with} \quad \mathbb{E}_Z  \zeta_\nu(\vt,Z) = 0 \quad \forall \vt\in \cX\,,
\end{equation}
where we have used Assumptions \eqref{A1} and \eqref{A2}.

\begin{remark}
    Lemma \ref{lem:boundJprime} in Appendix provides a explicit and uniform upper bound on $J'_\nu$ for any measure $\nu \in \mathcal{M}(\cX)_+$. From this remark, the boundedness of the functions $\vg$ and $\vf$ is indeed a reasonable assumption within this context. Furthermore, in addition to the conclusions presented in Equations \eqref{def:grad_sto_J} and \eqref{def:grad_sto_gradJ}, that establish unbiased estimations of $J'_\nu$ and $\nabla_\vt J'_\nu$, it is imperative to highlight that our assumptions \eqref{A1} and \eqref{A2} also result in almost sure upper bounds on $|J_\nu'(.,Z)|$ and $\|\mathbf{D}_{\nu}(.,Z)\|$ (for more comprehensive details, we refer to Lemma \ref{lem:boundJprime} and Proposition \ref{prop:bornes_techniques_J} ).
\end{remark}

\begin{remark}
    Assumption~\eqref{A2} implicitly assumes that the left hand side gradient exists. This is the case when the kernel is continuously differentiable, see \cite[Definition 4.35]{steinwart2008support} for a proper definition. This later property is equivalent to the continuous differentiability of the feature map, see \cite[Lemma 4.34]{steinwart2008support}. 
\end{remark}

\subsubsection{Mini-batch and variance reduction}
Introducing randomness in the CPGD algorithm creates some issues in terms of convergence. It is indeed necessary to control the fluctuation of the random terms $\mJ'_\nu(.,Z)$ and $\mD_\nu(.,Z)$ w.r.t. their deterministic counterparts at each iteration. Assumptions \eqref{A1} and \eqref{A2} already provides some tools in this direction: it allow to obtain almost sure bounds on these quantities. Nevertheless, our analysis will require slightly stronger results and in particular some control on the variance of theses stochastic approximations. To this end, we introduce in this section a mini-batch step.

\medskip

The principle is to draw a $m$-sample of i.i.d. random variables $Z_1,\dots, Z_m$ having the same law than~$Z$ introduced in \eqref{eq:J'sto}. The term $m\in \mathbb{N}^*$ denotes the mini-batch sample size and might depends on the iteration step. Then, we introduce, for any $\nu \in \mathcal{M}_+(\rad)$ and $\vt \in \rad$
\begin{equation}
\widehat{\mJ'_{\nu}}(\vt) := \frac{1}{m} \sum_{l=1}^m \mJ'_\nu(\vt,Z_l) \quad \mathrm{and} \quad \widehat{\mD_{\nu}}(\vt) :=  \frac{1}{m} \sum_{l=1}^m \mD_\nu(\vt,Z_l).
\label{eq:minibatch}
\end{equation}
The terms $\widehat{\mJ'_{\nu}}(.)$ and $\widehat{\mD_{\nu}}(.)$ simply correspond to averaged version of the stochastic approximations introduced in the previous section. Their second order moments are controlled by the mini-batch size $m$. To alleviate notation, we do not highlight the dependency of the average terms introduced in \eqref{eq:minibatch} with respect to the value of $m$. This dependency will be clear following the context. 

\medskip

\subsubsection{Stochastic gradient updates}

Both weights and position updates are based on the CPGD principle which will be discussed in details in Section \ref{s:algo} below. The weights optimization is performed employing an exponential weights update on a per-particle basis:
\begin{equation}\label{eq:up_w}
\forall j \in \{1,\ldots,p\}\,, \qquad \omega^{k+1}_j = \omega^{k}_j e^{- \alpha \widehat{\mJ'_{\nu_k}}(\vt_j^k)}\,.
\end{equation}
In a similar way, the positions updates satisfies, for any $j\in \lbrace 1,\dots, p \rbrace$
\begin{align}
\vt_j^{k+1}
& =  \arg\min_{\vt \in \cX}\left\{\langle \vt, \widehat{\mD_{\nu_k}}(\vt_j^k))_j\rangle + \frac{1}{2\eta} \|\vt-\vt^j_k\|^2\right\} \label{eq:optim_t_projgrad} \\
 & = \arg\min_{\vt \in \cX}\left\{\left\|\vt-(\vt_j^k- \eta \widehat{\mD_{\nu_k}}(\vt_j^k))_j) \right\|^2\right\}  \nonumber
\end{align}
which leads to projected gradient descent, after simplifying the similar weighting terms appearing on the one hand in the gradient, and on the other hand the conic metric. This update is then given as follows:
\begin{equation}
\forall j\in \lbrace 1,\dots, p \rbrace, \qquad \vt_{j}^{k+1} =\pi_\rad\left(\vt_{j}^{k} -  \eta  \widehat{\mD_{\nu_k}}(\vt_j^k) \right),
\label{eq:update_pos_sto}
\end{equation}
{where $\pi_\rad$ denotes the projection operator over $\rad$. Recall that $\rad = \bar{B}(0,R_{\mathcal X})$ for some radius~$R_{\mathcal X}>0$ and, in this setting, the projection $\pi_\rad$ is uniquely defined and may be easily computed as:
$$\forall u \in \R^d  \qquad \pi_{\rad}(u) = \frac{R_{\mathcal X}}{\|u\|} u \mathbf{1}_{\|u\| > R_{\mathcal X}} + u \mathbf{1}_{\|u\| \leq R_{\mathcal X}}.$$
We refer again to Section \ref{s:algo} for a complete presentation of the rationale behind this strategy. \\
}

\noindent
Algorithm \ref{algo:SCPGD_pro} below summarizes our stochastic optimization strategy that uses unbiased stochastic realizations, as outlined in the previous assumptions. At each iteration $k$, the elements $Z_l^{k+1}$ of the mini-batch sample are made of i.i.d. random vectors $(T_{\ell}^{k+1},U_{\ell}^{k+1},V_{\ell}^{k+1})$ where the $(U_{\ell}^{k+1},V_{\ell}^{k+1})$ have the same distribution as $(U,V)$ introduced in Assumption \eqref{A1}, and $T_l^{k+1} \sim \nu_k / \nu_k(\mathcal{X})$.

\begin{center}
\begin{algorithm}
{\caption{{Stochastic \& Random Feature Conic Particle Gradient Descent ({\tt FastPart}}) \label{algo:SCPGD_pro}}}
\begin{algorithmic}[1]
\Require{Learning rates $(\alpha,\eta)$, Initialization $(\weights^0,\pos^0)$, Projection radius $R_\rad$, Mini-batch size $(m_k)_{k \ge 1}$}; 
\State{Weights: $\weights^k$ and Positions: $\pos^k$};
    \For{$k =1,\ldots, K$ }\Comment{$K$ gradient steps}
        \State{Set $\nu_k\longleftarrow\nu(\weights^{k},\pos^{k})$;}\Comment{Particles}
        \State{Sample 
        \begin{equation*}
            (Z^{k+1}_{\ell})_{1\leq \ell \leq m_k} \longleftarrow ((T_{\ell}^{k+1},U_{\ell}^{k+1},V_{\ell}^{k+1}))_{1\leq \ell \leq m_k}
        \quad
        \text{where}
        \quad
        T_l^{k+1}\sim \nu_k/\nu_k(\mathcal{X})\,;
        \end{equation*}
        \Comment{Stochastic mini-batch variables;}        
        }
        \State{Use Equations \eqref{eq:J'sto} and \eqref{eq:def_D} and compute}  
        \[
        \widehat{\mJ'_{\nu_k}}(\vt_j^k) := \frac{1}{m_k} \sum_{\ell=1}^{m_k} \mJ'_{\nu_k}(\vt_j^k,Z^{k+1}_\ell)\quad  \text{and} \quad 
        \widehat{\mD_{\nu_k}}(\vt_j^k) := \frac{1}{m_k} \sum_{\ell=1}^{m_k} \mD_{\nu_k}(\vt_j^k,Z^{k+1}_{\ell})
        \,;
        \]
        \State{Update the weights and the positions with Equations \eqref{eq:up_w} and \eqref{eq:update_pos_sto}:
\begin{equation}\label{eq:up_w_pos}
\forall j \in \{1,\ldots,p\}\,, \qquad \omega^{k+1}_j = \omega^{k}_j e^{-\alpha \widehat{\mJ'_{\nu_k}}(\vt_j^k)}\quad \text{and} \quad
\vt_{j}^{k+1} =\pi_\rad\left(\vt_{j}^{k} - \eta  \widehat{\mD_{\nu_k}}(\vt_j^k)\right),
\end{equation}
\Comment{Projected Mirror descent}
        }
    \EndFor
\end{algorithmic}
\end{algorithm}
\end{center}

\begin{remark}[Importance of the projection step]
The projection operator $\pi_\rad$ appearing in the position updates \eqref{eq:update_pos_sto} constraints the particles to stay on the set $\rad$ along the iterations. In the deterministic version of the algorithm (namely by using directly $\mJ'_{\nu_k}(\vt_j^k)$ and $\mD_{\nu_k}(\vt_j^k)$ in Step $5$ of Algorithm \ref{algo:SCPGD_pro}), this projection step appears to be necessary only for the global convergence results (analog of Theorem \ref{theo:global_sto}) while control of the TV-norm (Proposition \ref{prop:TV_nuk}) along with local convergence (Theorem \ref{theo:local_sto}) still hold.  Some issues arise when introducing randomness during the gradient descent. In particular, the introduction of the variable $T$ in Section \ref{s:algosto} does not allow to manage a control of $\| \nu_k \|_{\mathrm{TV}}$ along the iteration. In particular, the terms $\langle \varphi_\vt, \varphi_\vs\rangle$ is no more lower bounded when $\vt$ and $\vs$ do not necessarily belong to $\mathcal{X}$. We stress that these issues may correspond to an artefact effect related to our proof techniques. We can indeed check that for a particle whose location $\vt_j^k$ is far away from $\mathcal{X}$ (in a sense to made precise) at step $k$, then $\omega_j^{k+1} \leq \omega_j^k e^{-\lambda/2}$. From an heuristic point of view, such a particle should not disturb the convergence process after few steps.  
\end{remark}

\subsection{Main results}
\label{s:mainresults}
We denote by $\nu_k$ the measure of particles produced at step $k$ by Algorithm \ref{algo:SCPGD_pro}. Our main contributions are threefold: 
\begin{itemize}
\item A control on the total-variation norm of the measures $\nu_k$ along the different iterations; 
\item A global minimization result;
\item A local investigation on the evolution of $J'_{\nu_k}$ and its gradient. 
\end{itemize}
We emphasize that these results are stated in a finite horizon setting and are therefore non-asymptotic in terms of $K$.

\medskip

Here and below, we will require additional notation. We set
\begin{align*}
\| \underline{\varphi}\|_\mathds{H} 
&:= \inf_{s,t\in \cX} \langle \varphi_t,\varphi_s \rangle_\mathds{H}
&
\| \mathbf{g} \|_{\mathds{\mathrm{Inf}}} 
&:= \essinf_{s,t\in\rad} \bm g_{t,s}\\
\|\vg\|_\infty 
&:= \esssup_{s,t\in\rad} |\vg_{s,t}|
&
\|\vh\|_\infty 
&:= \esssup_{t\in\rad,v} |\vh_{t}|
\end{align*}
where the functions $\bm g$ and $\bm h$ have been introduced in \eqref{A1} and the essential infimum and supremum are taken with respect to the probability space defining the random variable $Z=(T,U,V)$. Furthermore, Assumption~\eqref{ass:kernel_invariant} implies the following bound 
\[
\forall t \in \cX\,, \quad 
\|\varphi_{\vt}\|^2_{\bbH}  = \langle \varphi_\vt ,\varphi_\vt\rangle_{\bbH} = k(0).
\]
and we will use below the notation $\|\varphi\|_{\infty,\mathbb{H}} :=\sqrt{k(0)}$ to refer to the norm of any element $\varphi_\vt$ in $\bbH$. We also define $\|\varphi\|_{\mathrm{Lip}}$ as the Lipschitz constant associated to $\varphi$, namely
\[
\| \varphi\|_{\mathrm{Lip}} = \sup_{s\neq t\in \mathcal{X}} \frac{\| \varphi_t -\varphi_s\|_\mathbb{H}}{\| s-t\|}.
\]

\subsubsection{Boundedness of the sequence}
The next result establishes  a preliminary upper bound of the total variation norm of $(\nuk)_{k \ge 1}$ that holds \textit{uniformly} over the iterations. It will be the key to obtain the convergence towards minimizers. 

\begin{proposition} 
\label{prop:TV_nuk} 
{Assume \eqref{A1} and let $\mathbf{g}$ and $\mathbf{h}$ be the functions appearing in this assumption. Assume that  $\alpha \leq 1$, that $\|\bm g\|_{\mathrm{Inf}} > 0$ and define $r_0$ as:
\begin{equation}
r_0:= \frac{\max(0,\| \mathbf{h}\|_\infty -\lambda)}{\| \mathbf{g}\|_{\mathrm{Inf}}} e^{  \max(0,\| \mathbf{h}\|_\infty -\lambda)}
\label{def:r0}
\end{equation}
Then, for any $k\in \mathbb{N}$, we have 
\begin{equation}
\|\nu_k\|_{\mathrm{TV}} \leq R_0 := \|\nu_0\|_{\mathrm{TV}} \vee (p\, r_0) .
\label{def:R0}
\end{equation}}
\end{proposition}

\medskip

\noindent
Provided the mass of the measure $\nu_0$ at the initialization step is not too large, $\|\nu_k\|_{\mathrm{TV}}$ remains bounded along the iterations of the algorithm. We stress that the control is deterministic although we consider a stochastic algorithm. The main ingredient of the proof is to take advantage of the relationship between $\omega_j^k$ and $J_{\nu_k}'$ (together with its stochastic counterpart). The update displayed in Algorithm \ref{algo:SCPGD_pro} then allows to conclude. The complete proof is postponed to Section \ref{s:proofTV_nuk}.

\begin{remark}
    Note that $|\langle\vy,\varphi_\vt\rangle|\leq \|\vh\|_{\infty}$ and that the KKT condition reads 
    \[
    \forall\vt\in\cX\,,\quad
    J'_{\mu^\star}(\vt)=\sum_{j=1}^{p^\star} \omega_j^\star \langle \varphi_{\vt}, \varphi_{\vt_j^\star} \rangle_\mathds{H} -   \langle \varphi_\vt,  \bm y \rangle_{\mathds{H}} + \lambda\geq 0\]
    We deduce that if $\lambda\geq \|\vh\|_{\infty}$ then the KKT conditions are satisfied by the null solution ($\mu^\star=0$), which is the unique solution to \eqref{def:Blasso+}. Hence, the condition $r_0=0$ implies that the null measure is solution. Proposition~\ref{prop:TV_nuk} is informative when $r_0>0$ which occurs as soon as there exists a non-zero solution. 
\end{remark}

\begin{remark}
The assumption $\|\bm g\|_{\mathrm{Inf}} > 0$ appears to be quite reasonable as soon as $\| \underline{\varphi}\|_\mathds{H}$>0. For any $\vs,\vt\in \cX$, the variable $\bm g_{\vs,\vt}(U)$ is indeed a stochastic approximation of $\langle \varphi_s,\varphi_t\rangle_\mathds{H}$. We get that $\|\bm g\|_{\mathrm{Inf}} > 0$ with common assumptions on the construction of this approximation. 

In the case of an approximation by convolution, displayed in~\eqref{eq:approx_kernel_convolution}, note that if $\inf_{\vt,\vt',\vu}\tilde k(\vt-\vt'-\vu)>0$ then the assumption $\|\bm g\|_{\mathrm{Inf}} > 0$ is satisfied. For instance, if $\tilde k(\cdot)$ is a Gaussian density and $\vu\mapsto\sigma(\vu)$ has compact support then  $\inf_{\vt,\vt',\vu}\tilde k(\vt-\vt'-\vu)>0$ as $\vt,\vt'\in\cX$ are also bounded.
\end{remark}

\subsubsection{Global minimization with projected swarm stochastic optimization}

Consider $\mus$ a measure that \textit{globally} minimizes $J$, obtained from Theorem \ref{theo:mu_star}. The aim of this section is to show that our algorithm can produce a solution close to $\mus$ (in a sense made precise in Theorem \ref{theo:global_sto} below), under some specific conditions. Given a fixed number $K$ of iterations of Algorithm \ref{algo:SCPGD_pro}, we define the Ces\`aro average of our sequence $(\nu_k)_{k \ge 0}$ by:
\begin{equation}
\label{def:nu_bar} \bar{\nu}_K = \frac{1}{K+1} \sum_{k=0}^K \nuk\,.
\end{equation}
We then obtain the following global minimization result whose proof is displayed in Section \ref{s:proof_global_sto}.

\begin{theorem}\label{theo:global_sto}
{Assume that $\mus$ has a support contained in $\rad$. Consider an integer $K$ and the sequence of $(\nuk)_{1 \leq k \leq K}$ defined in Algorithm \ref{algo:SCPGD_pro}. We set the learning rates as:
\[
    \alpha
    =\sqrt{\frac{d \|\mus\|_{\mathrm{TV}}}{R_0^3 K}} \quad \text{and} \quad \eta = \sqrt{\frac{d R_0}{K^{3} \|\mus\|_{\mathrm{TV}}}}\,,
\]
where $R_0$ is introduced in \eqref{def:R0}. 
Assume that the   measure $\nu_0$ is uniformly distributed over  a uniform grid of step-size $\delta =2\sqrt{\frac{d}{\|\mus\|_{\mathrm{TV}} KM}}$, then:
\[
\mathbb{E} \left[ J(\bar{\nu}_K)-J(\mus) \right]   \leq \mathfrak{C} \sqrt{\frac{d \|\mus\|_{\mathrm{TV} }R_0^3}{K}} \left[ \log (d \|\mus\|_{\mathrm{TV}} R_0^3 K) + \frac{\log(|\cX|)}{d}\right]
\,,
\]
for some positive constant $\mathfrak{C}$ depending only on $\|\varphi\|_{\infty,\mathbb{H}}, \|\bm y\|_\mathbb{H}, \|\mathbf{g}'\|_\mathbb{H}, \|\bm h'\|_\mathbb{H}$.}
\end{theorem}

\medskip

\noindent
A careful inspection of the previous upper bound shows that to obtain an $\epsilon$ approximation with our Ces\`aro averaged measure $\bar{\nu}_K$, (while removing the effect of the log term) we need to choose $K$ as: \[K_{\varepsilon}=d R_0^3  \|\mus\|_{\mathrm{TV}} \varepsilon^{-2}.\]
Then, the  grid step-size is then of the order $\delta_{\varepsilon}$ given by:
\[\delta_{\varepsilon} = \frac{\varepsilon}{\|\mus\|_{\mathrm{TV}} }.\]
We finally observe that the number of particles $p$ needed to obtain an $\varepsilon$ approximation is then of the order:
\[
p_\varepsilon = |\cX| \|\mus\|_{\mathrm{TV}}^{d} \varepsilon^{-d}.
\]
Hence, if the number of iteration varies polynomially in terms of $\varepsilon^{-2}$, we observe the degradation of the number of particles needed to well approximate any distribution over $\cX$ in terms of the dimension $d$.

\begin{remark}
Results displayed in Theorem \ref{theo:global_sto}  (together with Proposition \ref{prop:TV_nuk}) hold for every possible value for the mini-batch sample size $(m_k)_{k\in \mathbb{N}^*}$. In particular, one can set $m_k=1$ for any $k\in \lbrace 1,\dots, K \rbrace$, which corresponds to the case where a single variable $Z^{k+1}$ is drawn at each step. Indeed, the control on the TV-norm and the global convergence results only require an almost sure bounded stochastic approximation of our gradients. In the next section (local convergence), we need stronger constraint, and in particular control on the variance of these approximations.  
\end{remark}

{
\begin{remark}
The bound \eqref{eq:LemmeF1Chizat} displayed at the end of the proof of Theorem \ref{theo:global_sto} can be considered as the analogue of Lemma F1 in \cite{chizat2022sparse} which is obtained with the deterministic instance of our algorithm. Taking advantage of this lemma, Theorem 4.2 of \cite{chizat2022sparse} then provide a global convergence result in some kind of asymptotic context. This result (with similar calibration for $\alpha$ and $\beta$ but associated exponential convergence rates) rely in particular on a Polyak–Łojasiewicz inequality that appears quite difficult to manage in stochastic context. Theorem \ref{theo:global_sto} should hence be understood as a non-asymptotic version of this control which appears to be of order $\mathcal{O}(\log(K)/\sqrt{K})$.
\end{remark}}

\subsubsection{Local minimization with swarm stochastic optimization}
 
To conclude this contribution, we state a complementary result that quantifies the behaviour of Algorithm~\ref{algo:SCPGD_pro} when the number of particles used is not ``as large'' as the one indicated in Theorem \ref{theo:global_sto}.

\begin{theorem}\label{theo:local_sto}
{Assume that the kernel is two times continuously differentiable}. Set $\alpha=\eta=1/\sqrt{K}$, {$m_k = \sqrt{K}$ for all $k\in \lbrace 1,\dots, K\rbrace$ and assume that   $\alpha C_1 (R_0+1) <1$ where $\mathcal{C}_1$ is introduced in Lemma \ref{lem:boundJprime}. For any initial measure satisfying the requirement of Proposition \ref{prop:TV_nuk}}, if $\tau_K$ refers to a random variable uniformly distributed over $\{1,\ldots,K\}$, independent from  $(\nuk)_{k \ge 1}$, then:
$$
\mathbb{E}\left[ \|J'_{\nu_{\tau_K}}\|^2_{\nu_{\tau_K}} + \|\nabla J'_{\nu_{\tau_K}}\|^2_{\nu_{\tau_K}} \right]  \leq  \frac{J(\nu_0)+\mathfrak{C} (1+R_0^4)}{\sqrt{K}},
$$
for some constant $\mathfrak{C}$ that depends on $\|\varphi\|_{\mathrm{Lip}}$.
\end{theorem}

\medskip

\noindent
Even though weaker than Theorem \ref{theo:global_sto}, the previous ``local'' result deserves several comments as it raises some meaningful and challenging questions.

\begin{remark} 
The previous result is a strong indicator of the convergence of our swarm stochastic particle algorithms towards a minimizer of $J$. 
Indeed, Proposition \ref{prop:first_order} stated in Appendix \ref{sec:gradients} states that any minimizer $\mu^{\star}$ of~$J$ necessarily satisfies $J'_{\mu^{\star}}=0$ on the support of $\mu^{\star}$ and $J'_{\mu^{\star}} \ge 0$ everywhere, which implicitly means that $\nabla J'_{\mu^\star}=0$ on the support of $\mu^\star$ otherwise the previous positivity condition would not hold. In Theorem \ref{theo:local_sto}, we obtain that 
$\|J'_{\nuk}\|^2_{\nu_{\tau_K}}$ and $\|\nabla J'_{\nuk}\|^2_{\nu_{\tau_K}}$ become arbitrarily small when $k$ becomes larger and larger, which is perfectly in accordance with the previous conditions. Nevertheless, the nature of our algorithm does not permit to push further our conclusions, especially about the positivity of $J'_{\nuk}$ \textit{everywhere} (and not only on the support of $\nuk$). In particular, to extend our conclusion towards a global minimization result, we need to be able to address a kind of ``density of the support'', which is absolutely unattainable in our present analysis without multiplying the number of particles all over the state space, which is in some sense what is done in Theorem \ref{theo:global_sto}.
{ Our convergence statements are expressed in terms of objective values $J(\nu_k)-J(\mu^\star)$ (Theorem~1.2) or derivative-based quantities such as $|J'(\nu_k)|^2+|\nabla J'(\nu_k)|^2$ (Theorem~1.3). Earlier works (\cite{chizat2022sparse,de2021supermix}) reported convergence in terms of proximity between $\nu_k$ and $\mu^\star$, which requires stronger local assumptions (``strong source conditions'') that yield a form of local strong convexity for $J$. Establishing analogous results in the stochastic setting (convergence in a measure divergence) is an interesting direction for future research.}
\end{remark}
{
\begin{remark}
The bound \eqref{eq:descente_J} displayed in the proof of Theorem \ref{theo:local_sto} can be related to the descent property provided in deterministic case (see \cite{chizat2022sparse}, Lemma 2.5) and holds for a wide range a possible values for $\alpha$ and $\beta$. These parameters are then specifically calibrated to obtain the local result displayed in Theorem~\ref{theo:local_sto}. 
\end{remark}}
{
\begin{remark}
 The stochastic setting yields the rate $K^{-1/2}$ in Theorem~1.3, which is minimax optimal for stochastic convex problems. Deterministic CPGD can achieve faster rates (even linear convergence under strong assumptions), as discussed in Chizat~(2022).
\end{remark}}

\section{Rationale of a Stochastic CPGD algorithm. }
\label{s:algo}

The primary objective of this contribution is to introduce a stochastic algorithm designed for tackling the optimization Program \eqref{def:Blasso}, followed by an exploration of its underlying theoretical properties. Our approach initially stems from a deterministic algorithm, which is subsequently adapted into a stochastic variant to enhance computational efficiency. Considering Program \eqref{def:Blasso} as an optimization challenge, it can be effectively addressed through the application of gradient descent (GD) techniques within the domain of measures. A straightforward algorithm would entail performing a discretized gradient descent on the space of non-negative measures $\measet_+$. However, in the absence of a significant conceptual breakthrough pertaining to an efficient parametrization of the preceding iteration within $\measet_+$ (or its dual spaces and Hilbert basis), we resort to emulating this gradient descent process through a collection of measures encoded with particles. This method of approximating optimization over measures through particles finds its conceptual foundation in swarm optimization approaches, as exemplified in recent works such as those of \cite{BolteMiclo} and \cite{MicloSolo}.

\subsection{A mirror principled conic gradient descent}
\subsubsection{The Mirror descent principle}
We first introduce a basic ingredient related to optimization problems on geometric spaces, that permits to adapt the evolution of an algorithm to some constrained sets where the  problem is embedded. Mirror Descent (MD below) originates from the pioneering work of \cite{NemirovskiYudin} and permits to naturally handle optimization problems especially when the mirror/proximal mapping is explicit, which is
indeed the case for a convex problem constrained on measures as $\measet_+$
(see {\it e.g.}, \cite{lan2012validation}, \cite{bubeck2015convex}). 

\medskip

Consider a strongly convex function $h$ {on the convex set} $\mathds{R}_+^{p}\times\cX^p$, we define the Bregman divergence associated to $h$ as follows: for two pairs $ (\weights_1,\pos_1) $ and $(\weights_2,\pos_2)$  in  $\mathds{R}_+^{p}\times\cX^p$, we denote:
\begin{equation}
    \label{def:Bregman}
    D_h( (\weights_1,\pos_1) , (\weights_2,\pos_2)) 
    = h (\weights_1,\pos_1) - h (\weights_2,\pos_2) - \langle \nabla h( \weights_2,\pos_2), (\weights_1,\pos_1) - (\weights_2,\pos_2) \rangle\,.
\end{equation}
The Bregman divergence $D_h$ is then used to define the MD with the following variational characterisation
\begin{equation}
\label{def:md}
    (\weights^{k+1},\pos^{k+1}) 
    = \arg\min_{(\weights,\pos)} 
    \Big\{
        \big\langle \nabla F(\weights^{k},\pos^{k}),(\weights,\pos)-(\weights^{k},\pos^{k})\big\rangle 
        + 
        \frac{1}{\kappa} D_h\left((\weights,\pos),(\weights^{k+1},\pos^{k+1})\right)
    \Big\}\,,
\end{equation}
where $\kappa>0$ is a gradient step size. 
\subsubsection{A conic descent}
In this contribution, we will consider the entropy function $\mathrm{Ent}$ on $\{\mathds{R}_+\}^{p}$ as:
\begin{equation}
\label{def:h}
    \mathrm{Ent}(\weights) 
    := \sum_{j=1}^p \omega_j \log(\omega_j)  
\end{equation}
It induces the Bregman divergence defined on the set of positive weights as
\begin{equation}
\label{def:bd}
    D_{\mathrm{Ent}}(\weights^1 , \weights^2 ) 
    = 
    \sum_{j=1}^p \omega^1_j 
    \Bigg( 
        \frac{\omega^2_j}{\omega^1_j}-1-\log \frac{\omega^2_j}{\omega^1_j}
    \Bigg)\,.  
\end{equation}
We then define the global divergence on the set $\mathds{R}_+^{p}\times \cX^p$, as
\begin{align}
   \label{def:Delta}
    \Delta_{\alpha,\eta}( (\weights^1,\pos^1) , (\weights^2,\pos^2)) 
        &:= \frac{1}{\alpha} D_{\mathrm{Ent}}(\weights^1 , \weights^2 ) + \frac{1}{\eta} D_{\mathrm{Conic}}( (\weights^1,\pos^1) , (\weights^2,\pos^2))\,,
\end{align}
where 
\begin{align}
    \label{def:Delta_conic}
    D_{\mathrm{Conic}}( (\weights^1,\pos^1) , (\weights^2,\pos^2))
        &:={\frac{1}{2}}\sum_{j=1}^p {\omega^{2}_j} \|\vt^1_j-\vt^2_j\|^2\,.
\end{align}
In Equation \eqref{def:Delta} above, the parameter $\alpha>0$ is the gradient step size for the weights update and $\eta>0$ for the positions updates. {Without loss of generality, we will consider from now on that $\kappa =1$ as the divergence~$\Delta_{\alpha,\eta}$ already incorporates gradients step sizes $\alpha$ and $\eta$}. This global divergence $\Delta_{\alpha,\eta}$ is used to compute the gradient descent updates thanks to the variational formulation~\eqref{def:md}. We incorporate the term $D_{\mathrm{Conic}}$ with the specific intention of aligning it with the gradient updates associated with Conic Particle Gradient Descent (CPGD), as presented for instance by \cite[Section 2.2]{chizat2022sparse}. 

\begin{remark}[The conic metric...]
We recall that in the CPGD framework \cite[Section 2.2]{chizat2022sparse}, the set of particles $(\omega,\vt)\in\bbR_+\times\cX$ is equipped with the Riemannian metric defined by $(1/2)\nabla_\omega^2+  (\omega/2)\| \nabla_\vt\|^2$ where $(\nabla_\omega,\nabla_\vt)\in\mathbb R\times T_{(\omega,\vt)}(\cX)$ is a tangent vector at point $(\omega,\vt)$. As displayed below, we recognize a ‘‘conic'' metric where two given fixed points~$\vt_1,\vt_2$ of $\cX$ gets linearly closer as $\sqrt{\omega}$ goes to zero.
\begin{center}
  \begin{tikzpicture}[scale=0.7,transform shape]
  \def\height{-3}
  \def\semimajoraxis{3.5}
  \def\semiminoraxis{1}
  \draw (0,0) ellipse (\semimajoraxis cm and \semiminoraxis cm);
  \draw[-latex] (0,0) -- (0,2);
  \foreach \theta in {-19,199}
    \draw (\theta:\semimajoraxis cm and \semiminoraxis cm) -- (0,\height);
    \draw[dashed] ({\semimajoraxis*cos(10)}, {\semiminoraxis*sin(10)}) -- (0,\height);
    \draw[dashed] ({\semimajoraxis*cos(30)}, {\semiminoraxis*sin(30)}) -- (0,\height);
  \draw[dashed] (0,0) -- (0,\height);
  \node[right] at (0,1.5) {$\sqrt{\omega}$};
  \draw[-latex] (0,0) -- ({\semimajoraxis*cos(30)}, {\semiminoraxis*sin(30)});
    \node[above right] at ({\semimajoraxis*cos(30)}, {\semiminoraxis*sin(30)}) {$\vt_1$};

  \draw[-latex] (0,0) -- ({\semimajoraxis*cos(10)}, {\semiminoraxis*sin(10)});
    \node[right] at ({\semimajoraxis*cos(10)}, {\semiminoraxis*sin(10)}) {$\vt_2$};
  
\end{tikzpicture}  
\end{center}
Then a mirror retraction is applied to $\omega$ (see \cite{chizat2022sparse}[Definition 2.3]), corresponding to the Bregman divergence term $D_{\mathrm{Ent}}$ in our variational formulation. We notice that the term $D_{\mathrm{Conic}}$ comes from the latter conic metric.
\end{remark} 

\begin{remark}[...is not a Bregman divergence]
It is noteworthy that $D_{\mathrm{Conic}}$ does not conform to the definition of a Bregman divergence, as outlined in Definition \eqref{def:Bregman}. Specifically, there exists no global function $h$ such that $D_{\mathrm{Conic}}$ can be expressed in the form of $D_h$. This is evident, for instance, by considering the boundedness of Bregman balls, a property that is conspicuously absent in the case of $D_{\mathrm{Conic}}$. Furthermore, a direct proof of this assertion can be established through a contradiction argument. If such a function $h$ were to exist for $p=1$, it would imply the following relationship
\[
h(\omega_1,\vt_1)-h(\omega_2,\vt_2)-\langle \nabla h(\omega_2,t_2),(\omega_1-\omega_2,\vt_1-\vt_2))= \omega_1 \|\vt_1-\vt_2\|^2\,.
\]
Then, consider $\vt_2=0$ and observe that necessarily $h(\omega_1,\vt_1) = h(0,0)+\omega_1 \|t_1\|^2$, by removing the linear part that is necessarily vanishing. Subsequent straightforward calculations would lead to a contradiction, thereby confirming the incompatibility of $D_{\mathrm{Conic}}$ with the Bregman divergence framework.
\end{remark} 

According to the above remark, our descent algorithm should rather be understood  as a Riemannian (stochastic) gradient descent instead of a purely mirror descent. We will keep the MD abuse of terms in what follows as it refers essentially to the evolution of the weights and since it is the commonly used term in machine learning with exponentially parametrized weights.  Nevertheless, the difference between~$\Delta_{\alpha,\eta}$ in~\eqref{def:Delta} and a Bregman divergence prevents the use of standard arguments of convergence for mirror descent algorithm.

{\subsection{Projected Stochastic conic particle gradient descent (\texttt{FastPart})} \label{sec:fastpartpro}}

With all these essential components in place, we are now prepared to construct our algorithm. The fundamental concept behind this approach is to substitute the deterministic gradient $\nabla F$ of $F$ in the mirror descent~\eqref{def:md} with its stochastic counterpart, as derived from the stochastic gradients on weights~\eqref{eq:J'sto} and positions~\eqref{eq:def_D}  under Assumptions \eqref{A1} and \eqref{A2}.\\

\paragraph{Mirror descent updates} 

{An important remark for the tractability of the stochastic CPGD is that Equation~\eqref{def:md} may be made explicit. In particular, the optimization with respect to  $\weights$ can be carried out independently of $\pos$. This leads to the updates
\begin{equation}
\weights^{k+1} = \arg\min_{\weights}  \left\{ \left\langle \weights - \weights^k,
    (\widehat{\mJ'_{\nu_k}}(\vt_j^k))_{j=1\dots p}\right\rangle +\frac{1}{ \alpha} D_{Ent}(\weights,\weights^k)\right\}
\label{eq:Woptim}
\end{equation}
and 
\begin{equation}
\pos^{k+1} = \arg\min_{\pos}
\left\{ \left\langle \pos - \pos^k,
 (\omega_j^k \widehat{\mD_{\nu_k}}(\vt_j^k))_{j=1\dots p}\right\rangle +\frac{1}{\eta} D_{Conic}((\weights,\pos),(\weights^k,\pos^k))\right\}.
 \label{eq:Toptim}
\end{equation}
which gives the expressions \eqref{eq:up_w} and \eqref{eq:update_pos_sto} of the algorithm.
}

\paragraph{Proximal methods guarantees}
{Analysis of Algorithm \ref{algo:SCPGD_pro} will be addressed by using some standard tools of proximal methods, and in particular the \textit{generalized projected gradient}. We briefly sketch these tools below.
For any $\vt \in \cX$, any "direction'' $d$ and any step-size $\eta$, we introduce:
$$
\vt^{+} := \arg\min_{u \in\cX} \left\{ \langle u,d\rangle + \frac{1}{2 \eta} \|u-\vt\|^2\right\}.
$$
The generalized projected gradient associated to a direction $d$ is then defined as
$$
P_{\cX}(\vt,d,\eta) := \frac{\vt-\vt^+}{\eta}\,,
$$
which depends implicitly on $d$ through $t^+$ as we also have that $d-P_{\cX}(\vt,d,\eta)$ is a vector of the normal cone of $\rad$ at point $t^+$. Hence, leading to the key identity
\begin{equation}
    \label{eq:projected_gradient}
    \vt^+ = \vt - \eta P_{\cX}(\vt,d,\eta).
\end{equation}
Some useful properties of this generalized projected gradient can be found for instance in \cite{Ghadimi-Lan-Zhang}. In particular Lemma 1 of \cite{Ghadimi-Lan-Zhang} may be stated as follows:
\begin{lemma}[Correlation of the projected gradient and gradient lower bound, \cite{Ghadimi-Lan-Zhang}\label{lem:Ghadimi-Lan-Zhang}]
    For any $\vt_j \in \cX$, $d \in \bbR^d$ and $\eta>0$:
    $$
    \langle  d, P_{\cX}(\vt_j,d,\eta) \rangle \ge \|P_{\cX}(\vt_j,d,\eta)\|^2.$$
\end{lemma}
\noindent A second key property is the generalization of the $1$-Lipschitz inequality for projection, which is obvious in our framework here:
\begin{lemma}[Lipschitz continuity of generalized projected gradients \cite{Ghadimi-Lan-Zhang}]\label{lem:Ghadimi-Lan-Zhang2}
    For any $\vt_j \in \cX$, $(d_1,d_2) \in \bbR^d$ and $\eta>0$:
    $$
    \|P_{\cX}(\vt,d_1,\eta) -P_{\cX}(\vt,d_2,\eta) \|  \le \|d_1-d_2\|.$$
\end{lemma}
\noindent
Using \eqref{eq:optim_t_projgrad} and \eqref{eq:update_pos_sto}, we have, for any $k\in \mathbb{N}$ and $j\in \lbrace 1,\dots, p \rbrace$
\begin{equation}
\vt_j^{k+1} - \vt_j^k = -\eta P_\rad(\vt_j^k, \widehat{\mD_{\nu_k}}(\vt_j^k),\eta).
\label{eq:update_generalized_gradient}
\end{equation}
A direct application of Lemma \ref{lem:Ghadimi-Lan-Zhang2} (setting $\vt=\vt_j^k$, $d_1= \widehat{\mD_{\nu_k}}(\vt_j^k)$ and $d_2 = 0$) then leads to 
\begin{equation}
\| \vt_j^{k+1} - \vt_j^k \| \leq \eta \| \widehat{\mD_{\nu_k}}(\vt_j^k) \|.
\label{eq:control_increment_pos}
\end{equation}
Despite the projection step over $\mathcal{X}$, inequality \eqref{eq:control_increment_pos} allows to connect the distance between $\vt_j^{k+1}$ and  $\vt_j^k$ to the norm of the gradient $\widehat{\mD_{\nu_k}}(\vt_j^k)$. This inequality will be a cornerstone in our analysis. 
}

\section{Some examples from Unsupervised learning and Signal processing}
\label{s:examples}
\subsection{Mixture Models (GMM)}
\label{sec:deconvolution_mixture}
\subsubsection{Introduction}
For the sake of clarity, we discuss briefly in this section the specific case of statistical mixture models. They are a class of statistical models that can be used for various purposes such as inference, testing, and modelling, have garnered significant attention in recent years due to their versatility and simplicity. However, the estimation of mixture models remains a complex task, with many aspects of the process not yet fully understood. The Expectation-Maximization (E.M.) algorithm, introduced by \cite{EM}, and its subsequent generalization to stochastic variants in \cite{EM-Sto}, have played a crucial role in the development of M.M.. Notably, the E.M. algorithm has been reinterpreted on exponential families as a descent algorithm with a surrogate in \cite{EM-Mirror}, which has led to a renewed interest in M.M. within the machine learning and optimization communities. Our work is also related to preliminary experiments conducted in \cite{de2021supermix}, which employed a deterministic version of particle gradient descent. These experiments demonstrated the potential of using such methods in the context of M.M., and have inspired further research in this area. While M.M. may appear straightforward at first glance, their estimation poses significant challenges. However, recent advances in the field, including the reinterpretation of the EM algorithm and the use of descent algorithms with surrogates, have reignited interest in these models and their potential applications.

In this setting, the data $\mX=(x_1,\dots, x_N)$ are i.i.d. random variables having a density $\rho$ in $\mathds{R}^d$ (w.r.t. the Lebesgue measure) verifying:
\[
    \rho 
    = \sigma \star \bar\mu = \sum_{j=1}^{\bar s} {\bar\omega}_j \sigma_{\bar\vt_j} 
    \text{ with } 
    \bar\mu := \sum_{j=1}^{\bar s} {\bar\omega}_j \delta_{\bar\vt_j}\,,
\]
where $\star$ denotes the convolution product, $\bar \mu$ is an \textit{unknown} mixing distribution, $\sigma$ is a \textit{known} bounded even probability density function on $\mathds{R}^d$ ({\it e.g.,} a Gaussian) and we denote by $\sigma_\vt(\cdot):=\sigma(\vt-\cdot)$. The goal is to recover the target $\bar \mu$ and/or the corresponding weights $\overline{\weights} := (\bar\omega_1,\ldots, \bar\omega_{\bar s})$ and positions $\overline{\pos} := (\bar\vt_1,\ldots, \bar \vt_{\bar s})$.

\subsubsection{Model specification}
We consider the Hilbert space $\mathds{H}$ defined as the RKHS associated to the Gaussian kernel, denoted by $\gamma_m$, with variance parameter $m^2 \mathrm{Id}$ and leading to 
\[
\mathds{H} :=  
    \Big\lbrace 
    f:\mathds{R}^d \rightarrow \mathds{R} \ : \ \| f\|_\mathds{H}^2 = \frac{1}{(2\pi)^d}\int_{\mathds{R}^d} 
    \frac
    {| \mathcal{F}[f](\vt)|^2}
    {\mathcal{F}[\gamma_m](\vt)} 
    \mathrm d\vt <+\infty   
    \Big\rbrace  \quad
    \text{with} \quad {\mathcal{F}[\gamma_m]}=\exp\big(-\frac{m^2}{2}\|\cdot\|_2^2\big)\,,
\]
where $\mathcal{F}$ is the Fourier transform defined as
\[
\forall \vu \in \mathds{R}^d\,, \quad 
\forall f \in \mathds{H}\,, \qquad \mathcal{F}[f](\vu) := \int_{\mathds{R}^d} f(\vt) e^{-\ii \langle \vu,\vt\rangle} \mathrm d \vt,
\]
and $\ii$ refers to the complex number. The inner product associated to the space $\mathds{H}$ hence verifies
\begin{equation} 
\forall f,g\in \mathds{H}\,,\qquad
    \langle f,g \rangle_\mathds{H} = \frac{1}{(2\pi)^d}
            \mathcal{R}\mathrm{e}\Big(\int_{\R^d} \frac{\mathcal{F}[f](\vt) \times \overline{\mathcal{F}[g](\vt)}}{\mathcal{F}[\gamma_m](\vt)} \mathrm d\vt 
\Big)\,,
\label{eq:scal}
\end{equation} 
where $\mathcal{R}\mathrm{e}(z)$ denotes the real part of a complex number $z$. We embed the sample $\mX$ and the density $\rho$ in the same Hilbert space $\mathds H$ taking a convolution with $\gamma_m$. It yields  
\[
    \vy =  \frac 1N \sum_{i=1}^N \gamma_m (x_i - \cdot)
    \quad
    \text{and}
    \quad
    \bar \vy :=\mathds E_\mX[\vy]=(\gamma_m\star \sigma)\star \bar\mu=\sum_{j=1}^{\bar s} {\bar\omega}_j(\gamma_m\star \sigma)(\bar\vt_j-\cdot)\,.
\]

\medskip

We uncover that the feature map~\eqref{eq:sparse solution} and the measure embedding~\eqref{def:Phi} are given by 
\[
\varphi_\vt(\cdot):=(\gamma_m\star \sigma)(\vt-\cdot)
\quad
\text{and}
\quad
\Phi(\mu):=(\gamma_m\star \sigma)\star \mu\,.
\] 
Note that the observation $\vy\in\mathds H$ corresponds to the non-parametric kernel estimator of $\bar\vy=\Phi(\bar\mu)$ based of the sample $\mX$. We also have the following kernel expression matching the expression~\eqref{eq:approx_kernel_convolution}
\begin{equation}
\label{eq:kernel_MM_convolution}
\Kernel(\vt,\vt')=\langle\varphi_\vt,\varphi_{\vt'}\rangle_\mathds{H}=(\tilde k\star \sigma)(\vt-\vt')
\quad \text{where}
\quad
\tilde k(\vt):=(\gamma_m\star \sigma)(\vt)\,,
\end{equation}
using the Fourier inversion formula. Last but not least, we can derive the following expression
\[
\langle\varphi_\vt,\vy\rangle_\mathds{H}=\frac 1N \sum_{i=1}^N \tilde k (x_i-\vt)=\vy(\vt)\,,
\]
by the same means. For the sake of simplicity, we will omit the dependency with the standard deviation $m$ as we are only interested in the optimization problem in the sequel.

\subsubsection{Gradients}
For any $\mu \in \mathcal{M}_+(\mathds{R}^d)$, the observation $\vy$ is compared to $\Phi(\mu) = (\gamma \star\sigma)\star\mu$ inside the criterion~$J(\mu)$. By~\eqref{eq:def_J_prime}, one has
\begin{equation}
    J_\nu'(\vt)
        = \langle \varphi_\vt, \Phi(\nu)  \rangle_\mathds{H} - \langle \varphi_\vt, \vy \rangle_\mathds{H} + \lambda
        = \sum_{j=1}^p \omega_j \langle \varphi_\vt, \varphi_{\vt_j}  \rangle_\mathds{H} - \langle \varphi_\vt, \vy \rangle_\mathds{H} + \lambda\,,
\label{eq:mixt1}
\end{equation}
which can be re-written as 
\begin{eqnarray}
    \label{eq:Jint}
    J_\nu'(\vt)= 
\sum_{j=1}^p \omega_j \int_{\R^d} \tilde k(\vt-\vt_j-\vu)\sigma(\vu)\mathrm{d}\vu 
- 
\frac 1N \sum_{i=1}^N \tilde k (x_i-\vt) + \lambda\,,
\end{eqnarray}
following~\eqref{eq:approx_kernel_convolution}.

\subsubsection{Assumptions}
We can identify the $\mathbf g$ and $\mathbf h$ functions appearing in Assumptions~\eqref{A1} and~\eqref{A2}. It holds 
\begin{align}
    \mathbf g_{\vt,\vt'}(\vu)&:=\tilde k(\vt-\vt'-\vu)\notag \\
    \mathbf h_\vt(\vv)&:=\tilde k (\vt-\vv)\,,\notag
\end{align}
and the laws of $U,V$ are independently distributed according to $U\sim \sigma$ and $V$ picking a data sample at random as we will see below. Now, observe that $\tilde k$ is bounded as $\gamma$ is, hence both functions are bounded and Assumption~\eqref{A1} is satisfied. Also 
\begin{align}
    \nabla_\vt\mathbf g_{\vt,\vt'}(\vu)&=\nabla\tilde k(\vt-\vt'-\vu)\notag \\
    \nabla_\vt\mathbf h_\vt(\vv)&=\nabla\tilde k (\vt-\vv)\,,\notag
\end{align}
and since $\gamma$ has bounded gradient and that the convolution is a contraction for the sup norm, it holds that Assumption~\eqref{A2} is satisfied. Finally, it is not hard to prove that Assumption~\eqref{ass:kernel_continuity} is satisfied since the kernel is a convolution by a smooth function.

\subsubsection{Stochastic counterpart}
Remark from \eqref{eq:Jint} that $J'_{\nu}$ is obtained through some integral computations. Given $\nu=\nu(\weights,\pos)$, we introduce a random variable $Z=(T,U,V)$ built as follows. Consider three independent random variables given by
\begin{equation} 
T \sim \frac{\nu}{\|\nu\|_{\mathrm{TV}}}\,,\quad 
U \sim \sigma\,,\quad 
\text{and}\quad V\sim \frac{1}{N} \sum_{i=1}^N \delta_{X_i}\,,
\label{eq:mixt2}
\end{equation}
where $\nu$ is assumed to be a discrete non-negative measure as displayed in Algorithm~\ref{algo:SCPGD_pro}, hence $\nu/\|\nu\|_{\mathrm{TV}}$ is a discrete probability measure. We have
\begin{align}
\mJ'_{\nu,\vt}(Z) 
&= 
\|\nu\|_{\mathrm{TV}}\tilde k(\vt-T-U)  
- \tilde k(\vt-V)
+\lambda\,,\\
&= \|\nu\|_{\mathrm{TV}}\mathbf g_{\vt,T}(U)  
                    - \mathbf h_\vt(V)
                    +\lambda\,.
\end{align}
and we uncover~\eqref{eq:J'sto}. According to \eqref{eq:mixt1} and \eqref{eq:mixt2}, we can verify that once $\nu$ is fixed, $\mJ'_{\nu,\vt}(Z) $ is an unbiased estimation of $J'_{\nu}$ and we can observe that 
 that there exists a random variable $\xi_{\nu}(\vt,Z)$ such that 
\begin{equation}\label{eq:grad_sto1}
\mJ'_{\nu,\vt}(Z) = J'_\nu(\vt) + \xi_{\nu}(\vt,Z) \quad \mathrm{with} \quad  \mathbb{E}_Z [\xi_{\nu}(\vt,Z)]= 0\,.
\end{equation}
A same remark occurs for the gradient of $J'_{\nu}$ since for any $\vt\in \mathds{R}^d$,
\begin{equation}
\label{eq:grad_J_sto}
\nabla_\vt J'_\nu(\vt) 
= \sum_{j=1}^p \omega_j \int_{\R^d} \nabla\tilde k(\vt-\vt_j-\vu)\sigma(\vu)\mathrm{d}\vu 
- 
\frac 1N \sum_{i=1}^N \nabla\tilde k (x_i-\vt)\,.
\end{equation}
We shall then define
\begin{align}
\mD_{\nu,\vt}(Z) &=\|\nu\|_{\mathrm{TV}}\nabla\tilde k(\vt-T-U)  
                    - \nabla\tilde k(\vt-V)
                    \,,\notag\\
                 &= \|\nu\|_{\mathrm{TV}}\nabla_\vt\mathbf g_{\vt,T}(U)  
                    - \nabla_\vt\mathbf h_\vt(V)
                    \,.\notag
\end{align}
and we uncover \eqref{eq:def_D}. We get that 
\begin{equation}
\label{eq:grad_sto2} 
\mD_{\nu,\vt}(Z) =  \nabla_\vt J'_\nu(\vt) + \zeta_{\nu}(\vt,Z) \quad \mathrm{with} \quad \mathbb{E}_Z [\zeta_{\nu}(\vt,Z)]= 0,
\end{equation}
for some centred random vector $\zeta_{\nu}(t,Z)$.

Therefore, in this example of mixture deconvolution, our situation perfectly fits the stochastic gradient setting where we can easily access some unbiased random realization of the true gradient of $F$ for any position of a current algorithm $\nu=\nu(\weights,\pos)$. 

\subsection{Interlude, a comparison with sketching}
    In the field of machine learning, replacing the computation of the complex integrals present in Equation~\eqref{eq:Jint} is commonly referred to as sketching \citep{keriven2018sketching}. The key distinction between our approach and sketching-type methods lies in how the frequencies $(U_k)$ are sampled. In the sketching approach, the frequencies are initially sampled only once at the beginning of the algorithm. Consequently, the gradients~$J'_{\nu}(\vt)$ and $\nabla_\vt J'_\nu(\vt)$ are approximated using a Monte-Carlo version of~\eqref{eq:Jint} that employs the same Monte-Carlo sample $(U_k)$ along the descent and the size of $U_k$ \textit{needs to be large} to guarantee a good Monte-Carlo approximation of the several integrals involved in the iterations of the method. Therefore, the sketching strategy of \citep{keriven2018sketching} leads to costly iterations as the size of $(U_k)$ is not negligible. In contrast, our method involves sampling a unique Monte-Carlo sample $U_k$ at each step $k$, and we replace the integrals of~\eqref{eq:Jint} with an evaluation at sample~$U_k$ (or we could also adopt a mini-batch strategy). It is important to note that in our case, the number of frequencies $(U_k)$ is equal to the number of steps $K$ (or times the size of mini-batch), whereas in sketching approaches, the number of frequencies is directly proportional to the number of parameters to be estimated, up to logarithmic factors. For more information on this topic, refer to \cite{keriven2018sketching}.

\subsection{Sparse deconvolution with positive definite kernel}
\label{sec:Super-Resolution}

\subsubsection{Introduction}

The statistical analysis of the $\ell_{1}$-regularization in the space of measures was initiated by Donoho \citep{donoho1992superresolution} and then investigated by \cite{gamboa1996sets}.  Recently, this problem has attracted a lot of attention in the ``{\it Super-Resolution}'' community and its companion formulation in ``{\it Line spectral estimation}''. In the Super-Resolution frame, one aims at recovering fine scale details of an image from few low frequency measurements, ideally the observation is given by a low-pass filter. The novelty of this body of work lies in new theoretical guarantees of the $\ell_{1}$-minimization over the space of discrete measures in a gridless manner, referred to as ‘‘off-the-grid'' methods. Some recent work on this topic can be found in \cite{Bredis_Pikkarainen_13,Tang_Bhaskar_Shah_Recht_13,candes2014towards,Candes_FernandezGranda_13,FernandezGranda_13,duval2015exact,de2012exact,azais2015spike}. More precisely, pioneering work can be found in~\cite{Bredis_Pikkarainen_13}, which treats of inverse problems on the space of Radon measures and \cite{candes2014towards}, which investigates the Super-Resolution problem via Semi-Definite Programming and the ground breaking construction of a ``{\it dual certificate}''. Exact Reconstruction property (in the noiseless case), minimax prediction and localization (in the noisy case) have been performed using the ‘‘{\it Beurling Lasso}'' estimator \eqref{def:Blasso} introduced in \cite{azais2015spike} and also studied in \cite{Tang_Bhaskar_Shah_Recht_13,FernandezGranda_13,Tang_Bhaskar_Recht_15} which minimizes the total variation norm over complex Borel measures. Noise robustness (as the noise level tends to zero) has been investigated in the captivating paper~\cite{duval2015exact}. A sketching formulation and a construction of dual certificates with respect to the Fisher metric has been pioneering studied in \cite{poon2021geometry}.

\subsubsection{Model specification}
In sparse deconvolution, the convolution $\bar\vy$ of some measure $\bar\mu$ with a $\mathcal C^1$-continuous {\bf positive definite} function~$\varphi$ is given by
\[
\forall \vt\in\cX\,,\quad
\bar\vy(\vt)=\sum_{j=1}^{\bar s}\bar\omega_j \varphi(\vt-\bar\vt_j)\,,
\]
and we observe $\vy$, a noisy version of $\bar\vy$, given by $\vy=\bar\vy+\ve$, where $\ve\,:\,\cX\to\bbR$ is some function. The feature map~\eqref{eq:sparse solution} is the convolution kernel and the measure embedding~\eqref{def:Phi} are given by 
\[
\varphi_\vt(\cdot):=\varphi(\vt-\cdot)
\text{ and }\Phi(\mu):=\varphi\star\mu\,.
\] 
Recall that $\varphi$ is a positive definite function. We assume that $\varphi$ is defined on the $d$-Torus $\mathds T^d$ or on $\bbR^d$, and we further assume that $\varphi(0)=1$, without log of generality on this latter point. By Bochner's theorem, there exists a probability measure $\sigma$, referred to as the spectral measure of $\varphi$, such that 
\begin{equation}
\forall \vt,\vs \in\cX\,, \quad  
\mathcal{F}[\varphi_\vt](\vs) =  \sigma(\vs) e^{-\ii \langle \vs,\vt\rangle}\,.
\label{eq:deconvol0}
\end{equation}
We consider the following assumption.
\medskip

\noindent
\textbf{Assumption $(\mathbf{A_{\mathrm{conv}}})$.} \textit{The spectral measure $\sigma$ of $\varphi$ has compact support}.

\paragraph{Super Resolution} In this case, $\cX=\mathds T^d$ and $\sigma=\sum_{\vu\in\mathds Z^d}\sigma_\vu\delta_\vu$ is a probability measure on $\mathds Z^d$. When $\sigma$ is the uniform measure on $[-f_c,f_c]^d$, with $f_c\geq 1$, we uncover the Super-Resolution framework and $\varphi$ is the Dirichlet kernel. For latter use, we denote 
\[
\int e^{-\ii \langle \vu,\vt-\vs\rangle}\mathrm d\sigma(\vu) := \sum_{\vu\in\mathds Z^d}\sigma_\vu e^{-\ii \langle \vu,\vt-\vs\rangle}\,.
\]

\paragraph{Continuous sampling Fourier transform}  In this case, $\cX$ is a compact set of $\mathds R^d$ and $\sigma$ is a probability measure on $\R^d$. For latter use, we denote 
\[
\int e^{-\ii \langle \vu,\vt-\vs\rangle}\mathrm d\sigma(\vu) := \int_{\mathds R^d}e^{-\ii \langle \vu,\vt-\vs\rangle}\mathrm d\sigma(\vu)\,.
\]

\medskip

By Lemma~\ref{lem:non_separable}, we can assume that $\bbH$ is the RKHS associated with the kernel defined by $\varphi$. In particular, 
\begin{equation}
    \label{eq:deconvol00}
    \forall \vt\in\cX\,, \quad 
    \langle \varphi_\vt, \varphi_{\vt_j}  \rangle_\mathds{H} = \varphi_{\vt_j}(\vt)
    \text{ and }
    \langle \varphi_\vt, \vy \rangle_\mathds{H}=\vy(\vt).
\end{equation}

\paragraph{About the noise term and the regularity of the observation} Using the projection $\Pi$ and the isometry~$\mathbb{i}$ of Lemma~\ref{lem:non_separable}, we can assume without loss of generality that the noise term $\ve$ belongs to $\bbH$, the RKHS associated with $\varphi$. By Assumption~$(\mathbf{A_{\mathrm{conv}}})$, it can shown that $\varphi$ is smooth, hence all the elements of $\bbH$ are smooth and so is $\ve$. Since $\cX$ is compact, we deduce that $\vy$ and its gradient are bounded. 

\subsubsection{Gradients}
For any $\mu \in \mathcal{M}(\mathds{R}^d)$, the observation $\vy$ is compared to $\Phi(\mu) = \varphi\star\mu$ inside the criterion~$J(\mu)$. A direct application of Proposition \ref{prop:jnuprime} leads to $J_\nu' =  \Phi^\star(\Phi(\nu) - \vy) + \lambda$, and one can check that, for any $\vt\in \mathds{R}^d$, 
\[
J_\nu'(\vt) 
    = \langle \varphi_\vt, \Phi(\nu) - \vy \rangle_\mathds{H} + \lambda\,.
\] 
The latter formula can be re-written as
\begin{equation}
    J_\nu'(\vt)
        = \langle \varphi_\vt, \Phi(\nu)  \rangle_\mathds{H} - \langle \varphi_\vt, \vy \rangle_\mathds{H} + \lambda
        = \sum_{j=1}^p \omega_j \langle \varphi_\vt, \varphi_{\vt_j}  \rangle_\mathds{H} - \langle \varphi_\vt, \vy \rangle_\mathds{H} + \lambda\,.
\label{eq:deconvol1}
\end{equation}
By \eqref{eq:deconvol0} and \eqref{eq:deconvol00}, it yields that
\begin{equation}
    J_\nu'(\vt)
        = \sum_{j=1}^p \omega_j \int  e^{-\ii \langle \vu,\vt-\vt_j\rangle}\mathrm d\sigma(\vu) - \vy(\vt) + \lambda\,.
\label{eq:deconvol11}
\end{equation}

\subsubsection{Assumptions}

We can identify the $\mathbf g$ and $\mathbf h$ functions appearing in Assumptions~\eqref{A1} and~\eqref{A2}. It holds 
\begin{align}
    \mathbf g_{\vt,\vt'}(\vu)&:=e^{-\ii \langle \vu,\vt-\vt'\rangle}\notag \\
    \mathbf h_\vt&:=\vy(\vt)\,,\notag
\end{align}
which are bounded functions and Assumption~\eqref{A1} is satisfied. Using the dominated convergence theorem for the bounded gradients 
\begin{align}
    \nabla_\vt\mathbf g_{\vt,\vt'}(\vu)&:=-\ii \vu e^{-\ii \langle \vu,\vt-\vt'\rangle}\mathbf{1}_{\{\vu\in\mathrm{Supp}(\sigma)\}}\notag \\
    \nabla_\vt\mathbf h_\vt&:=\nabla\vy(\vt)\,,\notag
\end{align}
where $\mathrm{Supp}(\sigma)$ denotes the support of the spectral measure $\sigma$. One can check that Assumption~\eqref{A2} is satisfied under Assumption~$(\mathbf{A_{\mathrm{conv}}})$.

\subsubsection{Stochastic counterpart}

Given $\nu=\nu(\weights,\pos)$, we introduce a random variable $Z=(T,U)$ built as follows (there is no $V$ random variable in this case). Consider two independent random variables given by
\begin{equation} 
T \sim \frac{\nu}{\|\nu\|_{\mathrm{TV}}} 
\quad
\text{and}\quad U \sim \sigma\,,
\label{eq:dec2}
\end{equation}
where $\nu$ assumed non-negative without loss of generality as discussed in~\eqref{def:nu}, hence $\nu/\|\nu\|_{\mathrm{TV}}$ is a discrete probability measure. We then define  
\begin{align}
\mJ'_{\nu,\vt}(Z) &= \|\nu\|_{\mathrm{TV}}e^{-\ii \langle U,\vt-T\rangle}   
                    - \vy(\vt)
                    +\lambda\,,\notag\\
                 &= \|\nu\|_{\mathrm{TV}}\mathbf g_{\vt,T}(U)  
                    - \mathbf h_\vt(V)
                    +\lambda\,,\notag
\end{align}
writing, with a slight abuse of notation, $\bm h_t(V)=\bm y(t)$. We hence uncover~\eqref{eq:J'sto}. According to \eqref{eq:deconvol1} and \eqref{eq:dec2}, we can verify that once $\nu$ is fixed, $\mJ'_{\nu,\vt}(Z) $ is an unbiased estimation of $J'_{\nu}$ and we can observe that 
 that there exists a random variable $\xi_{\nu}(\vt,Z)$ such that 
\begin{equation}\label{eq:grad_sto_dec}
\mJ'_{\nu,\vt}(Z) = J'_\nu(\vt) + \xi_{\nu}(\vt,Z) \quad \mathrm{with} \quad  \mathbb{E}_Z [\xi_{\nu}(\vt,Z)]= 0\,.
\end{equation}
A same remark occurs for the gradient of $J'_{\nu}$ since for any $\vt\in \mathds{R}^d$,
\begin{equation}
\label{eq:grad_J_sto_dec}
\nabla_\vt J'_\nu(\vt) 
= - \ii \sum_{j=1}^p \omega_j \int  \vu  e^{-\ii \langle \vu,\vt-\vt_j\rangle}\mathrm d\sigma(\vu) - \nabla\vy(\vt)\,.
\end{equation}
We shall then define
\begin{align}
\mD_{\nu,\vt}(Z) &=- \ii\|\nu\|_{\mathrm{TV}} U e^{-\ii \langle \vu,\vt-T\rangle}  
                    - \nabla\vy(\vt)
                    \,,\notag\\
                 &= \|\nu\|_{\mathrm{TV}}\nabla_\vt\mathbf g_{\vt,T}(U)  
                    - \nabla_\vt\mathbf h_\vt(V)
                    \,.\notag
\end{align}
and we uncover \eqref{eq:def_D}. We get that 
\begin{equation}
\label{eq:grad_sto2_dec} 
\mD_{\nu,\vt}(Z) =  \nabla_\vt J'_\nu(\vt) + \zeta_{\nu}(\vt,Z) \quad \mathrm{with} \quad \mathbb{E}_Z [\zeta_{\nu}(\vt,Z)]= 0,
\end{equation}
for some centered random vector $\zeta_{\nu}(t,Z)$.

\section{Numerical experiments on FastPart\label{sec:experiments}}

\subsection{Experimental setup}
In this short section, we develop a brief numerical study, that may be seen as a proof of concept, to assess the efficiency of our stochastic gradient descent, when using some sketched randomized evaluations of $J'_\nu$, with the help of sampling involved in Equations \eqref{eq:Jint} and \eqref{eq:mixt2}.\footnote{Our simulations greatly benefit from the previous work of Nicolas Jouvin \url{https://nicolasjouvin.github.io/}, while the original numerical Python code is made available here \url{https://forgemia.inra.fr/njouvin/particle_blasso}.}

For this purpose, we consider the Supermix problem introduced in \cite{de2021supermix}, which is described in Section \ref{sec:deconvolution_mixture}, when considering a mixture of Gaussian densities. We consider three toy situations in 1D. 
Figure \ref{fig:synthetic_gmm_1d} represents the mixture densities considered in this study, that contains for two of them 3 components, and for the last one 5 components. 

In Section~\ref{sec:NN_CH}, we present a different setting on a real dataset in dimension $d=8$ using two layers neural network and S-CPGD.

\begin{figure}
    \centering
    \includegraphics[width=0.32\textwidth]{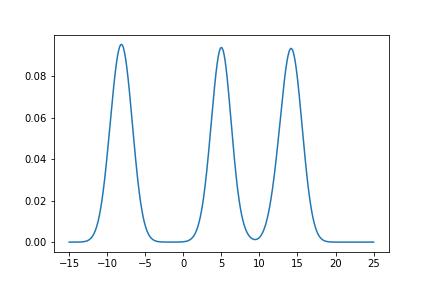}
    \includegraphics[width=0.32\textwidth]{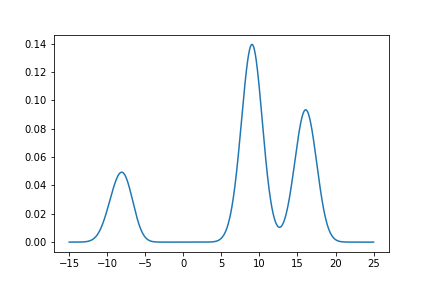}
    \includegraphics[width=0.32\textwidth]{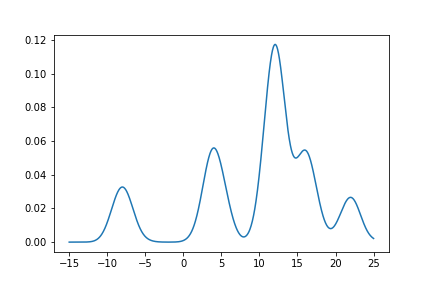}
    \includegraphics[width=0.32\textwidth]{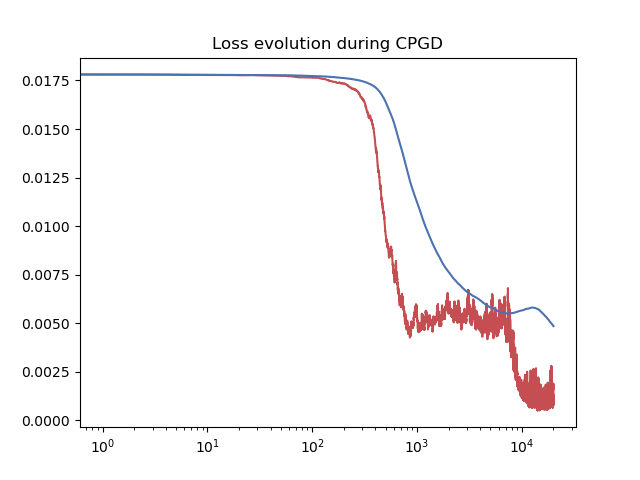}
    \includegraphics[width=0.32\textwidth]{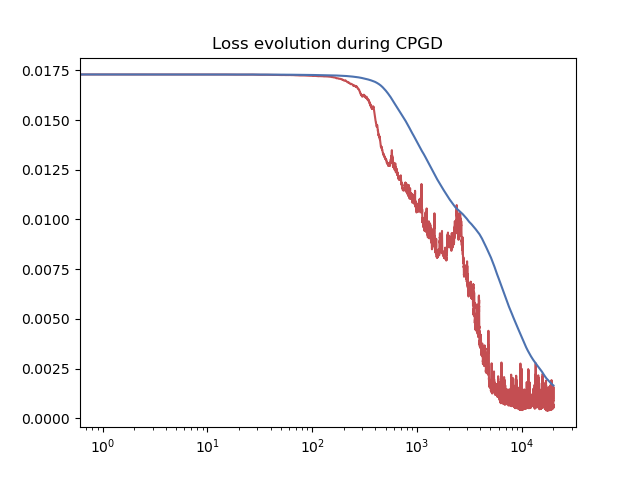}
    \includegraphics[width=0.32\textwidth]{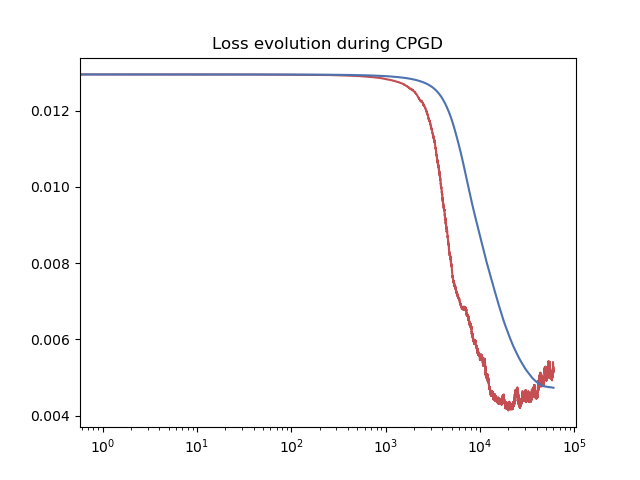}
    \caption{\textit{Top: Three 1-D Gaussian mixture distributions to be learnt by Supermix and Stochastic Conic Particle Gradient Descent. Bottom: Evolution of the loss (log scale) with our S-CPGD algorithm (red) and its averaged counterpart (in blue) when using our methods on the mixture problems.}}
    \label{fig:synthetic_gmm_1d}
    \label{fig:synthetic_loss_1d}
\end{figure}

\subsection{Benchmark}

Our experimental setup is essentially built with the help of three versions of the Conic Particle Gradient Descent.

\begin{itemize}
    \item   The first method we will use is the deterministic CPGD introduced in \cite{chizat2022sparse}, implemented by N. Jouvin \url{https://forgemia.inra.fr/njouvin/particle_blasso}. This method depends on the number of particles we use, the learning rate that encodes the gain of the algorithm at each iteration, and the number of iterations.
    \item   The second method is the Stochastic-CPGD introduced in this work. Our method depends on the same previous set of parameters (number of particles, learning rate, number of iterations) and of the batch size of the data we sample per iteration and the number of randomly sketched frequencies.
    \item   The last method is simply the Ces\`aro averaged counterpart of our Stochastic-CPGD, but this method raises some technical computational difficulties since averaging a sequence of measures seriously complicates the final estimates
    $\bar{\nu}_K = \frac{1}{K+1} \sum_{k=0}^K \nu_k$.
    To overcome this difficulty, we have chosen to use instead the measures supported by the averaged means all along the trajectory of the S-CPGD, and weighted by the averaged weights of the S-CPGD. For this purpose, we introduce
    $$
    \forall j \in \{1,\ldots,p\} \quad \forall K \ge 0 \qquad \bar{\vt}_j^K = \frac{1}{K+1} \sum_{k=0}^K \vt_j^k \qquad \text{and} \qquad 
    \bar{\omega}_j^K = \frac{1}{K+1} \sum_{k=0}^K \omega_j^k
    $$
    and we approximate $\bar{\nu}_K$ with the help of the sequence $\hat{\bar{\nu}}_K$, defined by:
\begin{equation}\label{def:approx_Cesaro}
\hat{\bar{\nu}}_K = \sum_{j=1}^p \bar{\omega}_j^K \delta_{\bar{\vt}_j^K}.
\end{equation}
Again, this sequence $\hat{\bar{\nu}}_K$ depends on several parameters, the learning rate, the number of particles and iterations, and the size of batches and sketches as well.
\end{itemize}

\subsection{Results}

\paragraph{Loss function: Averaging vs no averaging}
We show in Figure \ref{fig:synthetic_loss_1d} the evolution of the loss function $J_\nu$ over the iterations of the algorithm. We emphasize that the complexity of the S-CPGD and of the approximated Ces\`aro average are almost the same, since the sequence $\hat{\bar{\nu}}_K$ introduced in \eqref{def:approx_Cesaro} is a cheap approximation of the true Ces\`aro averaged sequence $\bar{\nu}_K$.

First, as indicated in Figure \ref{fig:synthetic_loss_1d}, we shall observe that the sequence $\hat{\bar{\nu}}_K$ always produces the desired smoothing effect all over the iterations of the algorithm, while slowing a bit the decrease of the loss function $J_\nu$ over the iterations. As a consequence, it seems more appropriate to use several long-range parallelized S-CPGD instead of a unique average thread of S-CPGD. 
In the same time, it may be remarked that our sequence $(\hat{\bar{\nu}}_K)_{K \ge 1}$ is a rough approximation of the true Ces\`aro averaging that is studied in our paper and the numerical approximation introduced in \eqref{def:approx_Cesaro} may not be as good as the true Ces\`aro sequence~$(\bar{\nu}_K)_{K \ge 1}$.

Second, as a classical phenomenon in machine learning when using stochastic approximation algorithm, or over-parametrized neural networks, our S-CPGD commonly generates some double-descent phenomena (see the 3 sub-figures of Figure \ref{fig:synthetic_loss_1d}) that translates some local minimizer escape of the swarm of particles.

\paragraph{Loss function: Averaging vs no averaging vs Deterministic}
Figure \ref{fig:synthetic_loss_1d_det} represents the evolution of the cost function with respect to the numerical cost which is a far better indicator than the number of iterations of the algorithm in our case since the S-CPGD algorithm is designed to be much more cheaper than the deterministic CPGD.
\begin{figure}[ht!]
    \centering
    \includegraphics[scale=0.34]{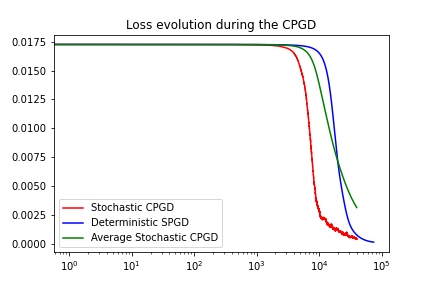}
    \includegraphics[scale=0.34]{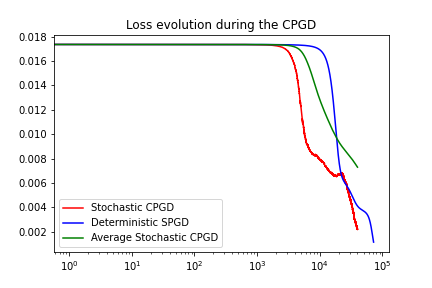}
    \includegraphics[scale=0.34]{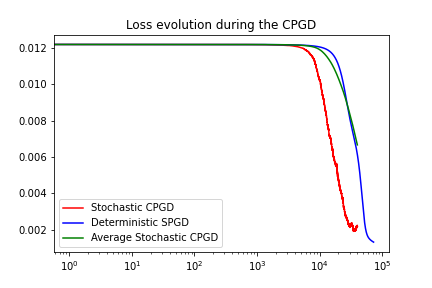} \\
    \includegraphics[scale=0.34]{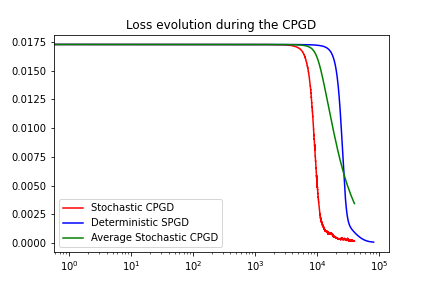}
    \includegraphics[scale=0.34]{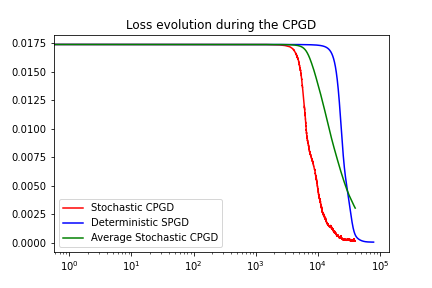}
    \includegraphics[scale=0.34]{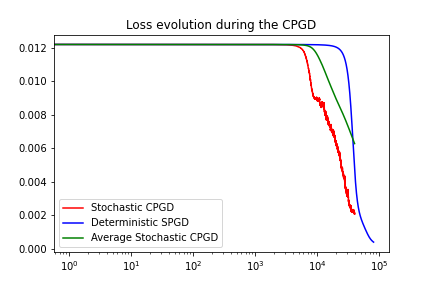} \\
    \caption{\textit{Evolution of the loss (log scale of the \textbf{computational time}) with our S-CPGD algorithm (red), its averaged counterpart (in green) and the deterministic CPGD (in blue) on the mixture problem of Figure~\ref{fig:synthetic_gmm_1d} with 20 particles (top) and with 50 particles (bottom). As indicated in the text, and observed with the shift to the right of the blue curve when compared to the red one, the cost of our S-CPGD is much cheaper than the deterministic CPGD.}}
    \label{fig:synthetic_loss_1d_det}
\end{figure}
Figure \ref{fig:synthetic_loss_1d_det} clearly illustrates the efficiency of our method with regards to the deterministic one as the red curve shows that the non-averaged S-CPGD produces comparable results as those obtained by CPGD with a significantly lower needs of computational cost: the red curve is clearly shifted on the left when compared to the blue one. It is furthermore possible to quantitatively assess the numerical gain produced by the S-CPGD when compared to the deterministic one: on our toy example, the deterministic CPGD requires approximately 4 more computations to attain the same decrease of the $J_\nu$, this effect being even amplified when the number of particles is increasing.

\paragraph{Loss function: Effect of the number of particles}

In the meantime, we observe that the loss function benefits from a large number of particles (see the comparison between top and bottom lines of Figure \ref{fig:synthetic_loss_1d_det}) but this should be tempered by the increasing number of simulations, which varies linearly with the number of particles. We should finally observe that using a large number of particles seems to be important especially in difficult situations (as illustrated in the right column of Figure \ref{fig:synthetic_loss_1d_det} where using 50 particles instead of 20 significantly improves the loss function, which is not the case on the right column of Figure \ref{fig:synthetic_loss_1d_det}.

\medskip

The effect of the number of particles can also be illustrated while looking at the trajectories themselves of the particles as shown in Figure \ref{fig:trajec_1d}. We observe that the number of particles is a clear key parameter that strongly influences the success of the method. In our example of 5 components GMM, (last example in Figure \ref{fig:synthetic_gmm_1d}), we see that a too small number of particles completely miss some components of the mixture while using a strongly over-parametrized set of particles permit to fully recover the support of the mixing distribution, even in the situation where some components of the mixture overlap.
\begin{figure}[ht!]
    \centering
    \includegraphics[width=0.49\textwidth]{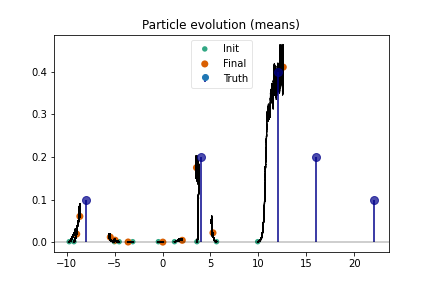}
    \includegraphics[width=0.49\textwidth]{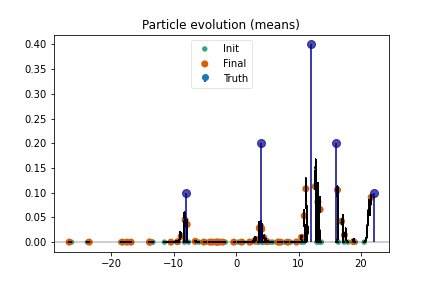}
\caption{\textit{Trajectories of the particles of our S-CPGD algorithm in the third example of the Gaussian mixture problem of Figure \ref{fig:synthetic_gmm_1d} with 5 modes. Left: trajectories using 10 particles. Right: Same with 50 particles. The l.h.s. shows the behaviour of S-CPGD when a too small number of particles is used. Particles concentrate around good positions but may miss some of the important locations of the Gaussian mixture. The r.h.s. demonstrates that a sufficiently large number of particles is necessary to guarantee an exhaustive reconstruction of the mixing distribution.}
    \label{fig:trajec_1d}}
\end{figure}

From our brief numerical study, we can conclude that both sketching and batch subsampling with a stochastic gradient strategy appears to strongly improve the numerical cost of the Conic Particle Gradient Descent, which permits to increase the number of particles used in the mean field approximation. We also have shown that in some difficult inverse problem examples, a large number of particles seems necessary to perfectly recover the solution of the optimization problem. It appears that the problem seriously benefits from a strong over-parametrisation, that may be handled with our cheap stochastic computing approach, which is not the case in a reasonable time with the deterministic CPGD.

\subsection{An example on a real dataset using two-layers neural network}
\label{sec:NN_CH}
We illustrate S-CPGD on the California Housing dataset, a standard dataset in Machine Learning see \url{https://inria.github.io/scikit-learn-mooc/python_scripts/datasets_california_housing.html}. It has $d=8$ features (dimension) and $20,640$ data points. We use a training set of size $N=18,576$, the remaining data points being the test set. Stochastic gradient descent was performed over batches of size $\mathrm{BS}=256$.

We use a simple neural network given by a hidden layer of size $p=500$ (number of particles) with ReLU activation function given by
\[
\sum_{i=1}^{500} \varepsilon_i\omega_i\, \mathrm{ReLu}(\langle\vt_i,\cdot\rangle)
\]
where $\vt_i\in\R^{8}$ are the location of the particles and $\varepsilon_i\omega_i\in\R$ their weights. Since the ReLu is one homogeneous, we project the locations onto the unit ball (radius $R_\cX=1$) at each iteration of the algorithm, the weights $\omega_i$ being scaled by $\|\vt_i\|_2$ in light of the identity:
\[
\sum_{i=1}^{500} \varepsilon_i\omega_i\, \mathrm{ReLu}(\langle\vt_i,\cdot\rangle)
=
\sum_{i=1}^{500} \varepsilon_i\omega_i\|\vt_i\|_2\, \mathrm{ReLu}(\langle\vt_i/\|\vt_i\|_2,\cdot\rangle)\,.
\]
This setup matches the one of \cite{bach2021gradient} and their experiments.

\begin{figure}
    \centering
    \includegraphics[width=0.45\linewidth]{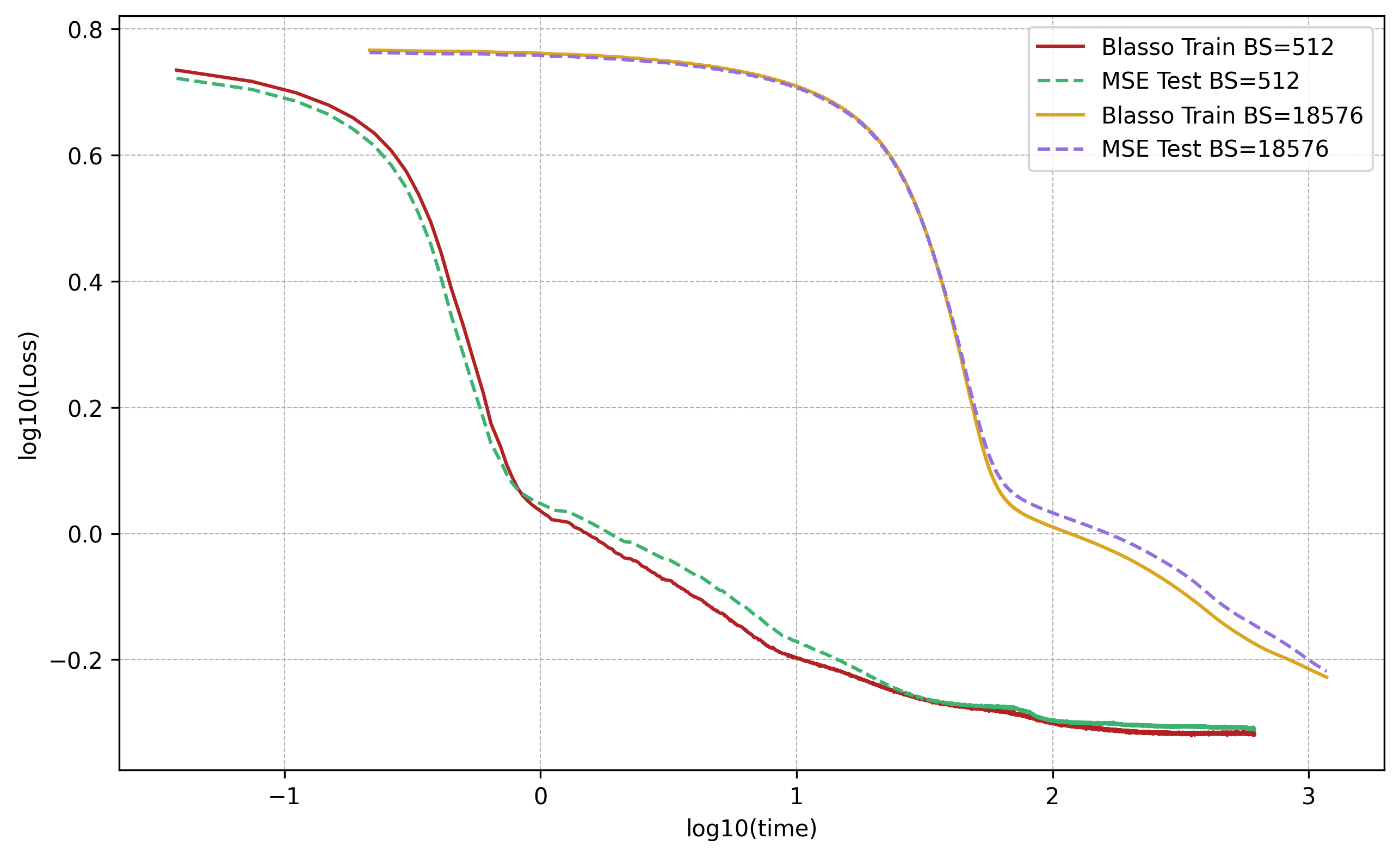}
    \includegraphics[width=0.45\linewidth]{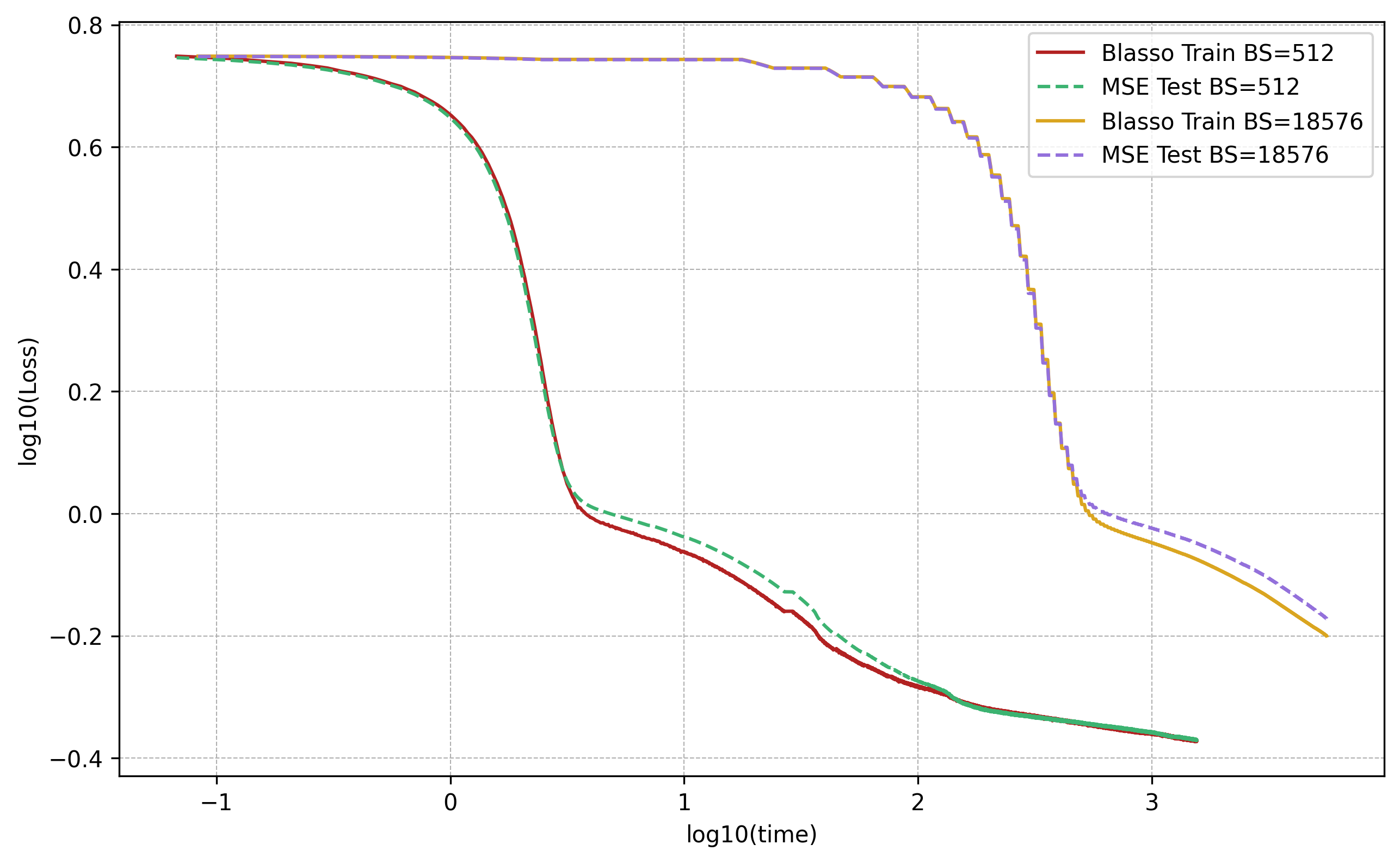}
    \caption{\textit{Stochastic/deterministic CPGD on California Housing dataset ($d=8$) using two-layers neural network with ReLu activation function. The stochastic version has a batch size $\mathrm{BS}=512$ while the deterministic CPGD computes gradients over the whole training set of size $\mathrm{BS}=18,576$}. We use $p=10$ particles on the right and $p=500$ on the left. The dotted line are the MSE on the test set of size $2,064$ over time. The plain line represents the BLASSO loss over time.}
    \label{fig:NN_SCPGD}
\end{figure}

\medskip

In Figure~\ref{fig:NN_SCPGD}, we observe a gain of two orders of magnitude when using the stochastic version of CPGD. The experiments were done on a single CPU of a computer (iMac with M3 chip) and shows that it takes around a few minutes to optimize $p=500$ particles in dimension $d=8$. 

\section{Proof of the main results \label{sec:proofs}}
In this section, $\mathfrak{C}$ will refer to a constant independent on $K$ and $p$, whose value may change from line to line.

{\subsection{Proof of Proposition \ref{prop:TV_nuk}}
\label{s:proofTV_nuk}}
\begin{proof}
{The principle of this proofs is as follows. We identify two radii $r$ and $r_0$ for which, for any $k \in \mathbb{N}^*$ and $j\in \lbrace 1,\dots, p \rbrace$, the two followings properties hold:
\begin{itemize}
\item If $\omega_j^k \leq r < r_0$, then $\omega_j^{k+1} \leq r_0$.
\item If $r \leq \omega_j^k \leq r_0$, then $\omega_j^{k+1} \leq \omega_j^k$.
\end{itemize}
Let $k\in \mathbb{N}$ and $j\in \lbrace 1,\dots, p\rbrace$ be fixed. First recall from \eqref{eq:up_w} that 
\[
\forall j \in \{1,\ldots,p\}\,, \qquad \omega^{k+1}_j = \omega^{k}_j e^{-\alpha \widehat{\mJ'_{\nu_k}}(\vt_j^k)}\, ,
\]
where
\[
\widehat{\mJ'_{\nu_k}}(\vt_j^k) = \frac{1}{m_k} \sum_{l=1}^{m_k} \mJ'_\nu(\vt_j^k,Z_l^{k+1}) = \frac{1}{m_k} \sum_{l=1}^{m_k} \left( \|\nu_k\|_{\mathrm{TV}} \, \vg_{\vt_j^k,T_l^{k+1}}(U_l^{k+1})  - \vh_{\vt_j^k}(V_l^{k+1}) + \lambda \right).
\]
In particular, we have 
\begin{equation}
\widehat{\mJ'_{\nu_k}}(\vt_j^k) \geq \omega_j^k \| \mathbf{g}\|_{\mathrm{Inf}} - \| \mathbf{h}\|_\infty +\lambda.
\label{eq:borneinf_J'}
\end{equation}
\underline{$1^{st}$ case: } If we have $ \lambda - \|\mathbf{h}\|_\infty \ge 0$,
then $\widehat{\mJ'_{\nu_k}}(\vt_j^k) \geq 0$ and $(\omega_j^k)_{k \ge 0}$ is a decreasing sequence, which yields $$ \forall k \ge 1 \qquad \|\nuk\|_{TV} \leq \|\nu_0\|_{TV}.$$
\underline{$2^{nd}$ case: } If we have $ \lambda - \|\mathbf{h}\|_\infty\le 0$, then 
\eqref{eq:borneinf_J'} together with \eqref{eq:up_w} entails that
\[\omega_j^k \geq \frac{\| \mathbf{h}\|_\infty -\lambda}{\| \mathbf{g}\|_{\mathrm{Inf}}}
\quad \Longleftrightarrow \quad \omega_j^k \| \mathbf{g}\|_{\mathrm{Inf}} - \| \mathbf{h}\|_\infty +\lambda \geq 0   \Longrightarrow \omega_j^{k+1} \leq \omega_j^k,
\]
whereas
\[
\omega_j^k  \leq \frac{\| \mathbf{h}\|_\infty -\lambda}{\| \mathbf{g}\|_{\mathrm{Inf}}}
\Longrightarrow \omega_j^{k+1} \leq \frac{\| \mathbf{h}\|_\infty -\lambda}{\| \mathbf{g}\|_{\mathrm{Inf}}} e^{ \| \mathbf{h}\|_\infty -\lambda}.
\]
Finally, we end the proof using an induction argument. }
\end{proof}
\subsection{Proof of Theorem \ref{theo:global_sto}}
\label{s:proof_global_sto}

\subsubsection{The shadow sequence}

Consider an integer $k \ge 0$ and the map $\mathcal{T}_{k+1}$ defined as:
{
\begin{equation} \label{def:Tk}
\forall t \in \cX \qquad 
\mathcal{T}_{k+1}(t) = \pi_\rad(t - \eta \widehat{\mD_{\nu_k}}(\vt_j^k)).
\end{equation}}
The sequence of maps $(\mathcal{T}_{k+1})_{k \ge 0}$ only acts on the positions of $\cX$ and is built with the random sequence $(\nu_k)_{k \ge 1}$.

Using $(\mathcal{T}_{k+1})_{k\geq 0}$, we then define the shadow sequence $(\nuke)_{k \ge 1}$ obtained through an iterative push-forward from a given initialisation measure $\nu_0^\epsilon \in \measet_+$ with the sequence of maps $(\mathcal{T}_k)_{k \ge 1}$. More formally, we set-up this definition in an iterative way:
\begin{equation}
\nukpe = \mathcal{T}_{k+1}^{\#}(\nuke) \qquad \forall k\in \mathbb{N}^\star,
\label{eq:shadow}
\end{equation}
where, for any continuous function $\psi$,
$$ \int_\cX \psi d\mathcal{T}_{k+1}^{\#}(\nu) = \int_\mathbb{} \psi(\mathcal{T}_{k+1}(.)) d\nu \quad \forall \nu \in \mathcal{M}(\cX). $$
The measure $\nu_0^{\varepsilon}$ will be defined carefully at the very end of our study. Roughly speaking, the shadow sequence $(\nuke)_{k \ge 1}$ moves exactly like $(\nuk)_{k \ge 1}$ and will share the same support, but the weights on the particles for the sequence $(\nuke)_{k \ge 1}$ will be optimised to allow for a good approximation of $\mus$. In particular, if we decompose the initial measure $\nu_0$ as
\[
\nu_0 = \sum_{j=1}^p \omega_j \delta_{\vt^0_j},
\]
then we can write $\nu_0^{\varepsilon} = \nu(\weights^{\varepsilon},\vt^{0})$, so that:
\[
\nu_0^{\varepsilon} = \sum_{j=1}^p \omega_j^{\varepsilon} \delta_{\vt^0_j},
\]
for some weights $(\omega_j^\epsilon)_{j=1..p}$ that will be chosen in an appropriate way.

\subsubsection{Excess risk decomposition}

The starting point is Proposition \ref{prop:jnuprime} that is used with $\nu=\nu_k$ and $\sigma=\mus-\nu_k$. We   write:
\begin{align}
    J(\nuk)-J(\mus)&=\int_{\cX} J'_{\nuk} \text{d}[\nuk-\mus] - \frac{1}{2} \|\Phi(\mus-\nuk)\|^2_\bbH \nonumber \\
    & = \underbrace{\int_{\cX} J'_{\nuk} \text{d}[\nu_k-\nuke]}_{:=\text{\textcircled{1}}} + 
    \underbrace{\int_{\cX} J'_{\nuk} \text{d}[\nuke-\mus]}_{:=\text{\textcircled{2}}}
    - \frac{1}{2} \|\Phi(\mus-\nuk)\|^2_\bbH \label{eq:decomposition}
\end{align}
where $(\nuke)_{k \ge 1}$ is the auxiliary shadow sequence of measures introduced in (\ref{eq:shadow}). First, we establish that the mirror descent adapts the weights of $(\nuk)_{k \ge 1}$ to those of the shadow sequence $(\nuke)_{k \ge 1}$. For any $\mu_1,\mu_2 \in \mathcal{M}(\cX)_+$, we introduce the following entropy:
\begin{equation}
\label{def:H}
\mathcal{H}(\mu_1,\mu_2) = -\int_{\cX} \log \left(\frac{\text{d}\mu_1}{\text{d}\mu_2}\right) \text{d} \mu_2 - \|\mu_2\|_{\mathrm{TV}}+\|\mu_1\|_{\mathrm{TV}}.
\end{equation}
The next proposition focuses on the first term of Equation \eqref{eq:decomposition}.

\begin{proposition}\label{prop:md}
Term \textcircled{1} of Equation \eqref{eq:decomposition} may be decomposed as:
    \begin{align*}
\text{\textcircled{1}} =    \int_{\cX} J'_{\nuk} \text{d}[\nuk-\nuke] &= 
\frac{1}{\alpha}\left[ \mathcal{H}(\nuk,\nuke) - \mathcal{H}(\nukp,\nukpe) \right]\\
&+ \frac{1}{\alpha} \sum_{j=1}^p \omega_j^{k}  \left[ \alpha  J'_{\nuk}(\vt_j^k) + e^{-\alpha \widehat{\mJ'_{\nu_k}}(\vt_j^k)} -  1 \right] 
+  \sum_{j=1}^p \omega_j^{\varepsilon}  {\widehat{\xi_{\nu_k}}(\vt_j^k))},
    \end{align*}
{where for all $k\in \lbrace 1,\dots, K \rbrace$ and for all $j\in \lbrace 1,\dots, p \rbrace$
\[
\widehat{\xi_{\nu_k}}(\vt_j^k):= \frac{1}{m_k} \sum_{l=1}^{m_k} \xi_{\nuk}(\vt_j^k,Z_l^{k+1}). \]}
\end{proposition}
\begin{proof}
Since both measures $\nuk$ and $\nuke$ share the same particle locations $\bm t^k$, we can remark that
    \begin{equation}\label{eq:dec_particules}
   \text{\textcircled{1}}=  \int_{\cX} J'_{\nuk} \text{d}[\nuk-\nuke] = \sum_{j=1}^p [\omega_j^{k}-\omega_j^{\varepsilon}] J'_{\nuk}(\vt_j^k)
    \end{equation}
We then observe from Equation \eqref{eq:up_w} that:
\[
\omega_j^{k+1} = \omega_j^{k} e^{-\alpha \widehat{\mJ'_{\nu_k}}(\vt_j^k} \quad \Rightarrow \quad \widehat{\mJ'_{\nu_k}}(\vt_j^k) = - \frac{1}{\alpha} \log \left( \frac{\omega_j^{k+1}}{\omega_j^{k}}\right).
\]
Using now Equation \eqref{def:grad_sto_J}, we observe that:
\[
J'_{\nuk}(\vt_j^{k}) = - \frac{1}{\alpha} \log \left( \frac{\omega_j^{k+1}}{\omega_j^{k}}\right) - \frac{1}{m_k} \sum_{l=1}^{m_k} \xi_{\nuk}(\vt_j^k,Z_l^{k+1}).
\]    
We then use the previous equality in \eqref{eq:dec_particules} and obtain that:
\begin{align*}
\text{\textcircled{1}} &= \sum_{j=1}^p  \left(\omega_j^{k} J'_{\nuk}(\vt_j^k)-\omega_j^{\varepsilon} \left[ - \frac{1}{\alpha} \log \left( \frac{\omega_j^{k+1}}{\omega_j^{k}}\right) - \widehat{\xi_{\nu_k}}(\vt_j^k)\right] \right),\\
&= \sum_{j=1}^p \left( \omega_j^{k} J'_{\nuk}(\vt_j^k) + \omega_j^{\varepsilon}   \frac{1}{\alpha} \log \left( \frac{\omega_j^{k+1}}{\omega_j^{k}}\right)\right)+ \sum_{j=1}^p \omega_j^{\varepsilon}  \widehat{\xi_{\nu_k}}(\vt_j^k) \\
& = \frac{1}{\alpha} \sum_{j=1}^p \left( \alpha \omega_j^{k} J'_{\nuk}(\vt_j^k) + \omega_j^{\varepsilon} \left[  \log \left( \frac{\omega_j^{k+1}}{\omega_j^{\varepsilon}}\right) - \log \left( \frac{\omega_j^{k}}{\omega_j^{\varepsilon}}\right)\right]\right)
+ \sum_{j=1}^p \omega_j^{\varepsilon}  \widehat{\xi_{\nu_k}}(\vt_j^k).
\end{align*}
We use the entropy $\mathcal{H}$ introduced in Equation \eqref{def:H} and deduce that:
\begin{align*}
\text{\textcircled{1}} &=\frac{1}{\alpha} \sum_{j=1}^p \left[ \alpha \omega_j^{k} J'_{\nuk}(\vt_j^k) + \omega_j^{k+1} -  \omega_j^{k} \right] + \frac{
\mathcal{H}(\nuk,\nuke) - \mathcal{H}(\nukp,\nukpe) }{\alpha} +  \sum_{j=1}^p \omega_j^{\varepsilon}  \widehat{\xi_{\nu_k}}(\vt_j^k) \\
& =\frac{1}{\alpha} \sum_{j=1}^p \omega_j^{k}  \left[ \alpha  J'_{\nuk}(\vt_j^k) + e^{-\alpha \widehat{\mJ'_{\nu_k}}(\vt_j^k)} -  1 \right] + 
 \frac{\mathcal{H}(\nuk,\nuke) - \mathcal{H}(\nukp,\nukpe) }{\alpha}  +  \sum_{j=1}^p \omega_j^{\varepsilon}  \widehat{\xi_{\nu_k}}(\vt_j^k).
\end{align*}
We then obtain the conclusion of the proof.
\end{proof}
Now, we study the second term of Equation \eqref{eq:decomposition}, which is an ``approximation'' term. We essentially follow the same methodology proposed in \cite{chizat2022sparse} but we use the specificity of our model to properly analyze this term.
For this purpose, we use the BL norm (over functions) and dual norm (over measures)
introduced in \cite{chizat2022sparse}, defined as:
\begin{equation}\label{def:norm_BL}
\forall f: \cX \rightarrow \mathds{R} \qquad \|f\|_{BL} = \|f\|_{\infty} + \|f\|_{\mathrm{Lip}},
\end{equation}
where $\|.\|_{\infty}$ refers to the supremum norm over $\cX$, $\|.\|_{\mathrm{Lip}}$ to the Lipschitz constant for $f$, and
\begin{equation}\label{def:dual_BL}
\forall \nu \in  \measet_+ \qquad \|\nu\|^*_{BL} = \sup_{\|f\|_{BL} \leq 1} \int f \text{d}\nu.
\end{equation}

Using these notation, we can propose a bound on the second term of Equation \eqref{eq:decomposition} as displayed in the following proposition. 

\begin{proposition}
\label{prop:approx}
The approximation term \textcircled{2} satisfies:
    \[
    \forall k \ge 1 \qquad 
\text{\textcircled{2}}=    \int_{\cX} J'_{\nuk} \text{d}[\nuke-\mus] \leq
    \mathfrak{A}_k \|\nuke-\mus\|_{BL}^* \quad \text{a.s.}  
    \]
    where $\mathfrak{A}_k$ is given by:
    \begin{equation}\label{def:Ak}
    \mathfrak{A}_k = \mathcal{C}_0(\|\nu_k\|_{\mathrm{TV}} +1) + Lip(\varphi)\left[ \|\nu_k\|_{\mathrm{TV}}\|\varphi\|_{\infty,\mathds{H}} + \|y\|_\mathds{H} \right]\,,
    \end{equation}
where 
$$ \mathcal{C}_0 = \max( \lambda + \|\varphi\|_{\infty,\mathds{H}}\|y\|_\mathds{H} ; \|\varphi\|_{\infty,\mathds{H}}^2)\,.$$
\end{proposition}

\begin{proof}
We can immediately remark that 
$$ \text{\textcircled{2}}=    \int_{\cX} J'_{\nuk} \text{d}[\nuke-\mus] \leq \| J_{\nu_k}' \|_{BL} \| \nuke-\mus\|_{BL}^*. $$
Then, according to Lemma \ref{lem:boundJprime}, 
$$ \| J_{\nu_k}' \|_{BL} = \| J_{\nu_k}' \|_\infty + \| J_{\nu_k}' \|_{\mathrm{Lip}} \leq \mathcal{C}_0(\|\nu_k\|_{\mathrm{TV}} +1) + \| J_{\nu_k}' \|_{\mathrm{Lip}}. $$
To conclude the proof, we have to propose an upper bound on $\| J_{\nu_k}' \|_{\mathrm{Lip}}$. For any $s,t\in \cX$ we have 
\begin{eqnarray*} 
|J_{\nu_k}'(s)- J_{\nu_k}'(t)| 
& = & \left| \sum_{j=1}^p \omega_j^k \langle \varphi_t-\varphi_s, \varphi_{\bm t_j^k}\rangle_{\mathds{H}} - \langle \varphi_t-\varphi_s,\bm y\rangle_\mathds{H} \right|, \\
& \leq & \sum_{j=1}^p \omega_j^k \| \varphi_t-\varphi_s\|_\mathds{H} \|\varphi_{\bm t_j^k}\|_\mathds{H} + \| \varphi_t-\varphi_s\|_\mathds{H} \|\bm y\|_\mathds{H}, \\
& \leq & Lip(\varphi)\left[ \|\nu_k\|_{\mathrm{TV}}\|\varphi\|_{\infty,\mathds{H}} + \|y\|_\mathds{H} \right] \times \|t-s\|_\cX.
\end{eqnarray*}
The results is obtained by gathering the previous bounds. 
\end{proof}

We finally introduce a key term that quantifies the way where $\mu^\star$ can be approximated by a discrete measure. This term, denoted by $\mathcal{Q}$, is defined as
\begin{equation}
    \label{def:Q}
    \mathcal{Q}_{\mus,\nu_0}(\tau) := \inf_{\mu \in \measet_+} \left[\|\mus-\mu\|_{BL}^*+\frac{1}{\tau}\mathcal{H}(\mu,\nu_0)\right] \quad \forall \tau>0.
\end{equation}

\subsubsection{Proof of Theorem \ref{theo:global_sto}}

\begin{proof}
The proof is decomposed into three steps.\\

\noindent \underline{Step 1: Decomposition of the excess risk with the convexity of $J$.}

\noindent We denote by $(\mathfrak{F}_k)_{k \ge 0}$ the natural canonical filtration associated to the sequence of random variables~$(Z^k)_{k \ge 0}$. The next upper bound is a consequence of the relationship (\ref{eq:decomposition}) and Propositions \ref{prop:md}, \ref{prop:approx}. We have
\begin{align*}
J(\nuk)-J(\mus) &\leq  \frac{\mathcal{H}(\nuk,\nuke) - \mathcal{H}(\nukp,\nukpe) }{\alpha}  +     \mathfrak{A}_k \|\nuke-\mus\|_{BL}^* \\
& + \frac{1}{\alpha} \sum_{j=1}^p \omega_j^{k}  \left[ \alpha  J'_{\nuk}(\vt_j^k) + e^{-\alpha \widehat{\mJ'_{\nu_k}}(\vt_j^k)} -  1 \right] 
+  \sum_{j=1}^p \omega_j^{\varepsilon}   \widehat{\xi_{\nu_k}}(\vt_j^k).
\end{align*}
We then use a telescopic sum argument and obtain that:
\begin{align*}
\sum_{k=0}^K \left(J(\nuk)-J(\mus)\right) & \leq \frac{\mathcal{H}(\nu_0,\nu_0^\varepsilon)}{\alpha} + \sum_{k=0}^K\mathfrak{A}_k \|\nuke-\mus\|_{BL}^* \\
& + \frac{1}{\alpha} \sum_{k=0}^K  \sum_{j=1}^p \omega_j^{k}  \left[ \alpha  J'_{\nuk}(\vt_j^k) + e^{-\alpha \widehat{\mJ'_{\nu_k}}(\vt_j^k)} -  1 \right] + \sum_{k=0}^K \sum_{j=1}^p \omega_j^{\varepsilon}   \widehat{\xi_{\nu_k}}(\vt_j^k)).
\end{align*}
Finally, using the convexity of $J$, the Ces\`aro average defined by:
\begin{equation}
\bar{\nu}_K = \frac{1}{K+1} \sum_{k=0}^K \nuk, \end{equation}
satisfies:
\begin{align*}
J(\bar{\nu}_K)-J(\mus) & \leq \frac{\mathcal{H}(\nu_0,\nu_0^\varepsilon)}{\alpha (K+1)} + \frac{\sum_{k=0}^K\mathfrak{A}_k \|\nuke-\mus\|_{BL}^*}{K+1} \\
& + \frac{\sum_{k=0}^K  \sum_{j=1}^p \omega_j^{k}  \left[ \alpha  J'_{\nuk}(\vt_j^k) + e^{-\alpha \widehat{\mJ'_{\nu_k}}(\vt_j^k)} -  1 \right]}{\alpha (K+1)}  + \frac{\sum_{k=0}^K \sum_{j=1}^p \omega_j^{\varepsilon}   \widehat{\xi_{\nu_k}}(\vt_j^k)}{K+1} \\
&  \leq \frac{\mathcal{H}(\nu_0,\nu_0^\varepsilon)}{\alpha K} + \frac{\sum_{k=0}^K\mathfrak{A}_k \left[  \|\nu_{0}^{\varepsilon}- \mus\|^*_{BL} + \sum_{\ell=0}^{k-1} \|\nu_{\ell+1}^{\varepsilon}- \nu_{\ell}^{\varepsilon}\|_{BL}^* \right]}{K} \\
& + \frac{\sum_{k=0}^K  \sum_{j=1}^p \omega_j^{k}  \left[ \alpha  J'_{\nuk}(\vt_j^k) + e^{-\alpha \widehat{\mJ'_{\nu_k}}(\vt_j^k)} -  1 \right]}{\alpha K}  + \frac{\sum_{k=0}^K \sum_{j=1}^p \omega_j^{\varepsilon}   \widehat{\xi_{\nu_k}}(\vt_j^k)}{K} \\
\end{align*}
where we used the triangle inequality on the telescopic decomposition \[\nuke-\mus = (\nuke-\nu_{k-1}^{\varepsilon}) + (\nu_{k-1}^{\varepsilon}-\nu_{k-2}^{\varepsilon})+\ldots + (\nu_0^{\varepsilon}-\mus).\]
We then take the expectation and use in particular a standard conditional expectation argument. Since $\mathbb{E}[ \xi_{\nuk}(\vt_j^k,Z_l^{k+1}) \vert \mathfrak{F}_k] = 0$ for any $l\in \lbrace 1,\dots, m_k\rbrace$, we deduce that: 
\begin{align}
\mathbb{E} \left[ J(\bar{\nu}_K)-J(\mus) \right] & \leq \frac{\mathcal{H}(\nu_0,\nu_0^\varepsilon)}{\alpha (K+1)} + \overbrace{
  \|\nu_{0}^{\varepsilon}- \mus\|^*_{BL} \frac{\sum_{k=0}^K\mathbb{E}[\mathfrak{A}_k]}{K+1}}^{:=A_1}+ 
\overbrace{  \frac{\sum_{k=0}^K\mathbb{E}\left[\mathfrak{A}_k \sum_{\ell=0}^{k-1} \|\nu_{\ell+1}^{\varepsilon}- \nu_{\ell}^{\varepsilon}\|_{BL}^*\right]}{K+1}}^{:=A_2} \nonumber\\ 
  & + \underbrace{\frac{\sum_{k=0}^K  \sum_{j=1}^p \mathbb{E} \left[ \omega_j^{k}  \left[ \alpha  J'_{\nuk}(\vt_j^k) + e^{-\alpha \widehat{\mJ'_{\nu_k}}(\vt_j^k)} -  1 \right] \right]}{\alpha (K+1)}}_{:=A_3}.
\end{align}
 
\noindent \underline{Step 2a: Study of $A_1$.}
We use the definition of $\mathfrak{A}_k$ in Equation \eqref{def:Ak} and observe that
$\mathfrak{A}_k \leq \mathfrak{C}(1+\|\nuk\|_{\mathrm{TV}})$. We then use Proposition \ref{prop:TV_nuk} and conclude that:
\begin{equation}\label{eq:up_A1}
\mathbb{E}[A_1] =  \|\nu_{0}^{\varepsilon}- \mus\|^*_{BL} \frac{\sum_{k=0}^K\mathbb{E}[\mathfrak{A}_k]}{K+1} \leq  \mathfrak{C} R_0  \|\nu_{0}^{\varepsilon}- \mus\|^*_{BL}.
\end{equation}

\noindent \underline{Step 2b: Study of $A_2$.}
We  focus on the shadow sequence that involves $\nu_{\ell+1}^{\varepsilon}- \nu_{\ell}^{\varepsilon}$. Using (\ref{eq:control_increment_pos}), observe that:
\begin{align}
\|\nu_{\ell+1}^{\varepsilon}- \nu_{\ell}^{\varepsilon}\|_{BL}^* & = \sup_{\|\psi\|_{BL} \leq 1} \int_{\cX} \psi(t) \text{d} [\nu^{\varepsilon}_{\ell+1}-\nu^{\varepsilon}_{\ell}](t)\nonumber\\
& = \sup_{\|\psi\|_{BL} \leq 1} \sum_{j=1}^p \omega_j^{\varepsilon} [\psi(\vt_j^{\ell+1}) - \psi(\vt_{j}^{\ell})] \nonumber\\
& \leq  \sum_{j=1}^p \omega_j^{\varepsilon} \|\vt_j^{\ell+1} - \vt_{j}^{\ell}\|_\cX \nonumber\\
& =  \eta \sum_{j=1}^p \omega_j^{\varepsilon} \|  \widehat{\mD_{\nu_k}}(\vt_j^k \|_\cX \nonumber\\
& \leq  \mathfrak{C} \eta \sum_{j=1}^p \omega_j^{\varepsilon} (1+\|\nu_{\ell}\|_{\mathrm{TV}}) \nonumber\\
& = \mathfrak{C} \eta \|\nu_0^{\varepsilon}\|_{\mathrm{TV}}(1+\|\nu_{\ell}\|_{\mathrm{TV}}),\label{eq:borne_tec_2}
\end{align}
where we used the almost sure upper bound in Proposition \ref{prop:bornes_techniques_J} and the fact that $\pi_\rad$ is $1$-Lipschitz. 
A simple sum yields:
\begin{align*}
\frac{\sum_{k=0}^K\mathfrak{A}_k  \sum_{\ell=0}^{k-1} \|\nu_{\ell+1}^{\varepsilon}- \nu_{\ell}^{\varepsilon}\|_{BL}^* }{K+1} 
&\leq \mathfrak{C} \eta \|\nu_0^{\varepsilon}\|_{\mathrm{TV}} \frac{\sum_{k=0}^K\mathfrak{A}_k \sum_{\ell=1}^k \left( 1+ \|\nu_{\ell}\|_{\mathrm{TV}}\right)}{K+1} \\
& \leq \mathfrak{C} \eta \|\nu_0^{\varepsilon}\|_{\mathrm{TV}}
\frac{\sum_{k=0}^K (1+\|\nuk\|_{\mathrm{TV}})\sum_{\ell=1}^k \left( 1+ \|\nu_{\ell}\|_{\mathrm{TV}}\right)}{K+1},
\end{align*}
for $\mathfrak{C}$ large enough. We apply Proposition \ref{prop:TV_nuk} and compute the expectation of the previous term. We then observe that:
\begin{equation}\label{eq:up_A2}
\mathbb{E}[A_2] = 
\frac{\sum_{k=0}^K\mathbb{E} \left[ \mathfrak{A}_k  \sum_{\ell=0}^{k-1} \|\nu_{\ell+1}^{\varepsilon}- \nu_{\ell}^{\varepsilon}\|_{BL}^*\right] }{K+1}  \leq \mathfrak{C} R_0^2 \eta \|\nu_0^{\varepsilon}\|_{\mathrm{TV}} K.
\end{equation}

\noindent \underline{Step 2c: Study of $A_3$.} This last term deserves a specific study. We use a conditional expectation argument and observe that, for any $k\in \lbrace 0,\dots, K\rbrace$, $j\in \lbrace 1,\dots, p \rbrace$,
$$
 \mathbb{E} \left[ \omega_j^{k} \left. \left[ \alpha  J'_{\nuk}(\vt_j^k) + e^{-\alpha \widehat{\mJ'_{\nu_k}}(\vt_j^k)} -  1 \right] \right\vert \mathfrak{F}_k \right]  = \omega_j^{k} 
  \mathbb{E} \left[  \left. \left[ \alpha  J'_{\nuk}(\vt_j^k) + e^{-\alpha \widehat{\mJ'_{\nu_k}}(\vt_j^k)} -  1 \right] \right\vert \mathfrak{F}_k\right].
$$
We then apply Proposition \ref{prop:hoeffding} ang get:
$$
\mathbb{E} \left[ \omega_j^{k}  \left.\left[ \alpha  J'_{\nuk}(\vt_j^k) + e^{-\alpha \widehat{\mJ'_{\nu_k}}(\vt_j^k)} -  1 \right] \right\vert \mathfrak{F}_k\right] \leq \mathfrak{C} \omega_{j}^k
  \alpha^2 (1+\|\nuk\|_{\mathrm{TV}}^2)\,.
$$
We then sum the previous upper bounds from $0$ to $K$ with a global expectation and Proposition \ref{prop:TV_nuk}. We obtain that:
\begin{equation}\label{eq:up_A3}
    \mathbb{E}[A_3] = \frac{\sum_{k=0}^K  \sum_{j=1}^p \mathbb{E} \left[ \omega_j^{k}  \left[ \alpha  J'_{\nuk}(\vt_j^k) + e^{-\alpha \widehat{\mJ'_{\nu_k}}(\vt_j^k)} -  1 \right] \right]}{\alpha (K+1)} \leq \mathfrak{C} \alpha R_0^3.
\end{equation}

\noindent \underline{Step 3: End of the proof.}

\noindent
We gather Equations \eqref{eq:up_A1}, \eqref{eq:up_A2}, \eqref{eq:up_A3} and obtain that:
\[
\mathbb{E} \left[ J(\bar{\nu}_K)-J(\mus) \right]  \leq  \frac{\mathcal{H}(\nu_0,\nu_0^\varepsilon)}{\alpha K}  +   \mathfrak{C} R_0 \|\nu_0^{\varepsilon}-\mus\|_{BL}^*+\mathfrak{C}R_0^2 \left[ \eta \|\nu_0^{\varepsilon}\|_{\mathrm{TV}} K + \alpha R_0\right].
\]
We then use the definition of $\mathcal{Q}$ given in Equation \eqref{def:Q} and observe that if $\nu_0^{\varepsilon}$ is chosen in an optimal way, as given in Proposition \ref{prop:approximation}, then:
\begin{align}
\mathbb{E} \left[ J(\bar{\nu}_K)-J(\mus) \right]  &\leq  R_0 \mathcal{Q}_{\mus,\nu_0}(\alpha K R_0) +\mathfrak{C} R_0^2 \left[ \eta \|\mus\|_{\mathrm{TV}} K + \alpha R_0\right], \label{eq:LemmeF1Chizat}\\
& \leq \mathfrak{C} \|\mus\|_{\mathrm{TV}} \left[\frac{d  \left(1+ \log\frac{\alpha K R_0}{2d}  + \frac{\log |\cX|}{d}\right)  }{\alpha K } + \frac{\alpha R_0^3}{\|\mus\|_{\mathrm{TV}}} + R_0^2 \eta K\right], \nonumber
\end{align}
where we have used the relationship $\|\nu_0^\varepsilon\|_{\mathrm{TV}}=\|\mu^\star\|_{\mathrm{TV}}$. The choices 
$$\alpha=\sqrt{\frac{d \|\mus\|_{\mathrm{TV}}}{R_0^3 K}} \quad \mathrm{and} \quad \eta = \sqrt{\frac{d R_0}{K^{3} \|\mus\|_{\mathrm{TV}}}}$$ 
then leads to
\[
\mathbb{E} \left[ J(\bar{\nu}_K)-J(\mus) \right]   \leq \mathfrak{C} \sqrt{\frac{d \|\mus\|_{\mathrm{TV} }R_0^3}{K}} \left[ \log (d \|\mus\|_{\mathrm{TV}} R_0^3 K) + \frac{\log(|\cX|)}{d}\right].
\]
\end{proof}
\subsection{Proof of Theorem \ref{theo:local_sto}}
 
\begin{proof}
The proof is slitted into three parts, and relies on a contraction argument with conditional expectation.  
In what follows, we will choose $\alpha$  such that $\alpha \mathcal{C}_1 (R_0+1)<1$ where $\mathcal{C}_1$ is involved in Lemma \ref{lem:boundJprime} and $R_0$ has been introduced in Proposition \ref{prop:TV_nuk}. We shall often use the inequality $|e^h-1| \leq 2 |h|$ which is valid when $|h|\leq 1$. We will frequently apply this inequality with $h=- \alpha \widehat{\mJ'_{\nu_k}}(\vt_j^k)$.

\medskip

\noindent 
\underline{Step 1: One-step evolution and second order term.}
Let $k\in \mathbb{N}^\star$ be fixed. According to Proposition \ref{prop:jnuprime}:
\begin{equation}
J(\nu_{k+1}) - J(\nu_k) = \int_\cX J_{\nu_k}' d(\nu_{k+1}-\nu_k) + \frac{1}{2} \| \Phi(\nu_{k+1}-\nu_k)\|_\mathds{H}^2.
\label{eq:dvp_J}
\end{equation}
Introducing the measure $\tilde \nu_{k+1} = \sum_{j=1}^p\omega_{j}^{k+1} \delta_{\bm t_j^k}$, we deduce that:
\begin{eqnarray*}
\| \Phi(\nu_{k+1}-\nu_k)\|_\mathds{H}^2
& = & \| \Phi(\nu_{k+1}-\tilde \nu_{k+1} + \tilde \nu_{k+1} - \nu_k)\|_\mathds{H}^2, \\
& \leq & 2 \| \Phi(\nu_{k+1}-\tilde \nu_{k+1})\|_\mathds{H}^2 + 2\| \Phi(\tilde\nu_{k+1}-\nu_k)\|_\mathds{H}^2
\end{eqnarray*}
First remark that, 
\begin{eqnarray*}
\|\Phi(\nu_{k+1}-\tilde \nu_{k+1})\|_\mathds{H}^2 
& = &  \left\| \sum_{j=1}^p \omega_j^{k+1} (\varphi_{\bm t_j^{k+1}}-\varphi_{\bm t_j^{k}})\right\|_\mathds{H}^2, \\
& \leq & \sum_{j=1}^p \omega_j^{k+1} \times \sum_{j=1}^p \omega_j^{k+1} \| \varphi_{\bm t_j^{k+1}} - \varphi_{\bm t_j^{k}} \|_\mathds{H}^2,\\
& \leq & Lip(\varphi) \| \nu_{k+1}\|_{\mathrm{TV}} \sum_{j=1}^p \omega_j^{k+1}\| \bm t_j^{k+1} - \bm t_j^{k} \|^2.
\end{eqnarray*}
According to \eqref{eq:update_pos_sto} and \eqref{eq:update_generalized_gradient}, we obtain that
\begin{align}
\|\Phi(\nu_{k+1}-\tilde \nu_{k+1})\|_\mathds{H}^2
&\leq  Lip(\varphi) \| \nu_{k+1}\|_{\mathrm{TV}} \ \eta^2 \sum_{j=1}^p {\omega_j^{k+1}} \|\widehat{\mD_{\nu_k}}(\vt_j^k)\|^2.
\label{eq:upper_bound_1}
\end{align}
In the same time, using (\ref{eq:up_w}),
\begin{eqnarray*}
\| \Phi(\tilde\nu_{k+1}-\nu_k)\|_\mathds{H}^2
& = & \left\| \sum_{j=1}^p (\omega_j^{k+1} - \omega_j^{k}) \varphi_{\bm t_j^k} \right\|_\mathds{H}^2, \\
& = & \left\| \sum_{j=1}^p \omega_j^k (e^{-\alpha \widehat{\mJ'_{\nu_k}}(\vt_j^k) }-1) \varphi_{\bm t_j^k} \right\|_\mathds{H}^2, \\
& \leq & \sum_{j=1}^p \omega_{j}^{k} \times \sum_{j=1}^p \omega_{j}^{k} (e^{-\alpha \widehat{\mJ'_{\nu_k}}(\vt_j^k) }-1)^2 \|\varphi_{\bm t_j^k} \|_\mathds{H}^2,
\end{eqnarray*}
where the last line comes from the Jensen inequality. We then observe from Lemma \ref{lem:boundJprime} that the terms $J_{\nu_k}'(\vt_j^k,Z_l^{k+1})$ are bounded. Hence, provided $\alpha \mathcal{C}_1 (R_0+1)<1$, 
a constant  $C_{\varphi}$ large enough exists such that:
\begin{equation}\label{eq:upper_bound_2}
\| \Phi(\tilde\nu_{k+1}-\nu_k)\|_\mathds{H}^2 \leq   C_{\varphi} \|\nu_k\|_{\mathrm{TV}}  \alpha^2    \sum_{j=1}^p \omega_j^k | \widehat{\mJ'_{\nu_k}}(\vt_j^k) |^2 .
\end{equation}
\noindent 
Gathering Equations \eqref{eq:upper_bound_1} and  \eqref{eq:upper_bound_2}, we then deduce that

\begin{equation}\label{eq:upper_bound_ordre_2}
\| \Phi(\nu_{k+1}-\nu_k)\|_\mathds{H}^2 \leq C_{\varphi}  \left( \eta^2\|\nu_{k+1}\|_{\mathrm{TV}} \sum_{j=1}^p {\omega_j^{k+1}} \|\widehat{\mD_{\nu_k}}(\vt_j^k)\|^2 + \alpha^2  \|\nu_k\|_{\mathrm{TV}}  \sum_{j=1}^p \omega_j^k | \widehat{\mJ'_{\nu_k}}(\vt_j^k) |^2 \right) .
\end{equation}

\noindent 
\underline{Step 2: Study of the drift first order term.}
We expand the first order term and observe that:
\begin{align*}
    \int_\cX J_{\nu_k}' d(\nu_{k+1}-\nu_k) & = \sum_{j=1}^p \left[(\omega_j^{k+1} - \omega_j^k) J_{\nu_k}'(\vt_j^k)  + \omega_j^k (J_{\nu_k}'(\vt_j^{k+1})-J_{\nu_k}'(\vt_j^{k}))\right]  \\
    & \quad +  \sum_{j=1}^p (\omega_j^{k+1}-\omega_j^{k})  (J_{\nu_k}'(\vt_j^{k+1})-J_{\nu_k}'(\vt_j^{k})), \\
    & =  \sum_{j=1}^p \left[(\omega_j^{k+1} - \omega_j^k) J_{\nu_k}'(\vt_j^k)  + \omega_j^k \langle \vt_j^{k+1} - \vt_j^{k} , \nabla J_{\nu_k}' (\vt_j^{k}) \rangle \right] \\ 
    & \quad +  \sum_{j=1}^p \left[\omega_j^k \langle \vt_j^{k+1} - \vt_j^{k} , \nabla^2 J_{\nu_k}' (\upsilon_j^{k})(\vt_j^{k+1} - \vt_j^{k}) \rangle  +
    (\omega_j^{k+1}-\omega_j^{k})  \langle \nabla J_{\nu_k}'(\tilde{\upsilon}_j^{k}), \vt_j^{k+1} -\vt_j^{k} \rangle\right],
\end{align*}
where $\upsilon_j^{k}$ and $\tilde{\upsilon}_j^{k}$ are some auxiliary points that belong to $(\mathbf{t}_j^k,\mathbf{t}_j^{k+1})$ obtained with the help of first and second order  Taylor expansions.
Using Proposition \ref{prop:bornes_techniques_J}, we deduce that:

\begin{align*}
    \int_\cX J_{\nu_k}' d(\nu_{k+1}-\nu_k) & \leq \sum_{j=1}^p \left[(\omega_j^{k+1} - \omega_j^k) J_{\nu_k}'(\vt_j^k)  + \omega_j^k \langle \vt_j^{k+1} - \vt_j^{k} , \nabla J_{\nu_k}' (\vt_j^{k}) \rangle\right] \\ 
& + \sum_{j=1}^p \omega_j^k {\| \nabla^2 J_{\nu_k}' \|_{\infty,op}} \|\vt_j^{k+1}-\vt_j^k\|^2 \\ 
& + \sum_{j=1}^p  |\omega_j^{k+1}-\omega_j^k| \times \| \nabla J_{\nu_k}'\| \|\vt_j^{k+1}-\vt_j^k\|\,.
\end{align*}
Using the total variation upper bound stated in Proposition \ref{prop:TV_nuk} by $R_0$, we then define for the sake of readability the constant $A$ as: 
\[
{A= (\|\nu\|_{\mathrm{TV}} \|\varphi\|_{\infty,\mathbb{H}}+\|y\|_\mathbb{H}) \| \nabla^2\varphi\|_{\infty,op} \vee  (\|\nu_k\|_{\mathrm{TV}} \|\varphi\|_{\infty,\mathbb{H}}+\|y\|_\mathbb{H}) \|\varphi'\|_\mathbb{H}.}
\]
We then derive:
\begin{align*}
    \int_\cX J_{\nu_k}' d(\nu_{k+1}-\nu_k) & \leq \sum_{j=1}^p (\omega_j^{k+1} - \omega_j^k) J_{\nu_k}'(\vt_j^k)  + \omega_j^k \langle \vt_j^{k+1} - \vt_j^{k} , \nabla J_{\nu_k}' (\vt_j^{k}) \rangle \\ &+ A  \sum_{j=1}^p \left[\omega_j^k  \|\vt_j^{k+1}-\vt_j^k\|^2+(\omega_j^{k+1}-\omega_j^k) \|\vt_j^{k+1}-\vt_j^k\|\right].
\end{align*}
We now use the surrogate update \eqref{eq:update_pos_sto} on the previous inequality and obtain that:
\begin{align*}
    \int_\cX J_{\nu_k}' d(\nu_{k+1}-\nu_k) & \leq  \sum_{j=1}^p   \omega_j^k (e^{-\alpha \widehat{\mJ'_{\nu_k}}(\vt_j^k) }-1) J_{\nu_k}'(\vt_j^k) { +   \omega_j^k \langle \vt_j^{k+1}-\vt_j^k,  \nabla J_{\nu_k}' (\vt_j^{k}) \rangle }\\
    & {+  A \sum_{j=1}^p   \omega_j^k \|\vt_j^{k+1}-\vt_j^k\|^2}  {+    \omega_j^k (e^{-\alpha \widehat{\mJ'_{\nu_k}}(\vt_j^k) }-1)  \|\vt_j^{k+1}-\vt_j^k\|}.
\end{align*}
{We pay a specific attention to the second term of the right hand side. Using the generalized projected gradient introduced in Section \ref{sec:fastpartpro} and in particular \eqref{eq:update_generalized_gradient}, we get for any $j\in \lbrace 1,\dots, p \rbrace$},
\begin{align*}
  \omega_j^k &\langle \vt_j^{k+1}-\vt_j^k,  \nabla J_{\nu_k}' (\vt_j^{k}) \rangle \\
  &= 
- \eta \omega_j^k \left\langle P_{\cX}(\vt_j^k,\widehat{\mD_{\nu_k}}(\vt_j^k) ,\eta) ,  \widehat{\mD_{\nu_k}}(\vt_j^k) \right\rangle  + \eta \omega_j^k \left\langle P_{\cX}(\vt_j^k,\widehat{\mD_{\nu_k}}(\vt_j^k) ,\eta)  , \widehat{\mD_{\nu_k}}(\vt_j^k)-\nabla J_{\nu_k}' (\vt_j^{k}) \right\rangle \\
& \leq - \eta \omega_j^k \left\|P_{\cX}(\vt_j^k,\widehat{\mD_{\nu_k}}(\vt_j^k) ,\eta) \right\|^2  + 
\eta \omega_j^k \left\langle P_{\cX}(\vt_j^k,\widehat{\mD_{\nu_k}}(\vt_j^k) ,\eta)  , \widehat{\mD_{\nu_k}}(\vt_j^k)-\nabla J_{\nu_k}' (\vt_j^{k}) \right\rangle \\
& = - \eta \omega_j^k \left\|P_{\cX}(\vt_j^k,\widehat{\mD_{\nu_k}}(\vt_j^k) ,\eta) \right\|^2  + 
\eta \omega_j^k \left\langle P_{\cX}(\vt_j^k,\nabla J_{\nu_k}' (\vt_j^{k}) ,\eta)  , \widehat{\mD_{\nu_k}}(\vt_j^k)-\nabla J_{\nu_k}' (\vt_j^{k}) \right\rangle \\
& \qquad + \eta \omega_j^k \left\langle P_{\cX}(\vt_j^k,\widehat{\mD_{\nu_k}}(\vt_j^k) ,\eta)  - P_{\cX}(\vt_j^k,\nabla J_{\nu_k}' (\vt_j^{k}) ,\eta)  , \widehat{\mD_{\nu_k}}(\vt_j^k)-\nabla J_{\nu_k}' (\vt_j^{k}) \right\rangle 
\end{align*} 
where for the first inequality, we have used Lemma \ref{lem:Ghadimi-Lan-Zhang} and simple algebraic manipulations hereafter. 
Using that $P_{\cX}$ is $1$-Lip (see Lemma \ref{lem:Ghadimi-Lan-Zhang2}) and the Cauchy-Schwarz inequality, we obtain
\begin{align*}
  \omega_j^k \langle \vt_j^{k+1} -\vt_j^k,  \nabla J_{\nu_k}' (\vt_j^{k}) \rangle 
  & \leq - \eta \omega_j^k \left\|P_{\cX}(\vt_j^k,\widehat{\mD_{\nu_k}}(\vt_j^k) ,\eta) \right\|^2   \\
  &  + \eta \omega_j^k \left\langle P_{\cX}(\vt_j^k,\nabla J_{\nu_k}' (\vt_j^{k}) ,\eta)  , \widehat{\mD_{\nu_k}}(\vt_j^k)-\nabla J_{\nu_k}' (\vt_j^{k}) \right\rangle + \eta \omega_j^k \left\| \widehat{\mD_{\nu_k}}(\vt_j^k)-\nabla J_{\nu_k}' (\vt_j^{k})\right\|^2.
\end{align*} 
Finally, for any $j\in \lbrace 1,\dots, p \rbrace$, we have
\begin{eqnarray*}
\mathbb{E}\left[\omega_j^k \langle \vt_j^{k+1} -\vt_j^k,  \nabla J_{\nu_k}' (\vt_j^{k}) \rangle \big\vert \mathfrak{F}_k \right] 
& \leq & - \eta \omega_j^k \mathbb{E} \left[ \left\|P_{\cX}(\vt_j^k,\widehat{\mD_{\nu_k}}(\vt_j^k) ,\eta) \right\|^2 \big\vert \mathfrak{F}_k \right] + \eta\omega_j^k \mathbb{E}\left[ \left\| \widehat{\mD_{\nu_k}}(\vt_j^k)-\nabla J_{\nu_k}' (\vt_j^{k})\right\|^2 \big\vert \mathfrak{F}_k\right]  , \\
& \leq & - \eta \omega_j^k \mathbb{E} \left[ \left\|\widehat{\mD_{\nu_k}}(\vt_j^k) \right\|^2 \big\vert \mathfrak{F}_k \right] + \eta\omega_j^k \mathbb{E}\left[ \left\| \widehat{\mD_{\nu_k}}(\vt_j^k)-\nabla J_{\nu_k}' (\vt_j^{k})\right\|^2 \big\vert \mathfrak{F}_k\right]  ,\\
& \leq &  - \eta \omega_j^k \mathbb{E} \left[ \left\|\nabla J_{\nu_k}' (\vt_j^{k}) \right\|^2 \big\vert \mathfrak{F}_k \right] + 2\eta\omega_j^k \mathbb{E}\left[ \left\| \widehat{\mD_{\nu_k}}(\vt_j^k)-\nabla J_{\nu_k}' (\vt_j^{k})\right\|^2 \big\vert \mathfrak{F}_k\right]  .
\end{eqnarray*}
At this step, we can take advantage of the mini-batch step. Indeed, according to \eqref{eq:minibatch}, we have
\begin{eqnarray*}
\mathbb{E}\left[ \left\| \widehat{\mD_{\nu_k}}(\vt_j^k)-\nabla J_{\nu_k}' (\vt_j^{k})\right\|^2 \big\vert \mathfrak{F}_k\right]    
& = & \frac{1}{m_k^2} \sum_{l=1}^{m_k} \mathbb{E} \left[ \left\| D_{\nu_k}(\vt_j^k,Z_l^k)\right\|^2 \big\vert \mathfrak{F}_k \right], \\
& \leq & \frac{\|\nu_k\|_{\mathrm{TV}}\| \mathbf{g}'\|_{\infty,\mathbb{H}}+ \| \mathbf{h}'\|_{\mathbb{H}}}{m_k}.
\end{eqnarray*}
Hence 
\[
\mathbb{E}\left[\omega_j^k \langle \vt_j^{k+1} -\vt_j^k,  \nabla J_{\nu_k}' (\vt_j^{k}) \rangle \big\vert \mathfrak{F}_k \right]  \leq 2\eta\omega_j^k \frac{\|\nu_k\|_{\mathrm{TV}}\| \mathbf{g}'\|_{\infty,\mathbb{H}}+ \| \mathbf{h}'\|_{\mathbb{H}}}{m_k}.
\]

We then consider the conditional expectation at time $k$ and apply Proposition \ref{prop:bornes_techniques_J} (to upper bound some rest terms) and Proposition \ref{prop:hoeffding} (to control the drift at iteration $k$). We deduce that a large enough $\mathfrak{C}$ such that:
\begin{eqnarray*}
\lefteqn{ \mathbb{E}  \left[ 
    \int_\cX J_{\nu_k}' d(\nu_{k+1}-\nu_k)\big\vert \mathfrak{F}_k \right]}\\
    & \leq & - \alpha \sum_{j=1}^p   \omega_j^k  J_{\nu_k}'(\vt_j^k)^2 + \mathfrak{C} \alpha^2  \sum_{j=1}^p   \omega_j^k  (1+\|\nuk\|_{\mathrm{TV}})^3   - \eta \sum_{j=1}^p  \omega_j^k \|\nabla J_{\nu_k}' (\vt_j^{k}) \|^2 \\
    & &+ \mathfrak{C}  \eta^2 \sum_{j=1}^p   \omega_j^k  (1+\|\nuk\|_{\mathrm{TV}})^2  + 
     \mathfrak{C} \eta \alpha \sum_{j=1}^p \omega_j^k  (1+\|\nuk\|_{\mathrm{TV}})^3
     {+ \mathfrak{C} \frac{\eta}{m_k} \sum_{j=1}^p \omega_j^k,} \\
    & \leq &  - \alpha \sum_{j=1}^p   \omega_j^k  J_{\nu_k}'(\vt_j^k)^2   - \eta \sum_{j=1}^p
    \omega_j^k \|\nabla J_{\nu_k}' (\vt_j^{k}) \|^2 \\
    & & +  \mathfrak{C} \|\nuk\|_{\mathrm{TV}}(1+\|\nuk\|_{\mathrm{TV}})^3 \left(\alpha^2+\eta^2 + {\frac{\eta}{m_k}}\right),
\end{eqnarray*}
where in the last line we used the Young inequality $2 \eta \alpha \leq  \alpha^2+  \eta^2$ and some rough upper bounds on the rest terms.
We now associate this last inequality with Equations \eqref{eq:upper_bound_ordre_2} and obtain the descent property:
\begin{equation}\label{eq:descente_J}
     \mathbb{E}  \left[ J(\nukp) \big\vert \mathfrak{F}_k \right]   \leq J(\nuk) - \alpha  \|J'_{\nuk}\|^2_{\nuk} - \eta  \|\nabla J'_{\nuk}\|^2_{\nuk} +
      \mathfrak{C} (1+R_0^4) \left(\alpha^2+\eta^2 {+ \frac{\eta}{m_k}}\right)
\end{equation}

\noindent 
\underline{Step 3: Conclusion of the proof.}
The rest of the proof proceeds with a standard argument. We use a telescopic sum + conditional expectation strategy and observe that:
$$
\alpha \sum_{k=1}^K \mathbb{E}[\|J'_{\nuk}\|^2_{\nuk}]  + \eta \sum_{k=1}^K \mathbb{E}[\|\nabla J'_{\nuk}\|^2_{\nuk}] \leq J(\nu_0) + \mathfrak{C} (1+R_0^4) \left(K \alpha^2+K \eta^2 {+ \eta\sum_{k=1}^K \frac{1}{m_k}}\right).
$$
Choosing $\alpha=\eta$, we deduce that:
$$
\frac{1}{K} \sum_{k=1}^K \left( \mathbb{E}[\|J'_{\nuk}\|^2_{\nuk}]  +  \mathbb{E}[\|\nabla J'_{\nuk}\|^2_{\nuk}]\right) \leq \frac{J(\nu_0)}{\alpha K} + \mathfrak{C} (1+R_0^4)  \alpha {+\frac{1}{K}\sum_{k=1}^K \frac{\mathfrak{C}}{m_k}}.
$$
Finally, if $\tau_K$ refers to a random variable uniformly distributed over $\{1,\ldots,K\}$, independent from the sequence $(\nuk)_{k \ge 1}$, the tuning $\alpha=\eta=1/\sqrt{K}$ {and $m_k = \sqrt{K}$ for all $k\in \lbrace 1,\dots, K \rbrace$}  yields:
$$
\mathbb{E}\left[ \|J'_{\nu_{\tau_K}}\|^2_{\nu_{\tau_K}} + \|\nabla J'_{\nu_{\tau_K}}\|^2_{\nu_{\tau_K}} \right]  \leq  \frac{J(\nu_0)+\mathfrak{C} (1+R_0^4)}{\sqrt{K}}.
$$
\end{proof}
\vfill

\paragraph{Acknowledgements} The authors would like to thank N. Jouvin for its remarks and time on the numerical aspects of this project. They are also in debt with L. Chizat for valuable discussions on CPGD during a seminar at Institut Henri Poincaré.  
\newpage

 \appendix

\bibliographystyle{plainnat}
\bibliography{biblio}

\newpage

\appendix

\begin{center}
    {\LARGE Appendix of ‘‘FastPart: Over-Parameterized Stochastic Gradient Descent for Sparse optimization on Measures"}
\end{center}

\bigskip

\section{Reformulations, proofs and technical lemmas \label{app:technical}}

\subsection{The non-separable case and the interesting reformulation of the objective}
There is a little subtlety on the properties that one should require on $\bbH$. At first glance, we need $\bbH$ separable to prove Bochner integrability in Lemma~\ref{lem:KME}. But $K$ is continuous on a compact space $\cX$, hence its RKHS is separable \citep[Lemma 4.3.3]{steinwart2008support}. This RKHS is isometric to a separable subspace of $\bbH$ as proven by the next lemma.

\begin{lemma}
    \label{lem:non_separable}
    Let $\bbH$ be Hilbert space and let $\cX$ be compact space. Under \eqref{ass:kernel_continuity}, there exists a separable cloded vector subspace $(\bbH_{\mathcal F},\|\cdot\|_\bbH)$ of $\bbH$ which is isometric to $({\mathcal F},\|\cdot\|_{\mathcal F})$, the RKHS defined by $K$. Denote by $\Pi$ the orthogonal projection onto $\bbH_{\mathcal F}$ then for any $\nu\in\measet$ 
    \begin{equation}
        \label{def:J_separable}
        J(\nu)=\frac12\|\vy - \Pi(\vy)\|_\bbH^2 + \frac12\|\Pi(\vy) - \Phi(\nu)\|_{\bbH_{\mathcal F}}^2+\lambda\|\nu\|_{\mathrm{TV}}\,.
    \end{equation}
\end{lemma}
\begin{proof}
We denote by $\mathcal F$ the RKHS defined by $K$. By \citep[Theorem 4.21]{steinwart2008support}, one has that
\[
\mathcal F:=
\Big\{
f:\cX\to\bbH\ :\ \exists h\in \bbH\,,\ 
\forall t\in\cX\,,\ f(x)=\langle h,\Phi(t)\rangle_{\bbH}
\Big\}\,.
\]
is the only RKHS defined by $K$ and 
\[
||f||_{\mathcal F}=\inf\Big\{
|| h ||_\bbH\ :\ h\in \bbH\ \ \text{s.t. }f(\cdot)=\langle h,\Phi(\cdot)\rangle_{\bbH}
\Big\}\,,
\]

Consider the functions 
\[
f_{h}\,:\, x\mapsto \langle h,\, \Phi(x)\rangle_\bbH\,,
\quad 
h = \sum_{j=1}^r\omega_j \Phi(t_j)
\]
defining the pre-dual of $\mathcal F$. Observe that 
\[
\langle f_{h_1},\, f_{h_2}\rangle_{\mathcal F}
=
\langle h_1,\, h_2\rangle_\bbH\,,
\]
where we denote by $\langle \cdot,\cdot\rangle_{\mathcal F}$, the dot product of $\mathcal F$. Define $\bbH_{\mathcal F}$ the vector subspace of $\bbH$ defined as the closure (in $\bbH$) of the span of $\Phi(\cX)$. The aforementioned equality shows that $h\mapsto f_{h}$ is an isometry from $(\bbH_{\mathcal F},\|\cdot\|_\bbH)$ onto $({\mathcal F},\|\cdot\|_{\mathcal F})$. Since ${\mathcal F}$ is separable \citep[Lemma 4.3.3]{steinwart2008support}, we deduce that $\bbH_{\mathcal F}$ is separable. The last statement is a consequence of the Pythagorean theorem.
\end{proof}

Note that in \eqref{def:J_separable}, the term $\frac12\|\vy - \Pi(\vy)\|_\bbH^2$ is constant. Hence, up to a constant term, and without loss of generality, one can assume that $\bbH$ is separable. 

\begin{remark}
    The proof of Lemma~\ref{lem:non_separable} is a consequence of \citep[Theorem~4.21]{steinwart2008support} and we choose to maintain it in this paper for sake of completeness. Moreover, it sheds light on an interesting reformulation of the quadratic term in \eqref{def:J}, the objective $J$. Indeed, it holds 
    \begin{equation}
        \label{eq:reformulation_J}
        J(\nu)=\frac12\Big\|(\mathbb{i}\circ\Pi)(\vy) - (\mathbb{i}\circ\Phi)(\nu)\Big\|_{{\mathcal F}}^2+\lambda\|\nu\|_{\mathrm{TV}}\,,
    \end{equation}
    up to a constant term and where $\mathbb{i}$ denotes the isometry between $(\bbH_{\mathcal F},\|\cdot\|_\bbH)$ and $({\mathcal F},\|\cdot\|_{\mathcal F})$.
\end{remark}

\subsection{Existence of the kernel measure embedding}
\label{app:kme}
Kernel mean embedding is a standard notion in Machine Learning, see for instance \cite{muandet2017kernel}. Extending this notion of measure with finite total variation norm is straightforward. We referred to this notion as {\it Kernel Measure Embedding} as the two notions coincides on probability measures. 
\begin{lemma}
    \label{lem:KME}
    Let $\bbH$ be separable Hilbert space and let $\cX$ be compact metric space. Under \eqref{ass:kernel_continuity}, the operator~$\Phi$ defined by~\eqref{def:Phi} is well defined and bounded linear as a function from $\measet$ to $\bbH$.
    Furthermore, the dual of $\Phi$ is given by
    \begin{equation}
        \label{eq:dual_Phi}
        \Phi^\star\,:\,h\in\bbH\mapsto \big(t\mapsto \langle h,\varphi_t\rangle_\bbH\big)\in\big(\mathcal{C}(\cX),\|\cdot\|_\infty\big)\,,
    \end{equation}
    and for any $(h,\nu)\in\bbH\times\measet$, 
    \begin{equation}
        \label{eq:fubini}
    \langle h,\Phi(\nu)\rangle_\bbH
    =\int_\cX\langle h, \varphi_t\rangle_\bbH\mathrm{d}\nu(t)
    =\langle\Phi^\star(h),\nu\rangle_{\mathcal{C}(\cX),\measet}
    \leq \sup_t\sqrt{\Kernel(t,t)} \|h\|_\bbH \|\nu\|_{\mathrm{TV}}
    \,.
    \end{equation}
\end{lemma}
\begin{remark}
\label{rem:weak_weak}
    A key result of Lemma~\ref{lem:KME} is that $\mathrm{Im}(\Phi^\star)\subseteq\big(\mathcal{C}(\cX),\|\cdot\|_\infty\big)$, this latter being a subset of~$\measet^\star$, the topological dual of $\measet$. {Strictly speaking, the dual of~$\Phi$ maps to the dual of the space of measures and the right manner to expose this results is as it is done by \cite{Bredis_Pikkarainen_13}, using the predual operator~$\Phi_\star$ which satisfies~$(\Phi_\star)^\star=\Phi$ and~$\Phi^\star= \imath \Phi_\star$ where~$\imath$ denotes the canonical embedding of continuous functions $(\mathcal{C}(\cX),\|\cdot\|_\infty\big)$ into the dual of measures.}
\end{remark}
\begin{proof}
Let $\nu\in\measet$. We say that $t\in\cX\mapsto f(x)\in\bbH$ is {\it simple} if it is finitely valued, namely
\[
f(t)=\sum_{i=1}^nh_i\mathbf{1}_{\{t\in B_i\}}\,,
\]
for some $n\geq 1$, $h_i\in\bbH$, and $B_i$ Borel set of $\cX$. In this case, one has
\[
\int_\cX f\mathrm{d}\nu=\sum_{i=1}^nh_i\nu(B_i)\,.
\]

Note that $\|\varphi_t\|_\bbH=\sqrt{\Kernel(t,t)}$ and, it holds that
\begin{equation}
    \label{eq:finite_H_norm}
    \int_\cX \|\varphi_t\|_\bbH\mathrm{d}\nu(t)\leq \sup_t\sqrt{\Kernel(t,t)}\|\nu\|_{\mathrm{TV}}<\infty\,,
\end{equation}
using the fact that $t\in\cX\mapsto\sqrt{\Kernel(t,t)}$ is a bounded continuous function by \eqref{ass:kernel_continuity}. We emphasize that this function not need to be vanishing at infinity. 

From \eqref{eq:finite_H_norm}, we deduce that the map $m\,:\, A\mapsto m(A):=\int_\cX \|\varphi_t\|_\bbH\mathbf{1}_{\{\varphi_t\in A\}}\mathrm{d}\nu(t)$ is a finite measure on the Borel sets of $\bbH$ and hence, by Oxtoby-Ulam theorem (see \cite[Proposition 2.1.4]{gine2021mathematical} for instance), a tight Borel measure. Given $0 < \varepsilon_n \to0$, let $K_n$ be a compact set such that $m(K^c_n)<\varepsilon_n/2$, let $A_{n,1},\ldots,A_{n,{k_n}}$ be a finite partition of $K_n$ consisting of sets of diameter at most $\varepsilon_n/2$, pick up a point $h_{n,k} \in A_{n,{k}}$ for each $k$ and define the simple function 
\[
f_n(t) = \sum_{k=1}^{k_n} h_{n,k} \mathbf{1}_{\{\varphi_t\in A_{n,k}\}}\,.
\]
Then 
\[
\int_\cX \|\varphi_t-f_n(t)\|_\bbH\mathrm{d}\nu(t)\leq \varepsilon_n/2+m(K^c_n)<\varepsilon_n\to0\,,
\]
showing that $t\in\cX\mapsto\varphi_t\in\bbH$ is Bochner integrable, hence Petti's integrable, and both integrals coincide (see for instance \cite[Section 2.6.1]{gine2021mathematical}). We deduce that $\Phi$ is well defined, using Bochner integration. Furthermore, one can deduce that 
\[
\Big\|\int_\cX \varphi_t\mathrm{d}\nu(t)\Big\|_\bbH\leq \int_\cX \|\varphi_t\|_\bbH\mathrm{d}\nu(t)\leq \sup_t\sqrt{\Kernel(t,t)}\|\nu\|_{\mathrm{TV}}\,,
\]
showing that $\Phi$ is bounded linear.

Also, if $h\in\bbH$ then 
\[
\big|\int_\cX\langle h, f_n(t)\rangle_\bbH -\langle h, \varphi_t\rangle_\bbH\mathrm{d}\nu(t)\big|
\leq \|h\|_\bbH\int_\cX \|\varphi_t-f_n(t)\|_\bbH\mathrm{d}\nu(t)\to0\,.
\]
Hence, $\displaystyle\int_\cX\langle h, \varphi_t\rangle_\bbH\mathrm{d}\nu(t)=\lim_n\int_\cX\langle h, f_n(t)\rangle_\bbH \mathrm{d}\nu(t)$ exists and is finite. We deduce that 
\begin{equation}
    \label{eq:fubini_H}
\int_\cX\langle h, \varphi_t\rangle_\bbH\mathrm{d}\nu(t)=\langle h,\Phi(\nu)\rangle_\bbH \,,
\end{equation}
using that $\int_\cX\langle h, f_n(t)\rangle_\bbH \mathrm{d}\nu(t)=\langle h,\int_\cX f_n \mathrm{d}\nu\rangle_\bbH$.

Using \eqref{eq:fubini_H} and Cauchy-Schwarz inequality, one gets that
\begin{equation}
    \label{eq:duality_Phi}
    \langle h,\Phi(\nu)\rangle_\bbH=\int_\cX\langle h, \varphi_t\rangle_\bbH\mathrm{d}\nu(t)
\leq \sup_t\sqrt{\Kernel(t,t)} \|h\|_\bbH \|\nu\|_{\mathrm{TV}}\,,
\end{equation}
and hence, we can write 
\[
\langle h,\Phi(\nu)\rangle_\bbH
=\langle\langle h, \varphi_t\rangle_\bbH,\nu\rangle_{\measet^\star,\measet}
\]
where $\measet^\star$ is the topological dual of $\measet$. It shows that the dual $\Phi^\star$ is given by $
\Phi^\star(h)(t)=\langle h, \varphi_t\rangle_\bbH$. As a function of $t$, it is clear that it is continuous by \eqref{ass:kernel_continuity} and that 
$\|\Phi^\star(h)\|_\infty\leq\sup_t\sqrt{\Kernel(t,t)} \|h\|_\bbH<\infty $, showing that it belongs to the space of bounded continuous functions. 
\end{proof}

\subsection{Proof of Theorem~\ref{theo:mu_star}}
\label{app:thm_existence}
Let $(\nu_n)$ be a minimizing sequence of measures of Program \eqref{eq:general_min}. Up to an extraction we can consider that $L(\Phi(\nu_n))+\lambda\|\nu_n\|_{\mathrm{TV}}\leq 1+\inf_\nu \{L(\Phi(\nu))+\lambda\|\nu\|_{\mathrm{TV}}\}$. In particular, it holds that 
\[
\lambda\|\nu_n\|_{\mathrm{TV}}\leq 1+\inf_\nu \{L(\Phi(\nu))+\lambda\|\nu\|_{\mathrm{TV}}\}\,.
\]
Up to an extraction, by Banach-Alaoglu theorem, we can consider that the sequence $(\nu_n)$ converges for the weak-$\star$ topology. We denote by $\mu^\star\in\measet$ its limit. Using \cite[Proposition 3.13(iii)]{brezis2011functional}, the $\mathrm{TV}$-norm is l.s.c. for the weak-$\star$ topology, and we get that
\[
\lim\inf_n \|\nu_n\|_{\mathrm{TV}}\geq \|\mu^\star\|_{\mathrm{TV}}\,.
\]
Using Lemma~\ref{lem:KME}, it holds that for any $h\in\bbH$ and for any convergent sequence $\nu_n\to\mu^\star$ for the weak-$\star$ topology, 
\[
    \langle h,\Phi(\nu_n)\rangle_\bbH
    =\langle\Phi^\star(h),\nu_n\rangle_{\mathcal C(\cX),\measet}
    \to \langle\Phi^\star(h),\mu^\star\rangle_{\mathcal C(\cX),\measet}\,,
\]
proving that $\Phi$ is continuous from $\measet$ weak-$\star$ to $\bbH$ weak (see Remark~\ref{rem:weak_weak}). Since $L$ is l.s.c for the weak topology of $\bbH$ \citep[Corollary 3.9]{brezis2011functional}, we get that 
\[
\lim\inf_n L(\Phi(\nu_n))\geq L(\Phi(\mu^\star))\,.
\]
Combining the aforementioned limits, we deduce that 
\[
\eqref{eq:general_min}=
\lim\inf_n \big\{L(\Phi(\nu_n))+\lambda\|\nu_n\|_{\mathrm{TV}}\big\}\geq L(\Phi(\mu^\star))+\lambda\|\mu^\star\|_{\mathrm{TV}}\geq \eqref{eq:general_min}\,,
\]
hence equality. The uniqueness of $\Phi(\mu^\star)$ follows by strict convexity.

\begin{remark}
   In this paper, we assume that $\cX$ is compact. Some of our bounds depend on the size of $\cX$ and do not hold for non-compact spaces. But, the existence of $\mu^\star$ can be proven in the non-compact case. 
   
   The subtlety is in \eqref{eq:dual_Phi}. To get the proof of Theorem~\ref{theo:mu_star} work when $\cX$ is a Polish space (not necessarily compact), one needs that $\mathrm{Im}(\Phi^\star)\in \big(\mathcal{C}_0(\cX),\|\cdot\|_\infty\big)$, the space of continuous functions vanishing at infinity. We already know that $\mathrm{Im}(\Phi^\star)\in \big(\mathcal{C}(\cX),\|\cdot\|_\infty\big)$ by Condition~\eqref{ass:kernel_continuity}. We have the following result.
\begin{theorem}
    Let $\bbH$ be Hilbert space and let $\cX$ be Polish space. Assume that 
   \begin{itemize}
       \item Assumption \eqref{ass:kernel_continuity} holds;
       \item the RKHS~$\mathcal F$ (defined by $\Kernel$) is contained in $\mathcal{C}_0(\cX)$;
       \item and $\sup_t\sqrt{\Kernel(t,t)}< \infty$;
   \end{itemize}   
   then the there exists a measure $\mu^\star\in\measet$ such that 
        \begin{equation}
        \notag
            J(\mu^\star) = \min_{\mu \in \measet } J(\mu)\,.
        \end{equation}
    Furthermore, the vector $\Phi(\mu^\star)\in\bbH$ is unique.
\end{theorem}
\end{remark}

\begin{remark}
    \label{rem:proof_mu_star_+}
    The same argument can be used to prove that Program~\eqref{eq:general_min} restricted to $\measet_+$ admits solutions. Indeed, take $(\nu_n)$ a sequence of nonnegative measures such that the objective converges towards the infimum. We can use the above proof to show the existence of $\mus$. The only point left to prove is that the measure $\mus$ is nonnegative, which is straight forward using weak-$star$ convergence and Riesz representation theorem \citep[Chapter 2]{walter1974real} of nonnegative linear functional defined by nonnegative continuous functions with compact support $($which are included in $\cC_0(\cX))$.
\end{remark}

\begin{remark}
    A similar result can be found in \cite[Proposition 3.1]{chizat2022sparse} using Prokorov's theorem.
\end{remark}

\section{Gradients of the objective \label{sec:gradients}}

\subsection{In the space of (nonnegative) measures}

We first consider the variation of $J$ in $\measet_+$ in terms of its Fréchet differential. 
\begin{proposition}
\label{prop:jnuprime}
    If $\nu+\sigma\in\measet_+$ and $\nu\in\measet_+$ then 
    \begin{equation}
        \label{eq:diff_J_nonnegative_app}
        J(\nu+\sigma)-J(\nu) = \int_\cX J'_\nu\mathrm d\sigma + \frac{1}{2}\|\Phi(\sigma)\|_\bbH^2\,,
    \end{equation}
    where $J'_\nu:=\Phi^\star(\Phi(\nu)-\vy)+\lambda$ and $\Phi^\star:\,(\bbH,\|\cdot\|_\bbH)\to(\cC(\cX),\|\cdot\|_\infty)$ is the dual of $\Phi$.
\end{proposition}

\begin{proof}[Proof of \ref{prop:jnuprime}]
The proof follows from the expansion of $J(\nu+\sigma)$:
\begin{align*}
J(\nu+\sigma)&=\frac{1}{2} \big\| \vy-\Phi(\nu+\sigma)\big\|_\bbH^2 +  \lambda \|\mu+\sigma\|_{\mathrm{TV}} \\
& = \frac{1}{2} \big\| \vy-\Phi(\nu) - \Phi(\sigma)\big\|_\bbH^2 +  \lambda \|\nu+\sigma\|_{\mathrm{TV}}  \\
& = \frac{1}{2} \big\| \vy-\Phi (\nu) \big\|_\bbH^2 - \langle \vy-\Phi (\nu) , \Phi (\sigma) \rangle_\bbH+\frac{1}{2} \big\| \Phi (\sigma)\big\|_\bbH^2 
+  \lambda \|\nu+\sigma\|_{\mathrm{TV}}  \\
& = J(\nu) - \langle \Phi^\star(\vy-\Phi(\nu)) , \sigma \rangle_\bbH+\frac{1}{2} \big\| \Phi(\sigma)\big\|_\bbH^2 + \lambda \left[ \|\nu+\sigma\|_{\mathrm{TV}}  - \|\nu\|_{\mathrm{TV}} \right].
\end{align*}
Using $\mathrm{Sign}(\nu)$ as a subgradient of the TV-norm at point~$\nu$ ($\nu$-almost everywhere equal to the sign of $\nu$ and with infinity norm less than one), we then observe that:
\[
\|\nu+\sigma\|_{\mathrm{TV}}  - \|\nu\|_{\mathrm{TV}} = \langle \mathrm{Sign}(\nu),\sigma \rangle_{\measet^\star,\measet} + \cD_\nu(\sigma),
\]
where  $\cD_\nu(\sigma)$ is the second order Bregman divergence of the TV-norm between $\nu$ and $\nu+\sigma$ using the subgradient $\mathrm{Sign}(\nu)$, given by:
\[
\cD_\nu(\sigma):=\|\nu+\sigma\|_{\mathrm{TV}}  - \|\nu\|_{\mathrm{TV}}-\langle \mathrm{Sign}(\nu),\sigma \rangle_{\measet^\star,\measet}\,,
\]
with $\measet^\star\subseteq (L^\infty(\cX),\|\cdot\|_\infty)$ the topological dual of $\measet$.
Gathering all the pieces, we obtain that:
\begin{equation}
    \label{eq:diff_J}
    J(\nu+\sigma)-J(\nu) = \langle J'_\nu,\sigma \rangle_{\measet^\star,\measet}
    +q(\sigma)\,, 
\end{equation}
where $J'_{\nu}$ is given in the statement of Proposition \ref{prop:jnuprime} and $q$ is a second order term given by:
\[
    q(\sigma):=\frac{1}{2}\|\Phi(\sigma)\|_\bbH^2+\lambda\cD_\nu(\sigma).
\]

Finally, we remark that when $\nu$ is nonnegative, one possible choice for the TV subgradient is $\mathrm{Sign}(\nu)=1$. In this case, the previous decomposition may be simplified as:
\[
 J'_{\nu} =\Phi^\star(\Phi(\nu)-\vy)+\lambda\]
and $\cD_\nu(\sigma)=0$ when  $\nu+\sigma\in\measet_+$ and $\nu\in\measet_+$.
\end{proof}

\begin{remark}
    The above proof shows that for any $\mu,\nu\in\measet$,
    \begin{equation}
        \notag
        J(\nu+\sigma)-J(\nu) = \int_\cX J'_\nu\mathrm d\sigma + \frac{1}{2}\|\Phi(\sigma)\|_\bbH^2+\lambda\cD_\nu(\sigma)\,,
    \end{equation}
    where $\cD_\nu(\sigma):=\|\nu+\sigma\|_{\mathrm{TV}}  - \|\nu\|_{\mathrm{TV}}-\langle \mathrm{Sign}(\nu),\sigma \rangle_{\measet^\star,\measet}$ and $\measet^\star\subseteq (L^\infty(\cX),\|\cdot\|_\infty)$ the topological dual of $\measet$.
\end{remark}

Besides the expression of the Fréchet differential of $J$ on the space  $\measet_+$, it is possible to explicit the value of $J'_{\nu}$ at any point $t \in \cX$. Using Lemma \ref{lem:KME}, it holds, 
\begin{equation}
J'_{\nu}(t)=\big\langle \varphi_t,\Phi(\nu)-\vy\big\rangle_\bbH+\lambda, \quad \forall t\in \cX.
\label{eq:jnu1}
\end{equation}
Note that $J'_\nu$ depends on $\nu$ through $\Phi(\nu)$. Recall that $\Phi(\mus)$ is constant across all solutions $\mus$ of Program~\eqref{def:Blasso+}, see Theorem~\ref{theo:mu_star}. Hence, the function
\[
J'_{\star}(x):=\big\langle \varphi_x,\Phi(\mus)-\vy\big\rangle_\bbH+\lambda\,,\quad
x\in\cX\,,
\]
is well defined and does not depend on the choice of the solution $\mus$ (it is the same function across all possible choice of $\mus$ solution to \eqref{def:Blasso+}). The next proposition gives the first order condition of Program~\eqref{def:Blasso+}.

\begin{proposition}
    \label{prop:first_order}
    It holds that $J'_{\star} \ge 0$ and, for any solution $\mu^\star$ to Program \eqref{def:Blasso+},  
    \begin{equation}
    \label{eq:support_condition}
        \mathrm{Supp}(\mus)\subseteq\big\{x\in\cX\,:\,J'_{\star}(x)=0\big\}\,.
    \end{equation}
    Conversely, if a measure $\nu\in\measet_+$ is such that $J'_\nu\geq 0$ and it satisfies the condition 
        \begin{equation}
    \notag
        \mathrm{Supp}(\nu)\subseteq\big\{x\in\cX\,:\,J'_{\nu}(x)=0\big\}\,,
    \end{equation}
    then~$\nu$ is a solution to Program~\eqref{def:Blasso+} and $J'_\star=J'_\nu$.
\end{proposition}

\begin{proof}
Let $x\in\cX$ and let $\varepsilon>0$ be defined later. By Proposition~\ref{prop:jnuprime}, one has 
\[
J'_\star(x)=\frac{J(\mus+\varepsilon\delta_x)-J(\mus)}{\varepsilon}-\frac{\varepsilon}{2}\|\varphi_x\|^2_\bbH\,,
\]
hence 
\[
J'_\star(x)\geq\liminf_{\varepsilon\downarrow 0}\Bigg\{\frac{J(\mus+\varepsilon\delta_x)-J(\mus)}{\varepsilon}\Bigg\}\geq 0\,,
\]
since $J(\mus+\varepsilon\delta_x)-J(\mus)\geq 0$. 

Assume now that there exists a point $x\in\cX$ such that $x\in\mathrm{Supp}(\mus)$ and $J'_\star(x)>0$. Since $J'_\star$ is continuous and $J'_\star(x)>0$ there exists $\varepsilon>0$ and a open neighborhood $U_x$ of $x$ such that 
\[
\forall t\in U_x\,,\quad J'_x(t)>\sqrt\varepsilon
\]
By Jordan decomposition theorem, there exists two nonnegative measures $\mus_+$ and $\mus_{-}$ with disjoints supports such that $\mus=\mus_+-\mus_{-}$ and $\mathrm{Supp}(\mus)=\mathrm{Supp}(\mus_+)\sqcup\mathrm{Supp}(\mus_{-})$. Without loss of generality, we assume that $x\in\mathrm{Supp}(\mus_+)$. Taking $\varepsilon>0$ and $U_x$ sufficiently smalls, one has $U_x\cap \mathrm{Supp}(\mus_{-})=\emptyset$ and
\[
\mus(U_x)>\sqrt\varepsilon\,.
\]
Let $B$ be a Borelian of $\cX$ and define $\sigma\in\measet$ by $\sigma(B):=-\mus_+(B\cap U_x)$. Remark that $\mathcal{D}_{\mus}(\sigma)=0$, this latter being straightforward when $\mus\in\measet_+$ (in this case $\mus+\sigma\in\measet_+$). By Proposition~\ref{prop:jnuprime}, one has 
\[
0\leq J(\mus+\sigma)-J(\mus)=\int_\cX J'_\star\mathrm{d}\sigma=-\int_{\cX\cap U_x}J'_\star\mathrm{d}\mus_+\,,
\]
and also
\[
\int_{\cX\cap U_x}J'_\star\mathrm{d}\mus_+\geq \varepsilon\,,
\]
which is a contradiction. The converse result is a consequence of \eqref{eq:diff_J_nonnegative_app} as
\[
J(\mu^\star)-J(\nu)=\int_\cX J'_\nu\mathrm d(\mu^\star-\nu) + \frac{1}{2}\|\Phi(\mu^\star-\nu)\|_\bbH^2
=\int_\cX J'_\nu\mathrm d\mu^\star + \frac{1}{2}\|\Phi(\mu^\star-\nu)\|_\bbH^2\geq 0
\]
and hence $\nu$ is a minimizer. Finally, $J'_\star$  does not depend on the choice of the solution $\mu^\star$ hence $J'_\star=J'_\nu$.
\end{proof}

The lemma displayed below provides some bounds on the Frechet differential of the objective function and on its stochastic estimate. 

\begin{lemma}
\label{lem:boundJprime}
There exists a positive constant $\mathcal{C}_0 = \mathcal{C}_0(y,\varphi,\lambda)$ such that 
$$ \| J_\nu'\|_\infty := \sup_{t\in\cX} |J'_\nu(t) | \leq \mathcal{C}_0 ( \|\nu\|_{\mathrm{TV}}+1) \quad \forall \nu\in \mathcal{M}(\cX)_+.$$
Moreover, provided Assumption \eqref{A1} is satisfied, we have almost surely for any $\nu\in \mathcal{M}(\cX)_+$
$$ \sup_{t\in\cX} |J'_\nu(t,Z) | \leq \mathcal{C}_1 ( \|\nu\|_{\mathrm{TV}}+1) \quad \mathrm{and} \quad \sup_{t\in\cX} |\xi_\nu(t,Z) | \leq \mathcal{C}_2 ( \|\nu\|_{\mathrm{TV}}+1),$$
for some constants $\mathcal{C}_1$ and $\mathcal{C}_2$ depending only on $y,\varphi$ and $\lambda$.
\end{lemma}
\begin{proof}
Let $\nu\in \mathcal{M}(\cX)_+$ be fixed. We denote by $\tilde \nu$ the normalized measure $\tilde\nu=\nu/\|\nu\|_{\mathrm{TV}} \in \mathcal{M}(\cX)_+$. According to (\ref{eq:jnu1}), we have
\begin{eqnarray*}
\sup_{t\in\cX}|J'_\nu(t)| 
& \leq & \lambda + \sup_{t\in\cX}|\langle \varphi_t,y\rangle_\mathds{H}| + \|\nu\|_{\mathrm{TV}} \sup_{t\in\cX}|\langle \varphi_t, \Phi(\tilde\nu)\rangle_\mathds{H}|, \\
& \leq & \lambda + \sup_{t\in\cX}|\langle \varphi_t,y\rangle_\mathds{H}| + \|\nu\|_{\mathrm{TV}} \sup_{t,s\in\cX}|\langle \varphi_t, \varphi_s \rangle_\mathds{H}|, \\
& \leq & \lambda + \|\varphi \|_{\infty,\mathds{H}} \|y\|_\mathds{H} + \|\nu\|_{\mathrm{TV}} \|\varphi\|_{\infty,\mathds{H}}^2,\\
& \leq & \mathcal{C}_0 (\|\nu\|_{\mathrm{TV}} + 1),
\end{eqnarray*}
with 
\begin{equation}
\mathcal{C}_0 = \max\left( \lambda +\|\varphi \|_{\infty,\mathds{H}} \|y\|_\mathds{H} ; \|\varphi\|_{\infty,\mathds{H}}^2 \right).
\label{eq:C0}
\end{equation}
Concerning the second part of the lemma, we first remark that, provided Assumption \eqref{A1} is satisfied, we have for any $\nu\in \mathcal{M}(\cX)_+$
$$ \sup_{t\in\cX} |J'_\nu(t,Z) | \leq \|\nu\|_{\mathrm{TV}} \sup_{t\in\cX} |\vg_{t,T}(U)| + \sup_{t\in\cX}|\vh_t(V)| + \lambda \leq \mathcal{C}_1(\|\nu\|_{\mathrm{TV}}+1),$$
with
\begin{equation}
\mathcal{C}_1 := \max\left(\|\vg\|_\infty ; \|\vh\|_\infty+\lambda \right).
\label{eq:C1}
\end{equation}
The last results is obtained thanks to a basic triangle inequality
$$ \sup_{t\in\cX} |\xi_\nu(t,Z) | \leq   \sup_{t\in\cX} |J'_\nu(t) | +  \sup_{t\in\cX} |J'_\nu(t,Z) | \leq \mathcal{C}_2 ( \|\nu\|_{\mathrm{TV}}+1),$$
with
\begin{equation}
\mathcal{C}_2=\mathcal{C}_0 + \mathcal{C}_1 = \max\left( \lambda +\|\varphi \|_{\infty,\mathds{H}} \|y\|_\mathds{H} ; \|\varphi\|_{\infty,\mathds{H}}^2 \right) + \max\left(\|\vg\|_\infty ; \|\vh\|_\infty+\lambda \right).
\label{eq:C2}
\end{equation}
\end{proof}

\subsection{In the space of particles}
We consider any set of positions $\pos$ and their associate weights $\weights$.
In order to compute the derivatives of~$F$ w.r.t. $\weights$ and $\pos$, our starting point is Equation \eqref{def:F} and we observe that the gradient with respect to $\weights$ is easily computed:
\[
    \nabla_{\bm  \omega} F(\bm  \omega,\pos)  = \bm\lambda-\vk_{\pos}+\Kernel_\pos\weights.
\]
Nevertheless, the interpretation in terms of Fréchet derivative and of $J$ allows to obtain the next result.
\begin{proposition}\label{prop:grad_F}
For any $\weights$ and $\pos$, denote $\nu=\nu(\weights,\pos)$, one has:
\begin{itemize}%
\item[$(i)$] Gradient w.r.t. weights: for any $j\in \{1,\ldots,p\}$, one has $\nabla_{\omega_j} F(\bm  \omega,\pos) = J'_{\nu}(\vt_j)$
\item[$(ii)$] Gradient w.r.t. positions: for any $j\in \{1,\ldots,p\}$, one has $\nabla_{\vt_j} F(\bm  \omega,\pos) = \omega_j \nabla_{\vt_j} J'_{\nu}(\vt_j)$
\end{itemize}
\end{proposition}

\begin{proof}
The starting point is the Fréchet derivative that, if we consider $\sigma \in \measet_+$ and $\varepsilon>0$ small enough:
\[
J(\nu+\varepsilon \sigma) = J(\nu) + \varepsilon \langle J'_{\nu},\sigma \rangle_\bbH + o(\varepsilon).
\]
\underline{Proof of $(i)$:}
Considering any particle $j \in \{1,\ldots, p\}$ and $\sigma = \delta_{\vt_j}$, we then obtain that:
\[\lim_{\varepsilon \to 0} \frac{J(\nu+\varepsilon \sigma)-J(\nu)}{\varepsilon} = \langle  J'_{\nu},\delta_{\vt_j} \rangle_\bbH  = J'_{\nu}(\vt_j),\]
where the last equality comes from the reproducing kernel property. In the meantime, we observe that
\[ \lim_{\varepsilon \to 0} \frac{J(\nu+\varepsilon \sigma)-J(\nu)}{\varepsilon} = \lim_{\varepsilon \to 0} \frac{F(\weights+\varepsilon \bm 1_{j},\pos) - F(\weights,\pos)}{\varepsilon}  = \frac{ \partial F(\weights,\pos)}{\partial \omega_j}.\]
We then conclude using the Fréchet derivative of $J$ that:
\[
J'_{\nu}(\vt_j) =  \nabla_{\omega_j} F(\bm  \omega,\pos).
\]
\underline{Proof of $(ii)$:}
Using the same consideration on the positions of the particles, we then consider any pertubed set of positions $\tilde{\pos}_{\varepsilon,j} = \pos + \varepsilon \bm 1_j$ where only the coordinate $j$ of $\pos$ is modified. We then write the partial derivative of $F$:
\[
\lim_{\varepsilon \to  0} \frac{F(\weights,\tilde{\pos}_{\varepsilon,j})-F(\weights,\pos)}{\varepsilon}= \frac{\partial F(\weights,\pos)}{\partial t_j}.
\]
In the meantime, we observe that with the Fréchet derivative of $J$ that:
\begin{align*}
\lim_{\varepsilon \to  0} \frac{F(\weights,\tilde{\pos}_{\varepsilon,j})-F(\weights,\pos)}{\varepsilon} &=   \lim_{\varepsilon \to  0} \frac{J(\nu(\weights,\tilde{\pos}_{\varepsilon,j}))-J(\nu(\weights,\pos))}{\varepsilon} \\
& =
 \lim_{\varepsilon \to  0} \frac{J(\nu(\weights,\pos)+ \omega_j [\delta_{\vt_j+\varepsilon}-\delta_{\vt_j}])-J(\nu(\weights,\pos))}{\varepsilon}\\
& =  \lim_{\varepsilon \to  0}
\frac{ \langle J'_{\nu(\weights,\pos)},\omega_j \delta_{\vt_j+\varepsilon}\rangle_\bbH -  \langle J'_{\nu(\weights,\pos)},\omega_j \delta_{\vt_j}\rangle_\bbH}{\varepsilon} \\
& = \omega_j \nabla_{t_j} J'_{\nu(\weights,\pos)}(\vt_j).
\end{align*}
\end{proof}
Finally, it is possible to quantify the way $F$ is modified when we change $(\weights,\pos)$ to $(\weights',\pos')$  thanks to the next proposition.
where $\nu=\sum_{k=1}^p \omega_k \delta_{t_k}$.
\begin{proposition}
Consider two pairs $(\weights,\pos)$ and $(\weights',\pos')$ of weights/positions and denote $\nu=\nu(\weights,\pos)$ defined in Equation \eqref{def:nu}, then:
    \begin{equation}
        \label{eq:diff_F_nonnegative}
        F(\bm  \omega',\pos')-F(\bm  \omega,\pos) = \sum_{j=1}^p (\omega_j' J'_\nu(\vt_j')-\omega_j J'_\nu(\vt_j)) + \frac{1}{2}(\bm  \omega',-\bm  \omega)^\top\Kernel_{(\pos',\pos)}(\bm  \omega',-\bm  \omega)\,,
    \end{equation}
where  $\Kernel_{(\pos',\pos)}$ is a $(2p\times 2p)$ symmetric matrix with $(p\times p)$ diagonal blocks~$\langle \varphi_{t_i'},\varphi_{t_j'}\rangle_{\bbH}$ and $\langle \varphi_{t_i},\varphi_{t_j}\rangle_{\bbH}$, and $(p\times p)$ off-diagonal block $\langle \varphi_{t'_i},\varphi_{t_j}\rangle_{\bbH}$.
\end{proposition}
\begin{proof}
We denote $\nu'=\nu(\weights',\pos')$ and apply Equation \eqref{eq:diff_J_nonnegative_app} with $\nu'=\nu+\sigma$ with $\sigma=\nu'-\nu$, we obtain that:
\[
F(\weights',\pos') - F(\weights,\pos)=J(\nu+\sigma)-J(\nu)=\int_{\cX} J'_{\nu} \text{d}\sigma + \frac{1}{2} \| \Phi(\sigma)\|^2_\bbH  =
\int_{\cX} J'_{\nu} \text{d}\sigma + \frac{1}{2} \| \Phi(\nu')-\Phi(\nu)\|^2_\bbH 
\] Note that the last term rewrites 
\[
    (\bm  \omega',-\bm  \omega)^\top\Kernel_{(\pos',\pos)}(\bm  \omega',-\bm  \omega)=\|\Phi(\nu')-\Phi(\nu)\|_\bbH^2\,,
\]
which is the squared Maximum Mean Discrepancy (MMD) between $\nu$ and $\nu'$ for the kernel $K$.
\end{proof}

\section{Technical results for Theorem \ref{theo:global_sto} \label{app:tech_tools}}

\begin{proposition}
\label{prop:bornes_techniques_J}
 Consider $(\vt,\tilde{\vt}) \in \cX^2$, the following technical inequalities hold:
\begin{enumerate}
    \item[$i)$] $ |\nabla J'_{\nu}(\vt)- \nabla J'_{\nu}(\tilde{\vt})| \leq \left[ Lip(\nabla \varphi) \|\nu\|_{\mathrm{TV}} + Lip(\bm \nabla y)\right] |\vt-\tilde{\vt}|$,
    \item[$ii)$] $\| \nabla J_\nu'(t) \|_\mathcal{X} \leq (\|\nu\|_{\mathrm{TV}} \|\varphi\|_{\infty,\mathbb{H}}+\|y\|_\mathbb{H}) \|\varphi'\|_\infty$ with $\|\varphi'\|_{\infty} := \sup_t\sup_{\psi:\|\psi\|_\mathbb{H}\leq 1} \| \nabla_{\bm t} \langle \varphi_{\bm t},\psi\rangle\|_\mathcal{X}$,
    \item[$iii)$] $\| D_\nu (t,Z)\|_\mathcal{X} \leq \|\nu\|_{\mathrm{TV}} \|\bm g'\|_{\infty,\mathbb{H}}+ \|\bm h'\|_{\mathbb{H}},$
    \item[$iv)$] $\| \nabla^2 J_\nu'(\bm t) \|_{op} \leq (\|\nu\|_{\mathrm{TV}} \|\varphi\|_{\infty,\mathbb{H}}+\|y\|_\mathbb{H}) \| \nabla^2\varphi\|_{\infty,op}$ with $\|\nabla^2\varphi\|_{\infty,op} := \sup_t\sup_{\psi:\|\psi\|_\mathbb{H}\leq 1} \| \nabla^2_{\bm t} \langle \varphi_{\bm t},\psi\rangle\|_{op}$,
\end{enumerate}
where Assumption \eqref{A2} is required for inequality $iii)$.
\end{proposition}

\begin{proof}[Proof of Proposition \ref{prop:bornes_techniques_J}] We consider any finite measure $\nu$.

\noindent 
    \underline{Proof of  $i)$} 
    We provide an upper bound on the Lipschitz constant of $\nabla J'_{\nu}$: consider $(\vt,\tilde{\vt}) \in \cX^2$, we repeat the same arguments as above and observe that:
    \begin{align*}
         |\nabla J'_{\nu}(\vt)- \nabla J'_{\nu}(\tilde{\vt})| \leq \left[ Lip(\nabla \varphi) \|\nu\|_{\mathrm{TV}} + Lip(\bm \nabla y)\right] |\vt-\tilde{\vt}|.
    \end{align*}
    \underline{Proof of $ii)$} Using a rough bound, we get
    \begin{eqnarray*}
    \| \nabla_{\bm t}J_\nu'(t)\|
    & = & \left\| \sum_{j=1}^p \omega_j \nabla_{\bm t} \langle \varphi_{\bm t},\varphi_{{\bm t}_j} \rangle_\mathbb{H} - \nabla_{\vt} \langle \varphi_\vt,y\rangle_\mathbb{H} \right\|, \\
    & \leq & \sum_{j=1}^p \omega_j \| \nabla_\vt \langle \varphi_\vt,\varphi_{\vt_j} \rangle_\mathbb{H}\|_\mathcal{X} + \| \nabla_\vt \langle \varphi_\vt,y\rangle_\mathbb{H} \|_\mathcal{X}, \\
    & \leq & \left( \|\nu\|_{\mathrm{TV}} \|\varphi\|_{\infty,\mathbb{H}} + \|y\|_\mathbb{H}\right) \times \sup_{\psi:\|\psi\|_\mathbb{H}\leq 1} \| \nabla_\vt \langle \varphi_\vt,\psi\rangle\|_\mathcal{X},
    \end{eqnarray*}
    which provides the desired result. \\
    \underline{Proof of $iii)$} Using Assumption \eqref{A2}, and in particular the boundedness of the derivative of $\bm g$ and $\bm h$, we get
    \begin{eqnarray*}
    \| D_\nu(\bm t,Z)\|_\mathcal{X}
    & = & \left\| \|\nu\|_{\mathrm{TV}} \nabla_\vt \bm g_{\vt,T}(U) - \nabla_\vt \bm h_\vt(V) \right\|_\mathcal{X}, \\
    & \leq & \|\nu\|_{\mathrm{TV}} \sup_{t,s,u} \| \nabla_t \bm g_{t,s}(u)\|_{\mathcal{X}} + \sup_{t,v} \|\nabla_t \bm h_t(v) \|_\mathcal{X}.
    \end{eqnarray*}
\underline{Proof of $iv)$} Similarly to item $iii)$, we have
\begin{eqnarray*}
    \| \nabla^2J_\nu'(\bm t)\|_{op}
    & = & \left\| \sum_{j=1}^p \omega_j \nabla^2 \langle \varphi_{\bm t},\varphi_{{\bm t}_j} \rangle_\mathbb{H} - \nabla^2 \langle \varphi_\vt,y\rangle_\mathbb{H} \right\|, \\
    & \leq & \sum_{j=1}^p \omega_j \| \nabla^2 \langle \varphi_\vt,\varphi_{\vt_j} \rangle_\mathbb{H}\|_{op} + \| \nabla^2 \langle \varphi_\vt,y\rangle_\mathbb{H} \|_{op}, \\
    & \leq & \left( \|\nu\|_{\mathrm{TV}} \|\varphi\|_{\infty,\mathbb{H}} + \|y\|_\mathbb{H}\right) \times \sup_{\psi:\|\psi\|_\mathbb{H}\leq 1} \| \nabla^2 \langle \varphi_\vt,\psi\rangle\|_{op},
    \end{eqnarray*}
\end{proof}

\begin{proposition}\label{prop:hoeffding}
    A large enough constant $\mathfrak{C}$ exists such that for any iteration $k\in \mathbb{N}$ and any particle $j\in \lbrace 1,\dots, p \rbrace$, :
    $$
\left|      \mathbb{E}  \left[  \alpha  J'_{\nuk}(\vt_j^k) + e^{-\alpha \widehat{\mJ'_{\nu_k}}(\vt_j^k)} -  1 \big\vert \mathfrak{F}_k \right] \right|  \leq \mathfrak{C} \alpha^2  (1+\|\nuk\|_{\mathrm{TV}})^2.
    $$
\end{proposition}

\begin{proof} This key technical argument relies on the Hoeffding inequality. We shall write:
    \begin{align*}
    \mathbb{E} \left[   \alpha  J'_{\nuk}(\vt_j^k) + e^{-\alpha \widehat{\mJ'_{\nu_k}}(\vt_j^k)} -  1   \big\vert \mathfrak{F}_k\right] & = 
    \alpha  J'_{\nuk}(\vt_j^k) - 1 + \mathbb{E}\left[e^{-\alpha \widehat{\mJ'_{\nu_k}}(\vt_j^k)}\big\vert \mathfrak{F}_k \right]  \\
    & = 
     \left[ \alpha  J'_{\nuk}(\vt_j^k) - 1 + e^{-\alpha  J'_{\nuk}(\vt_j^k)} \mathbb{E}\left[e^{-\alpha [\widehat{\mJ'_{\nu_k}}(\vt_j^k)- J'_{\nuk}(\vt_j^k)]}\big\vert \mathfrak{F}_k \right]\right] \\
    & =   \left[ \alpha  J'_{\nuk}(\vt_j^k) - 1 + e^{-\alpha  J'_{\nuk}(\vt_j^k)}\right]\\
    & +   e^{-\alpha  J'_{\nuk}(\vt_j^k)}
     \mathbb{E}\left[e^{-\alpha [\widehat{\mJ'_{\nu_k}}(\vt_j^k)- J'_{\nuk}(\vt_j^k)]} - 1\big\vert \mathfrak{F}_k \right]
\end{align*}
To derive an upper bound, we apply the Hoeffding Lemma to the random variable $J'_{\nuk}(\vt_j^k,Z^{k+1})- J'_{\nuk}(\vt_j^k)$ that is centered and bounded by $\bm T = \mathfrak{C} (1+\|\nuk\|_{\mathrm{TV}})$ from according to Lemma \ref{lem:boundJprime}. We obtain that:
\[
\left|
\mathbb{E}\left[e^{-\alpha [\widehat{\mJ'_{\nu_k}}(\vt_j^k))- J'_{\nuk}(\vt_j^k)]} - 1\big\vert \mathfrak{F}_k \right] \right|\leq e^{\frac{T^2 \alpha^2}{8}} - 1 \leq \mathfrak{C} \alpha^2 (1+\|\nuk\|_{\mathrm{TV}}^2).
\]
Using that $|e^{x}-1-x| \leq c |x|^2$ for bounded $x = \alpha J'_{\nuk}(\vt_j^k)$ and $c$ large enough, we finally obtain that:
\[\left| \mathbb{E}  \left[ \alpha  J'_{\nuk}(\vt_j^k) + e^{-\alpha \widehat{\mJ'_{\nu_k}}(\vt_j^k)} -  1   \big\vert \mathfrak{F}_k\right] \right| \leq \mathfrak{C} \alpha^2  (1+\|\nuk\|_{\mathrm{TV}}^2).\]

\end{proof}

We recall here the result essentially due to \cite{chizat2022sparse}, which is stated in a simplest way for our purpose.
\begin{proposition}
    \label{prop:approximation} Assume that $\mus$ is discrete and that $\nu_0$ is a uniform distribution over a grid of size $\delta= \frac{2 d}{\tau}$ where $d$ is the dimension of $\cX$, then:
    \[
    Q_{\tau}(\mus,\nu_0) \leq  
    \frac{\|\mus\|_{\mathrm{TV}} d}{\tau} \left(1+ \log\frac{\tau}{2d}  + \frac{\log |\cX|}{d}\right)     
    \]
    Moreover, the measure $\nu_0^\delta$ that meets this upper bound satisfies $\|\nu_0^{\delta}\|_{\mathrm{TV}}=\|\mus\|_{\mathrm{TV}}$.    
\end{proposition}
\begin{proof}
    We define $m$ as the size of the support of $\nu_0$ and \[\nu_0 = m^{-1} \sum_{i=1}^m \delta_{x_i},\] where $(x_i)_{1 \leq i \leq m}$ refers to the uniform grid of size $\delta$ on $\cX$.
    Since $\mus$ is  discrete, it may be written as: \[\mus=\sum_{j=1}^{m^\star} \mus_j \delta_{z_j^\star}.\]

\noindent For any support point $z_j^\star$ of $\mus$, we then consider $i_j \in \lbrace 1,ots, m\rbrace$ such that $\|x_{i_j}-z_j\| \leq \delta/2$ and we define $\nu_0^\delta$ as:
\[
\nu_0^\delta := \sum_{j=1}^{m^\star} \mus_j \delta_{x_{i_j}}
\]
We observe that by construction, $\|\nu_0^\delta\|_{\mathrm{TV}} = \|\mus\|_{\mathrm{TV}}$ and
\begin{align}
    \mathcal{H}(\nu_0,\nu_0^{\delta}) & = \sum_{j=1}^{m^\star} \nu_0^{\delta}(x_{i_j)} \log \left( \frac{ \nu_0^{\delta}(x_{i_j)}}{\nu_0(x_{i_j})}\right) \nonumber\\ 
    &=\sum_{j=1}^{m^\star} \mus_j  \log \left( \frac{ \mus_j}{\nu_0(x_{i_j})}\right)\nonumber\\
    &\leq - \mathrm{Ent}(\mus) + \|\mus\|_{\mathrm{TV}} \left[ d \log \left(\frac{1}{\delta}\right) + \log |\cX|\right], \label{eq:up_H}
\end{align}
where we used   the entropy of a discrete measure defined as 
$$\mathrm{Ent}(\mu)= - \sum_{x \in Supp(\mu)} \mu(x) \log(\mu(x))$$ 
and a lower bound of $\nu_0(x_{i_j})$, which is of the order $\delta^{d} |\cX|^{-1}$ where $|\cX|$ refers to the Lebesgue measure of $\cX$.

In the meantime, we also observe that the BL dual norm between $\nu_0^{\delta}$ and $\mus$ can be easily upper bounded. Indeed
\begin{align}
    \|\nu_0^{\delta}-\mus\|_{BL}^* & = \sup_{\|f\|_{BL} \leq 1} \int_{\cX} f \text{d}[\nu_0^{\delta}-\mus]\nonumber\\
    & =\sup_{\|f\|_{BL} \leq 1} \sum_{j=1}^{m^\star} \mus_j [f(x_{i_j})-f(z_j)] \nonumber\\
    & \leq \frac{\|\mus\|_{\mathrm{TV}} \delta }{2}\label{eq:up_BL}
    \end{align}
We then add the two upper bounds \eqref{eq:up_H} and \eqref{eq:up_BL} and minimize
\[
\delta \longmapsto \frac{\|\mus\|_{\mathrm{TV}} \delta }{2} + \frac{1}{\tau} \left(- \mathrm{Ent}(\mus) + \|\mus\|_{\mathrm{TV}} \left[ d \log \left(\frac{1}{\delta}\right) + \log |\cX|\right]\right).\]
\noindent 
We are led to choose $\delta= 2 d \tau^{-1}$ and we obtain the following upper bound:
\[
\frac{1}{\tau}\mathcal{H}(\nu_0,\nu_0^{\delta}) + \|\nu_0^{\delta}-\mus\|_{BL}^* \leq  \frac{\|\mus\|_{\mathrm{TV}} d}{\tau} \left(1+ \log\frac{\tau}{2d}  + \frac{\log |\cX|}{d}\right),
\]
which ends the proof of the proposition.\end{proof}

\end{document}